\documentclass[11pt]{article}
\usepackage[american]{babel}
\usepackage{a4}
\usepackage{hyperref}
\usepackage{amsmath, amssymb, amsthm}
\usepackage{subfig}
\usepackage{ifthen}
\usepackage{tikz}
\usetikzlibrary{positioning,arrows,patterns}
\usepackage{dsfont}

\DeclareMathOperator{\re}{Re}
\DeclareMathOperator{\im}{Im}

\newcommand{\mC}{{\mathds C}}
\newcommand{\mR}{{\mathds R}}
\newcommand{\mZ}{{\mathds Z}}

\newtheorem{theorem}{Theorem}[subsection]
\newtheorem{proposition}[theorem]{Proposition}
\newtheorem{lemma}[theorem]{Lemma}
\newtheorem{corollary}[theorem]{Corollary}

\theoremstyle{definition}

\newtheorem*{definition}{Definition}

\newtheorem*{remark}{Remark}

\begin{document}

\title{Discrete complex analysis on planar quad-graphs}

\author{Alexander I.~Bobenko\footnote{Institut f\"ur Mathematik, MA 8-3, Technische
Universit\"at Berlin, Stra{\ss}e des 17. Juni 136, 10623 Berlin, Germany.}$\;^{,1}$ \and Felix G\"unther\footnotemark[1]$\;^{,2}$}

\date{}
\maketitle

\footnotetext[1]{Supported by the DFG Collaborative Research Center SFB/TRR 109 ``Discretization in Geometry and Dynamics''. E-mail: bobenko@math.tu-berlin.de}
\footnotetext[2]{Supported by the Deutsche Telekom Stiftung. E-mail: fguenth@math.tu-berlin.de}

\begin{abstract}
\noindent
We develop a linear theory of discrete complex analysis on general quad-graphs, continuing and extending previous work of Duffin, Mercat, Kenyon, Chelkak and Smirnov on discrete complex analysis on rhombic quad-graphs. Our approach based on the medial graph yields more instructive proofs of discrete analogs of several classical theorems and even new results. We provide discrete counterparts of fundamental concepts in complex analysis such as holomorphic functions, derivatives, the Laplacian, and exterior calculus. Also, we discuss discrete versions of important basic theorems such as Green's identities and Cauchy's integral formulae. For the first time, we discretize Green's first identity and Cauchy's integral formula for the derivative of a holomorphic function. In this paper, we focus on planar quad-graphs, but we would like to mention that many notions and theorems can be adapted to discrete Riemann surfaces in a straightforward way.

\noindent
In the case of planar parallelogram-graphs with bounded interior angles and bounded ratio of side lengths, we construct a discrete Green's function and discrete Cauchy's kernels with asymptotics comparable to the smooth case. Further restricting to the integer lattice of a two-dimensional skew coordinate system yields appropriate discrete Cauchy's integral formulae for higher order derivatives.\\ \vspace{0.5ex}

\vspace{0.5ex}

\end{abstract}

\raggedbottom
\setlength{\parindent}{0pt}
\setlength{\parskip}{1ex}


\section{Introduction}\label{sec:intro}

We start with a short history of discrete complex analysis before sketching the organization of our paper. Our aim is to give just a very rough overview of the linear discrete theories of complex analysis we are enhancing and to name connections to statistical physics and the nonlinear discrete theory of complex analysis based on circle patterns. For a more detailed discussion of these topics, we refer to the survey of Smirnov \cite{Sm10S}, on which our introduction is based and that also discusses applications to probability theory and mathematical physics, and to the book \cite{BoSu08} of the first author and Suris that not only investigates the influence of discrete integrable systems to discrete theories of complex analysis, but also gives a good overview of discrete differential geometry in general.

Linear theories of discrete complex analysis look back on a long and varied history. Already Kirchhoff's circuit laws describe a discrete harmonicity condition for the potential function whose gradient describes the current flowing through the electric network. A notable application of Kirchhoff's laws in geometry was the article \cite{BSST40} of Brooks, Smith, Stone, and Tutte, who used coupled discrete harmonic functions (in fact, discrete holomorphic functions) to construct tilings of rectangles into squares with different integral side lengths.

Discrete harmonic functions on the square lattice were studied by a number of authors in the 1920s, including Courant, Friedrichs, and Lewy, who showed convergence of solutions of the Dirichlet boundary value problem to their corresponding continuous counterpart \cite{CoFrLe28}. Recently, Skopenkov studied Dirichlet boundary value problems of discrete analytic functions on general quad-graphs in \cite{Sk13}, and he showed a convergence result in the case of quad-graphs where the diagonals of quadrilaterals intersect orthogonally.

Discrete holomorphic functions on the square lattice were studied by Isaacs \cite{Is41}. He proposed two different definitions for holomorphicity. The first one said that $f$ is \textit{monodiffric of the first kind} if\[f(z+i\varepsilon)-f(z)=i(f(z+\varepsilon)-f(z)),\] where $\varepsilon$ denotes the side length of the squares. This definition is not symmetric on the square lattice, but it becomes symmetric on the triangular lattice obtained by inserting the diagonals parallel to $1-i$. Dynnikov and Novikov studied this notion in \cite{DN03}.

Isaac's second definition was given by \[f(z+i\varepsilon)-f(z+\varepsilon)=i(f(z+(1+i)\varepsilon)-f(z)).\] In this context, it is natural to consider the real part of a discrete holomorphic function as being defined on one type of vertices, say black, and the imaginary part on the other type of vertices, say white, corresponding to a bipartite decomposition of the square lattice. Lelong-Ferrand reintroduced this notion in \cite{Fe44} and developed the theory to a level that allowed her to prove the Riemann mapping theorem using discrete methods \cite{Fe55}. Duffin also studied discrete complex analysis on the square grid \cite{Du56} and was the first who extended the theory to rhombic lattices \cite{Du68}. Mercat \cite{Me01}, Kenyon \cite{Ke02}, and Chelkak and Smirnov \cite{ChSm11} resumed the investigation of discrete complex analysis on rhombic lattices or, equivalently, isoradial graphs. In these settings, it was natural to split the real and the imaginary part of a discrete holomorphic function to the two vertex sets of a bipartite decomposition.

Some two-dimensional discrete models in statistical physics exhibit conformally invariant properties in the thermodynamical limit. Such conformally invariant properties were established by Smirnov for site percolation on a triangular grid \cite{Sm01} and for the random cluster model \cite{Sm10}, by Chelkak and Smirnov for the Ising model \cite{ChSm12}, and by Kenyon for the dimer model on a square grid (domino tiling) \cite{Ke00}. In all cases, linear theories of discrete analytic functions on regular grids were highly important. Kenyon as well as Chelkak and Smirnov obtained important analytic results \cite{Ke02, ChSm11} that were instrumental in the proof that the critical Ising model is universal, i.e., that the scaling limit is independent of the shape of the lattice \cite{ChSm12}. Already Mercat related the theory of discrete complex analysis to the Ising model and investigated criticality \cite{Me01}.

Important non-linear discrete theories of complex analysis involve circle packings or, more generally, circle patterns. Rodin and Sullivan first proved that the Riemann mapping of a complex domain to the unit disk can be approximated by circle packings \cite{RSul87}. A similar result for isoradial circle patterns, even with irregular combinatorics, is due to B\"ucking \cite{Bue08}. In their paper \cite{BoMeSu05}, the first author, Mercat, and Suris showed how the linear theory of discrete holomorphic functions on quad-graphs can be obtained by linearizing the theory on circle patterns: Discrete holomorphic functions describe infinitesimal deformations of circle patterns. In the case of parallelogram-graphs, they embedded the quad-graph in $\mathds{Z}^n$ and introduced a discrete exponential function and a discrete logarithm, generalizing Kenyon's discrete exponential and discrete Green's function \cite{Ke02}.

Our setup in Section~\ref{sec:general} is a strongly regular cellular decomposition of the complex plane into quadrilaterals, called quad-graph, that we assume to be bipartite. Of crucial importance for our work is the medial graph of a quad-graph. It provides the connection between the notions of discrete derivatives of Kenyon \cite{Ke02}, Mercat \cite{Me07}, and Chelkak and Smirnov \cite{ChSm11}, extended from rhombic to general quad-graphs, and discrete differential forms and discrete exterior calculus as suggested by Mercat \cite{Me01,Me08}. Concerning discrete differential forms, we get essentially the same definitions as Mercat proposed in \cite{Me08}. However, our notation of discrete exterior calculus is slightly more general and shows its power when considering integral formulae. Discrete Stokes' Theorem~\ref{th:stokes} will be a consequence and not part of the definition of the discrete exterior derivate as in the work of Mercat \cite{Me01,Me08}, and in Theorem~\ref{th:derivation} we prove that the discrete exterior derivative is a derivation of the discrete wedge-product. These two theorems are the most powerful ones in our setting. In particular, the proof of discrete Green's identities in Theorem~\ref{th:Green_identities} is an immediate corollary. Here, a discrete version of Green's first identity is provided for the first time.

According to the paper \cite{ChSm11} of Chelkak and Smirnov, one of the unpleasant facts of all discrete theories of complex analysis is that (pointwise) multiplication of discrete holomorphic functions does not yield a discrete holomorphic function in general. We suggest at least a partial solution in Corollary~\ref{cor:discrete_product}, where we describe how the medial graph allows to (kind of pointwise) multiply discrete holomorphic functions to a function that is not defined on the vertex sets of the original graphs anymore, but that is discrete holomorphic in a certain sense.

Based on Skopenkov's results on the existence and uniqueness of solutions to the discrete Dirichlet boundary value problem \cite{Sk13}, we prove surjectivity of the discrete differentials and the discrete Laplacian seen as linear operators in Theorem~\ref{th:plane_surjective}. In particular, discrete Green's functions and discrete Cauchy's kernels $z^{-1}$ exist. As a consequence, we can formulate discrete Cauchy's integral formulae for discrete holomorphic functions in Theorem~\ref{th:Cauchy_formula} and for the discrete derivative of a discrete holomorphic function on the vertices of the quad-graph in Theorem~\ref{th:Cauchy_formula_derivative}. Unfortunately, we cannot require any certain asymptotic behavior of them in the general setup so far. But at least we show in Theorem~\ref{th:harmonic_asymptotics} that any discrete harmonic function with asymptotics $o(v^{-1/2})$ is constant, provided that all interior angles and side lengths of the quadrilaterals are bounded.

Section~\ref{sec:parallel} is devoted to discrete complex analysis on planar parallelogram-graphs. There, we construct discrete Green's functions and discrete Cauchy's kernels with asymptotics similar to the functions in the rhombic case \cite{Ke02,Bue08,ChSm11} and close to the smooth case. This will be proven in Theorems~\ref{th:Green_asymptotics}, \ref{th:Cauchy_asymptotics_1}, and~\ref{th:Cauchy_asymptotics_2}, assuming that the interior angles and the ratio of side lengths of all parallelograms are bounded. The construction of these functions is based on the discrete exponential introduced by Kenyon on quasicrystallic rhombic quad-graphs \cite{Ke02} and by the first author, Mercat, and Suris on quasicrystallic parallelogram-graphs \cite{BoMeSu05}. In the end, we close with the very special case of the integer lattice of a skew coordinate system in the complex plane. In this case, we show in Theorem~\ref{th:Cauchy_formula_n_derivative} that discrete Cauchy's integral formulae for higher order discrete derivatives of a discrete holomorphic function exist and that the asymptotics of these formulae match the expectations from the previous results.

Finally, in the appendix we give the proofs of some statements concerning the combinatorics of parallelogram-graphs that we use for determining the asymptotic behavior of certain discrete functions.


\section{Discrete complex analysis on planar quad-graphs}\label{sec:general}

Although we focus on planar quad-graphs in this paper, many of our notions and theorems generalize to discrete Riemann surfaces. A corresponding linear theory of discrete Riemann surfaces will be discussed in a subsequent paper and can also be found in the PhD thesis of the second author \cite{Gue14}.


\subsection{Basic definitions and notation} \label{sec:basics}

The aim of this section is to introduce first bipartite quad-graphs and some basic notation in Section~\ref{sec:quadgraph} and then to discuss the medial graph in Section~\ref{sec:medial}. 


\subsubsection{Bipartite quad-graphs} \label{sec:quadgraph}

We consider a strongly regular cellular decomposition of the complex plane $\mC$ described by an embedded bipartite \textit{quad-graph} $\Lambda$ such that $0$-cells correspond to vertices $V(\Lambda)$, $1$-cells to edges $E(\Lambda)$, and $2$-cells to quadrilateral faces $F(\Lambda)$. We refer to the maximal independent sets of vertices of $\Lambda$ as \textit{black} and \textit{white} vertices. Furthermore, we restrict to locally finite cellular decompositions, i.e., a compact subset of $\mC$ contains only finitely many quadrilaterals.

The assumption of strong regularity asserts that the boundary of a quadrilateral contains a particular vertex or a particular edge at most once and that two different edges or faces have at most one vertex or edge in common, respectively. As a consequence, if a line segment connecting two vertices of $\Lambda$ is the (possibly outer) diagonal of a quadrilateral, there is just one such face of $\Lambda$. Let $\Gamma$ and $\Gamma^*$ be the graphs defined on the black and white vertices where the edges are exactly the corresponding diagonals of faces of $\Lambda$. If the diagonal lies outside the face, it is more convenient to consider the corresponding edge of $\Gamma$ or $\Gamma^*$ not as the straight line segment connecting its two endpoints, but as a curve lying inside the face. Then, the \textit{duality} between $\Gamma$ and $\Gamma^*$ becomes obvious: Black and white vertices are in one-to-one correspondence to the white and black faces they are contained in, and black and white edges dual to each other are exactly these who cross each other. For simplicity, we identify vertices of $\Lambda$ or of $\Gamma$ and $\Gamma^*$ with their corresponding complex values, and to oriented edges of $\Lambda,\Gamma,\Gamma^*$ we assign the complex numbers determined by the difference of their two endpoints.

To $\Lambda$ we associate its \textit{dual} $\Diamond:=\Lambda^*$. In general, we look at $\Diamond$ in an abstract way, identifying vertices or faces of $\Diamond$ with corresponding faces or vertices of $\Lambda$, respectively. However, in the particular case that all quadrilaterals are parallelograms, it makes sense to place the vertices of $\Diamond$ at the centers of the parallelograms. Here, the center of a parallelogram is the point of intersection of its two diagonals. Further details will be given in Sections~\ref{sec:derivative_lambda} and~\ref{sec:derivative_diamond}.

If a vertex $v \in V(\Lambda)$ is a vertex of a quadrilateral $Q\in F(\Lambda)\cong V(\Diamond)$, we write $Q \sim v$ or $v \sim Q$ and say that $v$ and $Q$ are \textit{incident} to each other. The vertices of $Q$ are denoted by $b_-,w_-,b_+,w_+$ in counterclockwise order, where $b_\pm \in V(\Gamma)$ and $w_\pm \in V(\Gamma^*)$.

\begin{definition}
For a quadrilateral $Q \in V(\Diamond)$ we define \[\rho_Q:= -i \frac{w_+-w_-}{b_+-b_-}.\] Moreover, let \[\varphi_Q:=\arccos\left(\re\left(i\frac{\rho_Q}{|\rho_Q|}\right)\right)=\arccos\left(\re\left(\frac{ (b_+ - b_-)\overline{(w_+ - w_-)}}{|b_+ - b_-||w_+ - w_-|}\right)\right)\] be the angle under which the diagonal lines of $Q$ intersect.
\end{definition}

Note that $0<\varphi_Q<\pi$. Figure~\ref{fig:quadgraph} shows a finite bipartite quad-graph together with the notations we have introduced for a single quadrilateral $Q$ and the notations we are using later for the \textit{star of a vertex} $v$, i.e., the set of all faces incident to $v$.

\begin{figure}[htbp]
\begin{center}
\beginpgfgraphicnamed{quad}
\begin{tikzpicture}
[white/.style={circle,draw=black,fill=white,thin,inner sep=0pt,minimum size=1.2mm},
black/.style={circle,draw=black,fill=black,thin,inner sep=0pt,minimum size=1.2mm},
gray/.style={circle,draw=black,fill=gray,thin,inner sep=0pt,minimum size=1.2mm}]
\node[white] (w1) [label=left:$v'_{s-1}$]
at (-1,-1) {};
\node[white] (w2) [label=below:$v'_s$]
 at (1,-1) {};
\node[white] (w3) [label=below:$v'_k$]
 at (1,0) {};
\node[white] (w4) [label=left:$v'_1$]
 at (1,1) {};
\node[white] (w5) [label=left:$v'_2$]
 at (-1,1) {};
\node[white] (w6) [label=above:$w_+$]
 at (-4,0) {};
\node[white] (w7)
 at (1,-2) {};
\node[white] (w8)
 at (3,0) {};
\node[white] (w9)
 at (1,2) {};
\node[white] (w10) [label=below:$w_-$]
 at (-6,-2) {};
\node[white] (w11)
 at (-2,-2) {};
\node[black] (b1) [label=above:$v$]
 at (0,0) {};
\node[black] (b2) [label=right:$v_1$]
 at (2,1) {};
\node[black] (b3)
 at (2,-1) {};
\node[black] (b4) [label=below:$v_s$]
 at (0,-2) {};
\node[black] (b5) [label=below:$b_+$]
 at (-4,-2) {};
\node[black] (b6) [label=above:$v_2$]
 at (0,2) {};
\node[black] (b7) [label=above :$b_-$]
 at (-6,0) {};
\node[black] (b8)
 at (-2,0) {};
\draw[color=white] (w1) --node[midway,color=black] {$Q_s$} (w2);
\draw (b2) -- (w9) -- (b6) -- (w5) -- (b8);
\draw (b4) -- (w7) -- (b3) -- (w8) -- (b2);
\draw (b8) -- (w1) --  (b4) --  (w2) -- (b3) -- (w3) -- (b2) -- (w4) -- (b6);
\draw (b1) -- (w1);
\draw (b1) -- (w2);
\draw (b1) -- (w3);
\draw (b1) -- (w4);
\draw (b1) -- (w5);
\draw (b8) -- (w11) -- (b4);
\draw (w6) -- (b5) -- (w10) -- (b7) -- (w6);
\draw (b5) -- (w11);
\draw (w6) -- (b8);
\draw[dashed] (b5) -- (b7);
\draw[dashed] (w10) -- (w6);
\node[gray] (z) [label=below:$Q$]
at (-5,-1) {};
\draw (-4.45,-1.55) arc (-45:43:0.8cm);
\coordinate[label=right:$\varphi_Q$] (phi) at (-4.9,-1);
\end{tikzpicture}
\endpgfgraphicnamed
\caption{Bipartite quad-graph with notations}
\label{fig:quadgraph}
\end{center}
\end{figure}
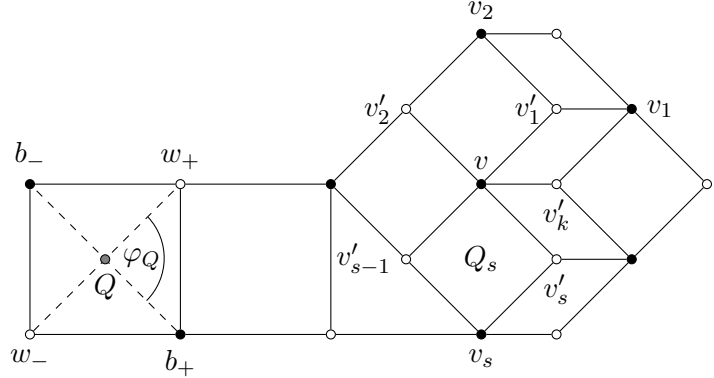

In addition, we denote by $\Diamond_0$ always a connected subset of $\Diamond$. It is said to be \textit{simply-connected} if the corresponding set of cells in $\mC$ is simply-connected. Its vertices induce a subgraph $\Lambda_0$ of $\Lambda$ together with subgraphs $\Gamma_0$ of $\Gamma$ and $\Gamma_0^*$ of $\Gamma^*$. For simplicity, we assume that the induced subgraphs are connected as well. By $\partial \Lambda_0$ we denote the subgraph of $\Lambda_0$ that consists of boundary vertices and edges.


\subsubsection{Medial graph} \label{sec:medial}

\begin{definition}
The \textit{medial graph} $X$ of $\Lambda$ is defined as follows. Its vertex set is given by all the midpoints of edges of $\Lambda$, and two vertices are adjacent if and only if the corresponding edges belong to the same face and have a vertex in common. The set of faces of $X$ is in bijection with $V(\Lambda)\cup V(\Diamond)$: A face $F_v$ corresponding to $v\in V (\Lambda)$ consists of the midpoints of edges of $\Lambda$ incident to $v$, and a face $F_Q$ corresponding to $Q\in V(\Diamond)$ consists of the midpoints of the four edges of $\Lambda$ belonging to $Q$.

Any edge $e$ of $X$ is the common edge of two faces $F_Q$ and $F_v$ for $Q \sim v$, denoted by $[Q,v]$.
\end{definition}

Let $Q \in V(\Diamond)$ and $v_0 \sim Q$. Due to Varignon's theorem, $F_Q$ is a parallelogram and the complex number assigned to the edge $e=[Q,v_0]$ connecting the midpoints of edges $v_0v'_-$ and $v_0v'_+$ of $\Lambda$ is just half of $e=v'_+-v'_-$. Note that if $Q$ is nonconvex, some part of $F_Q$ lies outside $Q$ and it may happen that $v$ lies not inside $F_v$. These situations do not occur if all quadrilaterals are convex.

In Figure~\ref{fig:medial}, showing $\Lambda$ with its medial graph, the vertices of $F_Q$ and $F_v$ are colored gray.

\begin{figure}[htbp]
\begin{center}
\beginpgfgraphicnamed{medial}
\begin{tikzpicture}
[white/.style={circle,draw=black,fill=white,thin,inner sep=0pt,minimum size=1.2mm},
black/.style={circle,draw=black,fill=black,thin,inner sep=0pt,minimum size=1.2mm},
gray/.style={circle,draw=black,fill=gray,thin,inner sep=0pt,minimum size=1.2mm}]
\node[white] (w1)
at (-1,-1) {};
\node[white] (w2)
 at (1,-1) {};
\node[white] (w3)
 at (1,0) {};
\node[white] (w4)
 at (1,1) {};
\node[white] (w5)
 at (-1,1) {};
\node[white] (w6)
 at (-4,0) {};
\node[white] (w7)
 at (1,-2) {};
\node[white] (w8)
 at (3,0) {};
\node[white] (w9)
 at (1,2) {};
\node[white] (w10)
 at (-6,-2) {};
\node[white] (w11)
 at (-2,-2) {};
\node[black] (b1) [label=above:$v$]
 at (0,0) {};
\node[black] (b2)
 at (2,1) {};
\node[black] (b3)
 at (2,-1) {};
\node[black] (b4)
 at (0,-2) {};
\node[black] (b5)
 at (-4,-2) {};
\node[black] (b6)
 at (0,2) {};
\node[black] (b7)
 at (-6,0) {};
\node[black] (b8)
 at (-2,0) {};
\draw[color=white] (w1) -- (w2) --  (w3) --  (w4) --  (w5) --  (w1);
\draw[dashed] (b2) -- (w9) -- (b6) -- (w5) -- (b8);
\draw[dashed] (b4) -- (w7) -- (b3) -- (w8) -- (b2);
\draw[dashed] (b8) -- (w1) --  (b4) --  (w2) -- (b3) -- (w3) -- (b2) -- (w4) -- (b6);
\draw[dashed] (b1) -- (w1);
\draw[dashed] (b1) -- (w2);
\draw[dashed] (b1) -- (w3);
\draw[dashed] (b1) -- (w4);
\draw[dashed] (b1) -- (w5);
\draw[dashed] (b8) -- (w11) -- (b4);
\draw[dashed] (w6) -- (b5) -- (w10) -- (b7) -- (w6);
\draw[dashed] (b5) -- (w11);
\draw[dashed] (w6) -- (b8);

\coordinate[label=center:$Q$] (z)  at (-5,-1) {};

\coordinate (m4) at (-3,0);
\coordinate (m5) at (-2,-1);
\coordinate (m6) at (-3,-2);
\coordinate (m8) at (-1.5,-0.5);
\coordinate (m9) at (-1.5,0.5);
\coordinate (m13) at (0.5,-1.5);
\coordinate (m14) at (1.5,-1);
\coordinate (m15) at (1.5,-1.5);
\coordinate (m16) at (0.5,-2);
\coordinate (m17) at (-0.5,-1.5);
\coordinate (m18) at (-1,-2);
\coordinate (m19) at (-0.5,1.5);
\coordinate (m20) at (0.5,1.5);
\coordinate (m21) at (0.5,2);
\coordinate (m22) at (1.5,1.5);
\coordinate (m23) at (1.5,1);
\coordinate (m25) at (1.5,0.5);
\coordinate (m26) at (2.5,0.5);
\coordinate (m27) at (2.5,-0.5);
\coordinate (m28) at (1.5,-0.5);

\node[gray] (m1)
 at (-6,-1) {};
\node[gray] (m2)
 at (-5,0) {};
\node[gray] (m3)
 at (-4,-1) {};
\node[gray] (m7)
 at (-5,-2) {};
\node[gray] (m12)
 at (0.5,-0.5) {};
\node[gray] (m11)
 at (-0.5,-0.5) {};
\node[gray] (m10)
 at (-0.5,0.5) {};
\node[gray] (m24)
 at (0.5,0.5) {};
\node[gray] (m29)
 at (0.5,0) {};

\draw (m1) -- (m2) -- (m3) -- (m4) -- (m5) -- (m6) -- (m3) -- (m7) -- (m1);
\draw (m5) -- (m8) -- (m9) -- (m10) -- (m11) -- (m12) -- (m13) -- (m14) -- (m15) -- (m16) -- (m13) -- (m17) -- (m18) -- (m5);
\draw (m8) -- (m17) -- (m11) -- (m8);
\draw (m10) -- (m19) -- (m20) -- (m21) -- (m22) -- (m23) -- (m20) -- (m24) -- (m23) -- (m25) -- (m26) -- (m27) -- (m28) -- (m14) -- (m12) -- (m29) -- (m28) -- (m25) -- (m29) -- (m24) -- (m10);
\end{tikzpicture}
\endpgfgraphicnamed
\caption{Bipartite quad-graph (dashed) with medial graph (solid)}
\label{fig:medial}
\end{center}
\end{figure}
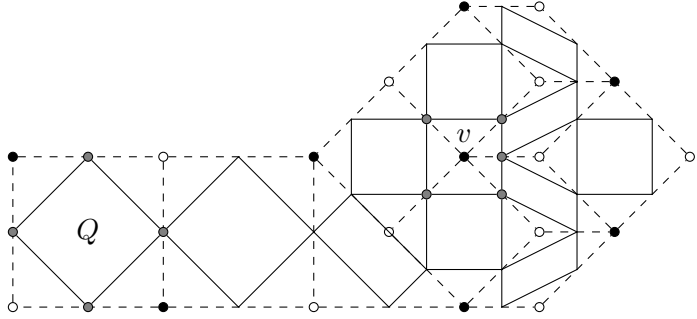

For a subgraph $\Diamond_0 \subseteq \Diamond$, we denote by $X_0 \subseteq X$ the subgraph of $X$ consisting of all edges $[Q,v]$ where $Q\in V(\Diamond_0)$ and $v\sim Q$. In the case of convex quadrilaterals, this means that we take all edges of $X$ lying inside $\Diamond_0$. The medial graph $X$ corresponds to a (strongly regular and locally finite) cellular decomposition of $\mC$ in a canonical way. In particular, we can talk about a topological disk in $F(X)$ and about a (counterclockwise oriented) boundary $\partial X_0$.

\begin{definition}
For $v\in V(\Lambda)$ and $Q \in V(\Diamond)$, let $P_v$ and $P_Q$ be the closed paths on $X$ connecting the midpoints of edges of $\Lambda$ incident to $v$ and $Q$, respectively, in counterclockwise direction. In Figure~\ref{fig:medial}, their vertices are colored gray. We say that $P_v$ and $P_Q$ are \textit{discrete elementary cycles}.
\end{definition}


\subsection{Discrete holomorphicity} \label{sec:holomorphicity}

In the classical theory, a real differentiable function $f:U\subseteq\mC \to \mC$ is holomorphic if and only if the Cauchy-Riemann equation $\partial_x f=-i\partial_y f$ is satisfied in all points of $U$. Moreover, holomorphic functions with nowhere-vanishing derivative preserve angles, and at a single point, lengths are uniformly scaled.

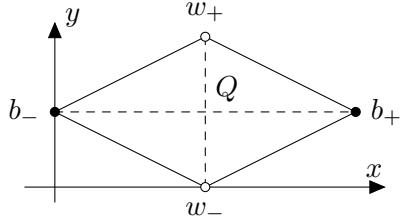
\begin{figure}[htbp]
\begin{center}
\beginpgfgraphicnamed{rhomb}
\begin{tikzpicture}[line cap=round,line join=round,>=triangle 45,x=2.0cm,y=2.0cm,
white/.style={circle,draw=black,fill=white,thin,inner sep=0pt,minimum size=1.2mm},
black/.style={circle,draw=black,fill=black,thin,inner sep=0pt,minimum size=1.2mm}]
  \draw[->,color=black] (-0.2,0) -- (2.2,0);
  \draw[->,color=black] (0,-0.1) -- (0,1.1);
  \draw[color=black] (0pt,-10pt);
  \draw [dashed] (1,1)-- (1,0);
  \draw [dashed] (0,0.5)-- (2,0.5);
  \draw (0,0.5)-- (1,1);
  \draw (1,1)-- (2,0.5);
  \draw (2,0.5)-- (1,0);
  \draw (1,0)-- (0,0.5);
  \node[white] (w1) [label=below:$w_-$] at (1,0) {};
  \node[white] (w2) [label=above:$w_+$] at (1,1) {};
  \node[black] (b1) [label=left:$b_-$] at (0,0.5) {};
  \node[black] (b2) [label=right:$b_+$] at (2,0.5) {};
  \coordinate[label=above right:$Q$] (Q) at (1,0.5);
  \coordinate[label=above right:$x$] (x) at (2,0);
  \coordinate[label=above right:$y$] (y) at (0,1);
\end{tikzpicture}
\endpgfgraphicnamed
\caption{Discretization of Cauchy-Riemann equation}
\label{fig:rhomb}
\end{center}
\end{figure}

Now, let us imagine a rhombus $Q$ in $\mC$ with vertices $b_-,w_-,b_+,w_+$ and diagonals $b_-b_+,w_-w_+$ aligned to the real $x$- and the imaginary $y$-axis, see Figure~\ref{fig:rhomb}. Then, \begin{align*} \frac{f(b_+)-f(b_-)}{b_+-b_-}\ &\textnormal{ discretizes}\qquad \partial_x f (Q),\\ \frac{f(w_+)-f(w_-)}{w_+-w_-} \ &\textnormal{ discretizes } -i \, \partial_y f (Q).\end{align*} This motivates the following definition of discrete holomorphicity due to Mercat \cite{Me08} that was also used previously in the rhombic setting by Duffin \cite{Du68} and others.

\begin{definition}
Let $Q \in V(\Diamond)$ and $f$ be a complex function on $b_-,w_-,b_+,w_+$. $f$ is said to be \textit{discrete holomorphic} at $Q$ if the \textit{discrete Cauchy-Riemann equation} is satisfied: \[\frac{f(b_+)-f(b_-)}{b_+ - b_-}=\frac{f(w_+)-f(w_-)}{w_+ - w_-}.\]
\end{definition}

If $Q$ is not a rhombus, we do not have an immediate interpretation of the discrete Cauchy-Riemann equation as before. But if a discrete holomorphic function $f$ does not have the same value on both black vertices $b_-$ and $b_+$, it preserves the angle $\varphi_Q$ and $f$ uniformly scales the lengths of the diagonals of $Q$. However, the image of $Q$ under $f$ might be a degenerate quadrilateral.

We immediately see that for discrete holomorphicity, only the differences at black and at white vertices matter. Hence, we should not consider constants on $V(\Lambda)$, but biconstants \cite{Me07} determined by each a value on $V(\Gamma)$ and $V(\Gamma^*)$. A function that is constant on $V(\Gamma)$ and constant on $V(\Gamma^*)$ is said to be \textit{biconstant}.

In the following, we will define discrete analogs of $\partial,\bar{\partial}$, first of functions on $V(\Lambda)$ in Section~\ref{sec:derivative_lambda} and later of functions on $V(\Diamond)$ in Section~\ref{sec:derivative_diamond}. Before, we introduce discrete differential forms in Section~\ref{sec:differential_forms}.


\subsubsection{Discrete derivatives of functions on \texorpdfstring{$V(\Lambda)$}{the vertices of the quad-graph}} \label{sec:derivative_lambda}

Remember that the derivatives $\partial:=\partial_z$ and $\bar{\partial}:=\partial_{\bar{z}}$ in complex analysis are defined through \[\partial=\left(\partial_x -i \partial_ y\right)/2 \textnormal{ and } \bar{\partial}=\left(\partial_x +i \partial_ y\right)/2;\] the coefficients in front of the partial derivatives in $\bar{\partial}$ are complex conjugate to the coefficients appearing in $\partial$. Using the interpretation \begin{align*}\frac{f(b_+)-f(b_-)}{b_+-b_-}&\cong \quad \partial_x f,\\ \frac{f(w_+)-f(w_-)}{w_+-w_-}&\cong -i \partial_y f\end{align*} for rhombi $Q$ as above, the definition of discrete derivatives in rhombic quad-graphs used by Chelkak and Smirnov in \cite{ChSm11}, \begin{align*}\partial_\Lambda f (Q)&=\frac{1}{2}\left(\frac{f(b_+)-f(b_-)}{b_+ - b_-}+\frac{f(w_+)-f(w_-)}{w_+ - w_-}\right),\\ \bar{\partial}_\Lambda f (Q)&=\frac{1}{2}\left(\frac{f(b_+)-f(b_-)}{\overline{b_+ - b_-}}+\frac{f(w_+)-f(w_-)}{\overline{w_+ - w_-}}\right),\end{align*} becomes plausible.

In a general quadrilateral $Q$, the diagonals usually do not intersect orthogonally and their quotient is not necessarily purely imaginary. Therefore, we have to choose different factors in front of the difference quotients, taking the deviation $\varphi_Q-\pi/2$ from orthogonality into account.

\begin{definition}
Let $Q \in V(\Diamond)$, and let $f$ be a complex function on $b_-,w_-,b_+,w_+$. The \textit{discrete derivatives} $\partial_\Lambda f$, $\bar{\partial}_\Lambda f$ are defined by
\begin{align*}
 \partial_\Lambda f (Q)&:=\lambda_Q \frac{f(b_+)-f(b_-)}{b_+ - b_-}+\bar{\lambda}_Q\frac{f(w_+)-f(w_-)}{w_+ - w_-},\\
 \bar{\partial}_\Lambda f (Q)&:=\bar{\lambda}_Q\frac{f(b_+)-f(b_-)}{\overline{b_+ - b_-}}+\lambda_Q\frac{f(w_+)-f(w_-)}{\overline{w_+ - w_-}},
\end{align*}
where $2\lambda_Q:=\exp\left(-i\left(\varphi_Q-\frac{\pi}{2}\right)\right)/\sin(\varphi_Q)$.
\end{definition}

Obviously, biconstant functions have vanishing discrete derivatives. If the quadrilateral $Q$ is a rhombus, $\varphi_Q=\pi/2$ and $\lambda_Q=1/2$. Thus, the definition above then reduces to the previous one in \cite{ChSm11}. The definition of discrete derivatives also matches the notion of discrete holomorphicity; and the discrete derivatives approximate their smooth counterparts correctly up to order one for general quad-graphs and up to order two for parallelogram-graphs:

\begin{proposition}\label{prop:examples}
 Let $Q \in V(\Diamond)$ and $f$ be a complex function on $b_-,w_-,b_+,w_+$.
 \begin{enumerate}
 \item $f$ is discrete holomorphic at $Q$ if and only if $\bar{\partial}_\Lambda f(Q)=0$.
 \item If $f(v)=v$, $\bar{\partial}_\Lambda f(Q)=0$ and $\partial_\Lambda f (Q)=1$.
 \item If $Q$ is a parallelogram and $f(v)=v^2$, $\bar{\partial}_\Lambda f(Q)=0$ and $\partial_\Lambda f(Q)=2Q$, where $Q$ as a vertex is placed at the center of the parallelogram.
 \item If $Q$ is a parallelogram and $f(v)=|v|^2$, $\bar{\partial}_\Lambda f(Q)=\overline{\partial_\Lambda f(Q)}=Q$, where $Q$ as a vertex is placed at the center of the parallelogram.
\end{enumerate}
\end{proposition}
\begin{proof}
(i) We observe that \begin{align*}\frac{2\sin(\varphi_Q)\bar{\lambda}_Q}{\overline{b_+-b_-}}&=\frac{\exp\left(i\left(\varphi_Q-\frac{\pi}{2}\right)\right)}{\overline{b_+-b_-}}=-i\frac{w_+-w_-}{|w_+-w_-| |b_+-b_-|},\\ \frac{2\sin(\varphi_Q)\lambda_Q}{\overline{w_+-w_-}}&=\frac{\exp\left(-i\left(\varphi_Q-\frac{\pi}{2}\right)\right)}{\overline{w_+-w_-}}=i\frac{b_+-b_-}{|w_+-w_-| |b_+-b_-|}.\end{align*}
So if we multiply $\bar{\partial}_\Lambda f (Q)$ by $2i|w_+-w_-| |b_+-b_-|\sin(\varphi_Q)\neq 0$, we obtain \[\left(w_+-w_-\right)\left(f\left(b_+\right)-f\left(b_-\right)\right)-\left(b_+-b_-\right)\left(f\left(w_+\right)-f\left(w_-\right)\right).\] The last expression vanishes if and only if the discrete Cauchy-Riemann equation is satisfied.

(ii) Clearly, $f(v)=v$ satisfies the discrete Cauchy-Riemann equation. By the first part, $\bar{\partial}_\Lambda f(Q)=0$. Due to $2\sin(\varphi_Q)=\exp\left(-i\left(\varphi_Q-\frac{\pi}{2}\right)\right)+\exp\left(i\left(\varphi_Q-\frac{\pi}{2}\right)\right)$, $\partial_\Lambda f (Q)$ simplifies to $\lambda_Q+\bar{\lambda}_Q=1$ .

(iii) For $f(v)=v^2$, the discrete Cauchy-Riemann equation is equivalent to $b_++b_-=w_++w_-$. But since $Q$ is a parallelogram, both $(b_++b_-)/2$ and $(w_++w_-)/2$ equal the center $Q$ of the parallelogram. Thus, $f$ is discrete holomorphic at $Q$ and \[\partial_\Lambda f (Q)=\lambda_Q (b_++b_-)+\bar{\lambda}_Q (w_++w_-)=2Q (\lambda_Q+\bar{\lambda}_Q)=2Q.\]

(iv) Since $f$ is a real function, $\bar{\partial}_\Lambda f(Q)=\overline{\partial_\Lambda f(Q)}$ follows straight from the definition. Let $z\in \mC$ be arbitrary. If $g(v):=v\bar{z}$, $\partial_\Lambda g(Q)=\bar{z}$ and $\partial_\Lambda \bar{g}(Q)=0$ by the second part. So if we define the function $h(v):=|v-z|^2=|v|^2-v\bar{z}-\bar{v}z+|z|^2$, then $\overline{\partial_\Lambda h(Q)}=\overline{\partial_\Lambda f(Q)}-z$. Hence, the statement is invariant under translation, and it suffices to consider the case when the center of the parallelogram $Q$ is the origin. Then, $b_+=-b_-$ and $w_+=-w_-$ since $Q$ is a parallelogram. It follows that $f(b_-)=f(b_+)$ and $f(w_-)=f(w_+)$, so $\partial_\Lambda f(Q)=0$ as desired.
\end{proof}

Our first discrete analogs of classical theorems are immediate consequences of the discrete Cauchy-Riemann equation:

\begin{proposition}\label{prop:zeroderivative}
Let $\Diamond_0 \subseteq \Diamond$ and $f:V(\Lambda)\to\mC$ be discrete holomorphic.
\begin{enumerate}
 \item If $f$ is purely imaginary or purely real, $f$ is biconstant.
 \item If $\partial_\Lambda f\equiv 0$, $f$ is biconstant.
\end{enumerate}
\end{proposition}
\begin{proof}
(i) Let $b_-$, $b_+$ be two adjacent vertices of $\Gamma$, and let $b_-,w_-,b_+,w_+ \in V(\Lambda_0)$ be the vertices of the corresponding quadrilateral in $\Diamond$. Let us assume that $f(b_+)\neq f(b_-)$. Due to the discrete Cauchy-Riemann equation, \[\frac{f(w_+)-f(w_-)}{f(b_+)-f(b_-)}=\frac{w_+ - w_-}{b_+ - b_-}.\] The left hand side is real, but the right hand side is not, contradiction. The same argumentation goes through if we start with two adjacent vertices of $\Gamma^*$.

(ii) Since $f$ is discrete holomorphic, \[\frac{f(b_+)-f(b_-)}{b_+ - b_-}=\frac{f(w_+)-f(w_-)}{w_+ - w_-}.\] $\partial_\Lambda f\equiv 0$ then yields that both sides of the discrete Cauchy-Riemann equation equal zero, so $f$ is constant on $V(\Gamma_0)$ and constant on $V(\Gamma_0^*)$.
\end{proof}


\subsubsection{Discrete differential forms} \label{sec:differential_forms}

We mainly consider two type of functions, functions $f:V(\Lambda)\to\mC$ and functions $h:V(\Diamond)\to\mC$. An example for a function on the quadrilateral faces is $\partial_\Lambda f$.

A \textit{discrete one-form} $\omega$ is a complex function on the oriented edges of the medial graph $X$, and a \textit{discrete two-form} $\Omega$ is a complex function on the faces of $X$. The evaluations of $\omega$ at an oriented edge $e$ of $X$ and of $\Omega$ at a face $F$ of $X$ are denoted by $\int_e \omega$ and $\iint_F \Omega$, respectively.

If $P$ is a directed path in $X$ consisting of oriented edges $e_1,e_2,\ldots,e_n$, the \textit{discrete integral} along $P$ is defined as $\int_P \omega=\sum_{k=1}^n \int_{e_k} \omega$. For closed paths $P$, we write $\oint_P \omega$ instead. In the case that $P$ is the boundary of an oriented disk in $X$, we say that the discrete integral is a \textit{discrete contour integral} with \textit{discrete contour} $P$. Similarly, the \textit{discrete integral} of $\Omega$ over a set of faces of $X$ is defined.

It turns out that discrete one-forms that actually come from discrete one-forms on $\Gamma$ and $\Gamma^*$ are of particular interest. We say that such a discrete one-form $\omega$ is of \textit{type} $\Diamond$. It is characterized by the property that for any $Q \in V(\Diamond)$ there exist complex numbers $p,q$ such that $\omega=p dz+qd\bar{z}$ on all edges $e=[Q,v]$, $v \sim Q$. In an analogous fashion we can define discrete one-forms $\omega'$ of \textit{type} $\Lambda$ by representing them as $\omega'=p dz+qd\bar{z}$ on all edges $e=[v,Q]$, $Q \sim v$, for some fixed vertex $v \in V(\Lambda)$ and complex numbers $p,q$ attached to it. However, these discrete one-forms do not play such an important role than the others. We mainly give them to get a complete picture.

As we are mainly interested in functions $f:V(\Lambda)\to\mC$ and $h:V(\Diamond)\to\mC$, discrete two-forms of particular interest are those that vanish on faces of $X$ corresponding to vertices of either $\Diamond$ or $\Lambda$. We will call them of being of \textit{type} $\Lambda$ or \textit{type} $\Diamond$, respectively. These discrete two-forms correspond to functions on $V(\Lambda)$ or $V(\Diamond)$ by the discrete Hodge star that will be defined later in Section~\ref{sec:hodge}.

\begin{definition}
Let $F$ be a face of the medial graph $X$. We define $\textnormal{ar}(F)$ to be twice the Euclidean area of $F$ in the complex plane. In contrast, $\textnormal{area}(P)$ will always denote the Euclidean area of a polygon $P$.
\end{definition}

\begin{remark}
As we have mentioned before, our main objects either live on the quad-graph $\Lambda$ or on its dual $\Diamond$. Thus, we have to deal with two different cellular decompositions at the same time. The medial graph has the crucial property that its faces are in one-to-one correspondence to vertices of $\Lambda$ and of $\Diamond$, i.e., to faces of $\Diamond$ and of $\Lambda$. Furthermore, the Euclidean area of the Varignon parallelogram inside $Q\in V(\Diamond)$ is just half of the area of $Q$. In some sense, a corresponding statement is true for the cells of $X$ corresponding to vertices of $\Lambda$, i.e., faces of $\Diamond$. However, there is no natural embedding of $\Diamond$ in the general setting. But in the particular case of parallelogram-graphs, we can make the statement precise: If an edge $QQ'$ of $\Diamond$ is represented by the two line segments that connect the centers of the parallelogram $Q$ and $Q'$ with the midpoint of their common edge, then the Euclidean area of the face of $X$ corresponding to a vertex $v \in V(\Lambda) \cong F(\Diamond)$ is exactly half of the area of the face $v$.

In summary, the medial graph allows us to deal with just one decomposition of the complex plane, but we have to count areas twice in order to get the right prefactors as in the continuous setup.
\end{remark}

For later purposes, let us compute $\textnormal{ar}(F)$. First, if $F$ is the Varignon parallelogram corresponding to the quadrilateral $Q\in V(\Diamond)$, \[\textnormal{ar}(F) =\frac{1}{2}|b_+-b_-| |w_+-w_-| \sin(\varphi_Q).\] Note again that the Euclidean area of $F$ is just half of $\textnormal{ar}(F)$. Second, if $F$ is the face of $X$ corresponding to a vertex $v\in V(\Lambda)$, its Euclidean area is a quarter of the area of the polygon $v'_1v'_2\ldots v'_k$ in the star of $v$. So $\textnormal{ar}(F)$ equals \[\frac{1}{2}\sum\limits_{Q_s \sim v}\textnormal{area}(\triangle vv'_{s-1} v'_{s})=\frac{1}{4}\sum\limits_{Q_s \sim v}\im\left(\left(v'_s-v\right)\overline{\left(v'_{s-1}-v\right)}\right)=\frac{1}{4}\sum\limits_{Q_s \sim v}\im\left(v'_s\bar{v}'_{s-1}\right),\] using that $\sum_{Q_s \sim v} \left(v \bar{v}'_{s-1}+\bar{v}v'_{s}\right)=\sum_{Q_s \sim v} \left(v \bar{v}'_{s}+\bar{v}v'_{s}\right)$ is real.

\begin{definition}
The discrete one-forms $dz$ and $d\bar{z}$ are given by $\int_e dz=e$ and $\int_e d\bar{z}=\bar{e}$ for any oriented edge $e$ of $X$. The discrete two-forms $\Omega_\Lambda$ and $\Omega_\Diamond$ are zero on faces of $X$ corresponding to vertices of $\Diamond$ or $\Lambda$, respectively, and defined by \[\iint_F \Omega_\Lambda=-2i\textnormal{ar}(F) \textnormal{ and } \iint_F \Omega_\Diamond=-2i\textnormal{ar}(F)\] on faces $F$ corresponding to vertices of $\Lambda$ or $\Diamond$, respectively. As defined above, $\textnormal{ar}(F)$ is twice the Euclidean area of $F$.
\end{definition}

\begin{remark}
$\Omega_\Lambda$ and $\Omega_\Diamond$ are the straightforward discretizations of $dz \wedge d\bar{z}$ having in mind that they are essentially defined on faces of $\Diamond$ or of $\Lambda$, respectively. It turns out that in local coordinates, we can perform our calculations with $\Omega_\Lambda$ and $\Omega_\Diamond$ in the discrete setting exactly as we do with $dz \wedge d\bar{z}$ in the smooth theory. Indeed, we will see in Section~\ref{sec:wedge} that $\Omega_\Diamond$ is indeed the discrete wedge product of $dz$ and $d\bar{z}$ seen as discrete one-forms of type $\Diamond$. The same is true for $\Omega_\Lambda$ if we consider $dz$ and $d\bar{z}$ as being of type $\Lambda$, but the discrete wedge-product is of interest just for discrete one-forms of type $\Diamond$.
\end{remark}

\begin{definition}
Let $f:V(\Lambda)\to\mC$, $h:V(\Diamond)\to\mC$, $\omega$ a discrete one-form, and $\Omega_1,\Omega_2$ discrete two-forms of type $\Lambda$ and $\Diamond$, respectively. For any edge $e=[Q,v]$ and any faces $F_v, F_Q$ of $X$ corresponding to the vertex star of $v\in V(\Lambda)$ or the Varignon parallelogram inside $Q \in V(\Diamond)$, we define the products $f\omega$, $h\omega$, $f\Omega_1$, and $h\Omega_2$ by
\begin{align*}
\int\limits_{e}f\omega:&=f(v)\int\limits_{e}\omega \ \quad \textnormal{ and } \quad \iint\limits_{F_v} f\Omega_1:=f(v)\iint\limits_{F_v}\Omega_1, \quad \iint\limits_{F_Q} f\Omega_1:=0;\\
\int\limits_{e}h\omega:&=h(Q)\int\limits_{e}\omega \quad \textnormal{ and } \quad \iint\limits_{F_v} h\Omega_2:=0, \qquad \qquad \quad \iint\limits_{F_Q} h\Omega_2:=h(Q)\iint\limits_{F_Q}\Omega_2.
\end{align*}
\end{definition}

In the following table, we give a quick overview of various discrete differential forms most of them will be discussed in Section~\ref{sec:exterior} and state whether they are essentially functions on $V(\Lambda)$ or functions on $V(\Diamond)$ or entirely objects on the cellular decomposition $X$. Although discrete one-forms of type $\Lambda$ or of type $\Diamond$ do not live themselves on $\Lambda$ or $\Diamond$, they are described by two functions defined on the vertices of $\Lambda$ or $\Diamond$, respectively. 

\begin{center}
\begin{tabular}{|c|c|c|c|} \hline
 & \textbf{$\Lambda$} & \textbf{$\Diamond$} & \textbf{$X$}\\
\hline 0-forms & $f,g : V(\Lambda) \to \mC$ & $h_1,h_2 : V(\Diamond) \to \mC$ & $f\cdot g=\int (fdg+gdf)$ \\
 & $\partial_{\Diamond} h, \bar{\partial}_{\Diamond} h$ & $\partial_{\Lambda} f, \bar{\partial}_{\Lambda} f$ &  \\
\hline
1-forms & $dh$ & $df$ & $fdg+gdf$ \\
& $h_1 dz + h_2 d\bar{z}$ & $f dz + gd\bar{z}$ & $fhdz$ \\
 & $\eta$ of type $\Lambda$ & $\omega, \omega'$ of type $\Diamond$ & $f\omega$ \\
\hline
2-forms & $\Omega_\Lambda$ & $\Omega_\Diamond$ &  \\
 & $\star f$ & $\star h$ & \\
 & $d\omega$ & $d\eta$ & $d(fhdz)$ \\
& $fd\omega$ & $\omega \wedge \omega'$  & $d(f\omega)$\\
\hline
\end{tabular}
\end{center}


\subsubsection{Discrete derivatives of functions on \texorpdfstring{$V(\Diamond)$}{the faces of the quad-graph}} \label{sec:derivative_diamond}

Before we pass on to discrete derivatives of functions on $V(\Diamond)$, we first prove an alternative formula for the discrete derivatives of functions on $V(\Lambda)$.

\begin{lemma}\label{lem:derivative_lambda}
Let $Q \in V(\Diamond)$ and $f$ be a complex function on the vertices of $Q$. Let $P_Q$ be the discrete elementary cycle around $Q$ and $F$ the face of $X$ corresponding to $Q$. Then,
\begin{align*}
 \partial_\Lambda f(Q)&=\frac{-1}{2i\textnormal{ar}(F)}\oint\limits_{P_Q} f d\bar{z},\\
 \bar{\partial}_\Lambda f (Q)&=\frac{1}{2i\textnormal{ar}(F)}\oint\limits_{P_Q} f dz.
\end{align*}
\end{lemma}
\begin{proof}
Since $F$ is a parallelogram, $f(b_+)$ and $-f(b_-)$ are multiplied by the same factor $\overline{(w_+-w_-)}/2$ when evaluating the discrete contour integral $\oint_{P_Q} f d\bar{z}$. Therefore, the prefactor in front of $f(b_+)-f(b_-)$ in the right hand side of the first equation is \[i\frac{\overline{w_+-w_-}}{4\textnormal{ar}(F)}=-\overline{i}\frac{\overline{w_+-w_-}}{2 \sin(\varphi_Q)|w_+-w_-| |b_+-b_-|}=\frac{\exp\left(-i\left(\varphi_Q-\frac{\pi}{2}\right)\right)}{2\sin(\varphi_Q)(b_+-b_-)}=\frac{\lambda_Q}{b_+-b_-}\] (compare with the proof of Proposition~\ref{prop:examples}~(i)), which is exactly the prefactor appearing in $\partial_\Lambda f(Q)$. In an analogous manner, the prefactors in front of $f(w_+)-f(w_-)$ are equal. This shows the first equation. The second one follows from the first, noting that the prefactors in front of $f(b_+)-f(b_-)$ and $f(w_+)-f(w_-)$ on both sides of the second equation are just complex conjugates of the corresponding prefactors appearing in the first equation.
\end{proof}

\begin{figure}[htbp]
   \centering
    \subfloat[Lemma~\ref{lem:derivative_lambda} for $\partial_\Lambda,\bar{\partial}_\Lambda$]{
    \beginpgfgraphicnamed{derivative_lambda}
			\begin{tikzpicture}[white/.style={circle,draw=black,fill=white,thin,inner sep=0pt,minimum size=1.2mm},
black/.style={circle,draw=black,fill=black,thin,inner sep=0pt,minimum size=1.2mm},
gray/.style={circle,draw=black,fill=gray,thin,inner sep=0pt,minimum size=1.2mm},scale=0.7]
			\clip(-1.5,-6.5) rectangle (7.2,-0.2);
			\draw (-0.6,-4.16)-- (2.02,-1.28);
			\draw (2.02,-1.28)-- (6.4,-4.16);
			\draw (6.4,-4.16)-- (2.94,-5.52);
			\draw (2.94,-5.52)-- (-0.6,-4.16);
			\draw [color=gray] (1.17,-4.84)-- (0.71,-2.72);
			\draw [color=gray] (0.71,-2.72)-- (4.21,-2.72);
			\draw [color=gray] (4.21,-2.72)-- (4.67,-4.84);
			\draw [color=gray] (4.67,-4.84)-- (1.17,-4.84);
			\node[white] (w1) [label=below:$w_-$] at (2.94,-5.52) {};
			\node[white] (w2) [label=above:$w_+$] at (2.02,-1.28) {};
			\node[black] (b1) [label=left:$b_-$] at (-0.6,-4.16) {};
			\node[black] (b2) [label=right:$b_+$] at (6.4,-4.16) {};
			\node[gray] (m1) at (0.71,-2.72) {};
			\node[gray] (m2) at (4.21,-2.72) {};
			\node[gray] (m3) at (4.67,-4.84) {};
			\node[gray] (m4) at (1.17,-4.84) {};
			\draw (3.06,-3.2) node {$P_Q$};
		\end{tikzpicture}
		\endpgfgraphicnamed}
		\qquad
		\subfloat[Definitions of $\partial_\Diamond,\bar{\partial}_\Diamond$]{
		\beginpgfgraphicnamed{derivative_diamond}
		\begin{tikzpicture}
		[white/.style={circle,draw=black,fill=white,thin,inner sep=0pt,minimum size=1.2mm},
black/.style={circle,draw=black,fill=black,thin,inner sep=0pt,minimum size=1.2mm},
gray/.style={circle,draw=black,fill=gray,thin,inner sep=0pt,minimum size=1.2mm},scale=0.75]
	\clip(2,-4) rectangle (10.8,2);
				\draw [dash pattern=on 5pt off 5pt] (7.83,0.18)-- (3.64,0.54);
				\draw [dash pattern=on 5pt off 5pt] (3.64,0.54)-- (4.31,-1.61);
				\draw [dash pattern=on 5pt off 5pt] (4.31,-1.61)-- (6.51,-3.02);
				\draw [dash pattern=on 5pt off 5pt] (8.3,-1.39)-- (6.51,-3.02);
				\draw [dash pattern=on 5pt off 5pt] (8.3,-1.39)-- (7.83,0.18);
				\draw (2.4,-1.03)-- (3.64,0.54);
				\draw (3.64,0.54)-- (6.18,1.3);
				\draw (3.64,0.54)-- (5.56,-1.12);
				\draw (5.56,-1.12)-- (4.31,-1.61);
				\draw (4.31,-1.61)-- (2.4,-1.03);
				\draw (4.31,-1.61)-- (4.6,-3.59);
				\draw (9.82,-0.09)-- (7.83,0.18);
				\draw (9.82,-0.09)-- (8.3,-1.39);
				\draw (7.83,0.18)-- (5.56,-1.12);
				\draw (8.3,-1.39)-- (5.56,-1.12);
				\draw (7.83,0.18)-- (6.18,1.3);
				\draw (5.56,-1.12)-- (6.51,-3.02);
				\draw (6.51,-3.02)-- (4.6,-3.59);
				\draw (6.51,-3.02)-- (8.48,-3.38);
				\draw (8.48,-3.38)-- (8.3,-1.39);
				\draw [color=gray] (4.94,-1.36)-- (4.6,-0.29);
				\draw [color=gray] (4.6,-0.29)-- (6.7,-0.47);
				\draw [color=gray] (6.7,-0.47)-- (6.93,-1.25);
				\draw [color=gray] (6.93,-1.25)-- (6.04,-2.07);
				\draw [color=gray] (6.04,-2.07)-- (4.94,-1.36);
					\node[white] (w1) [label=right:$v_s$] at (9.82,-0.09) {};
					\node[white] (w2) [label=above:$v$] at (5.56,-1.12) {};
					\node[white] (w3) at (2.4,-1.03) {};
					\node[white] (w4) at (8.48,-3.38) {};
					\node[white] (w5) at (4.6,-3.59) {};
					\node[white] (w6) at (6.18,1.3) {};
					\node[black] (b1) [label=above right:$v'_s$] at (7.83,0.18) {};
					\node[black] (b2) [label=below right:$v'_{s-1}$] at (8.3,-1.39) {};
					\node[black] (b3) at (3.64,0.54) {};
					\node[black] (b4) at (6.51,-3.02) {};
					\node[black] (b5) at (4.31,-1.61) {};
					\node[gray] (m0) [label=right:$Q_s$] at (8.07,-0.61) {};		
					\node[gray] (m1) at (6.93,-1.25) {};
					\node[gray] (m2) at (6.7,-0.47) {};
					\node[gray] (m3) at (4.6,-0.29) {};
					\node[gray] (m4) at (4.94,-1.36) {};
					\node[gray] (m5) at (6.04,-2.07) {};
					\draw (5.32,0) node {$P_v$};
			\end{tikzpicture}
		\endpgfgraphicnamed}
   \caption[]{Integration formulae for discrete derivatives}
   \label{fig:derivative_cycles}
\end{figure}
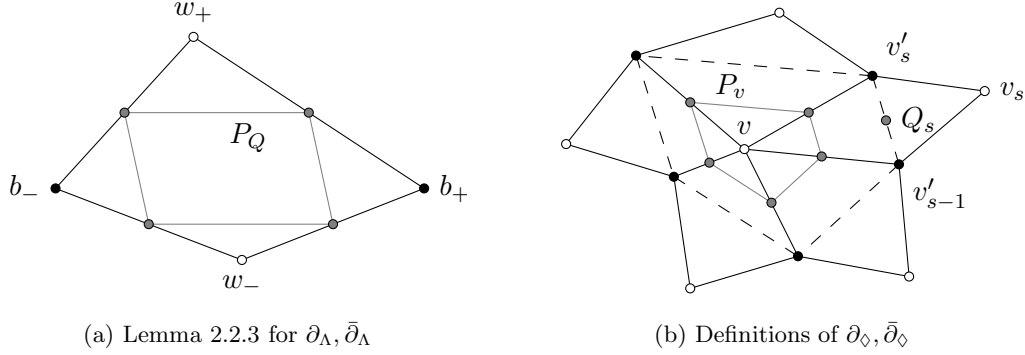

Inspired by Lemma~\ref{lem:derivative_lambda}, we can now define the discrete derivatives for complex functions on $V(\Diamond)$, see Figure~\ref{fig:derivative_cycles}~(b).

\begin{definition}
Let $v\in V(\Lambda)$ and $h$ be a complex function defined on all quadrilaterals $Q_s \sim v$. Let $P_v$ be the discrete elementary cycle around $v$ and $F$ the face of $X$ corresponding to $v$. Then, the \textit{discrete derivatives} $\partial_\Diamond h$, $\bar{\partial}_\Diamond h$ at $v$ are defined by
\begin{align*}
 \partial_\Diamond h(v)&:=\frac{-1}{2i\textnormal{ar}(F)}\oint\limits_{P_v} h d\bar{z},\\
 \bar{\partial}_\Diamond h(v)&:=\frac{1}{2i\textnormal{ar}(F)}\oint\limits_{P_v} h dz.
\end{align*}
$h$ is said to be \textit{discrete holomorphic} at $v$ if $\bar{\partial}_\Diamond h (v)=0$.
\end{definition}

Note that in the rhombic case, our definition coincides with the one used by Chelkak and Smirnov in \cite{ChSm11}. As an immediate consequence of the definition, we obtain a \textit{discrete Morera's theorem}.

\begin{proposition}\label{prop:morera}
Functions $f:V(\Lambda)\to\mC$ and $h:V(\Diamond)\to\mC$ are discrete holomorphic if and only if $\oint_{P} f dz=0$ and $\oint_{P} h dz=0$ for all discrete contours $P$.
\end{proposition}
\begin{proof}
Clearly, $\oint_{P_v} f dz=f(v)\oint_{P_v} dz=0$ for any discrete elementary cycle $P_v$ around a vertex $v$ of $V(\Lambda)$. Similarly, $\oint_{P_Q} h dz=0$ for any $Q\in V(\Diamond)$. Using Lemma~\ref{lem:derivative_lambda} and the definition of $\bar{\partial}_\Diamond$, $f$ and $h$ are discrete holomorphic if and only if $\oint_{P} f dz=0$ and $\oint_{P} h dz=0$ for all discrete elementary cycles $P$. To conclude the proof, we observe that any integration along a discrete contour can be decomposed into integrations along discrete elementary cycles.
\end{proof}

The discrete derivatives of constant functions on $V(\Diamond)$ vanish. As an analog of Proposition~\ref{prop:examples}, we prove that the discrete derivatives $\partial_\Diamond,\bar{\partial}_\Diamond$ locally approximate their smooth counterparts correctly up to order one if the vertices $Q$ are placed at the midpoints of black or white edges. Note that even for rhombic quad-graphs, these discrete derivatives generally do not coincide with the smooth derivatives in order two.

\begin{proposition}\label{prop:examples2}
Let $v \in V(\Lambda)$, and let $h$ be a complex function on all faces $Q_s\sim v$. Assume that the vertices $Q_s$ are placed at the midpoints $(v'_{s-1}+v'_{s})/2$ of edges of the star of $v$. Then, if $h(Q)=Q$, $\bar{\partial}_\Diamond h(v)=0$ and $\partial_\Diamond h (v)= 1$.
\end{proposition}
\begin{proof}
\begin{align*}
4\oint\limits_{P_v} h dz&=\sum\limits_{Q_s\sim v} (v'_{s-1}+v'_{s}) (v'_{s}-v'_{s-1})=\sum\limits_{Q_s\sim v} \left(\left(v'_{s}\right)^2-\left(v'_{s-1}\right)^2\right)=0,\\
4\oint_{P_v} h d\bar{z}&=\sum\limits_{Q_s\sim v} (v'_{s-1}+v'_{s}) \overline{(v'_{s}-v'_{s-1})}=\sum\limits_{Q_s\sim v} \left(\left|v'_{s}\right|^2-\left|v'_{s-1}\right|^2-2i\im\left(v'_s\bar{v}'_{s-1} \right)\right)\\&=-2i\sum\limits_{Q_s\sim v} \im\left(v'_s\bar{v}'_{s-1} \right)=-8i\textnormal{ar}(F_v).
\end{align*}
These equations yield $\bar{\partial}_\Diamond h(v)=0$ and $\partial_\Diamond h(v)=1$.
\end{proof}

\begin{remark}
As before, parallelogram-graphs play a special role. In a parallelogram-graph, the midpoint of a black edge equals the midpoint of its dual white edge. Placing the vertices $Q\in V(\Diamond)$ at the centers of the corresponding parallelograms then yields a global approximation statement. 
\end{remark}

In \cite{ChSm11}, Chelkak and Smirnov used averaging operators to map functions on $V(\Lambda)$ to functions on $V(\Diamond)$ and vice versa. On parallelogram-graphs, the \textit{averaging operator} $m(f)(Q):=\sum_{v \sim Q}f(v)/4$ actually maps discrete holomorphic functions $f:V(\Lambda)\to\mC$ to discrete holomorphic functions on $V(\Diamond)$. The corresponding statement for rhombic quad-graphs was shown in \cite{ChSm11}, our proof is similar.

\begin{proposition}\label{prop:average}
Let $\Lambda$ be a parallelogram-graph and $f:V(\Lambda)\to\mC$ be discrete holomorphic. Then, $m(f):V(\Diamond)\to\mC$ is discrete holomorphic.
\end{proposition}
\begin{proof}
Let us consider the star of the vertex $v\in V(\Lambda)$. Since $f$ is discrete holomorphic, the discrete Cauchy-Riemann equation is satisfied on any $Q_s \sim v$. Therefore, we can express $f(v_s)$ in terms of $f(v)$, $f(v'_s)$ and $f(v'_{s-1})$. Plugging this in the definition of the averaging operator, we obtain \begin{align*}4m(f)(Q_s)&=2f(v)+\frac{v_s-v+v'_{s}-v'_{s-1}}{v'_{s}-v'_{s-1}}f(v'_{s})+\frac{v_s-v-v'_{s}+v'_{s-1}}{v'_{s}-v'_{s-1}}f(v'_{s-1})\\&=2f(v)+2\frac{v'_{s}-v}{v'_{s}-v'_{s-1}}f(v'_{s}) -2\frac{v'_{s-1}-v}{v'_{s}-v'_{s-1}}f(v'_{s-1}).\end{align*} Here, we have used the properties $v_s-v'_{s-1}=v'_{s}-v$ and $v_s-v'_{s}=v'_{s-1}-v$ of the parallelogram $Q_s$. Therefore, \[2\oint\limits_{P_v}m(f)dz=f(v)\oint\limits_{P_v}dz+\sum\limits_{Q_s \sim v}(v'_{s}-v)f(v'_{s})-\sum\limits_{Q_s \sim v}(v'_{s-1}-v)f(v'_{s-1})=0,\] and $m(f)$ is discrete holomorphic at $v$.
\end{proof}
\begin{remark}
As noted by Chelkak and Smirnov, discrete holomorphic functions on $V(\Diamond)$ cannot be averaged to discrete holomorphic functions on $V(\Lambda)$ in general.
\end{remark}

As mentioned above, our main interest lies in functions that are defined either on the vertices or the faces of the quad-graph. Now, extending $f: V(\Lambda) \to \mC$ to a complex function on $F(X)$ by using its average $m(f)$ on $V(\Diamond)$ seems to be an option. However, functions on $V(\Lambda)$ and on $V(\Diamond)$ behave differently. In Corollary~\ref{cor:commutativity} we will see that $\partial_\Lambda f$ is discrete holomorphic if $f$ is, but $\partial_\Diamond m(f)$ does not need to be discrete holomorphic in general. So to make sense of differentiating twice, we can only consider functions on $V(\Lambda)$.

\begin{definition}
Let $f_1,f_2:V(\Lambda)\to\mC$ and $h_1,h_2:V(\Diamond)\to\mC$. Their \textit{discrete scalar products} are defined as \[\langle f_1,f_2 \rangle := -\frac{1}{2i}\iint\limits_{F(X)} f_1\bar{f}_2\Omega_\Lambda \quad \textnormal{and} \quad \langle h_1,h_2 \rangle := -\frac{1}{2i}\iint\limits_{F(X)} h_1\bar{h}_2\Omega_\Diamond,\]
whenever the right hand side converges absolutely.
\end{definition}

\begin{proposition}\label{prop:adjoint}
$-\partial_\Diamond$ and $-\bar{\partial}_\Diamond$ are the \textit{formal adjoints} of $\bar{\partial}_\Lambda$ and $\partial_\Lambda$, respectively. That is, if $f:V(\Lambda)\to\mC$ or $h:V(\Diamond) \to \mC$ is compactly supported, \[\langle\partial_\Lambda f, h\rangle+\langle f, \bar{\partial}_\Diamond h \rangle=0=\langle\bar{\partial}_\Lambda f, h\rangle+\langle f, \partial_\Diamond h\rangle.\]
\end{proposition}
\begin{proof}
Using Lemma~\ref{lem:derivative_lambda} and $\partial_\Diamond \bar{h}=\overline{\bar{\partial}_\Diamond h}$, we get \[ -2i\langle\partial_\Lambda f, h\rangle-2i\langle f, \bar{\partial}_\Diamond h \rangle=\sum\limits_{Q\in V(\Diamond)} \bar{h}(Q) \oint\limits_{P_Q} fd\bar{z}+\sum\limits_{v\in V(\Lambda)} f(v) \oint\limits_{P_v} \bar{h} d\bar{z}=\oint\limits_P f\bar{h}d\bar{z}=0,\] where $P$ is a large contour enclosing all the vertices of $\Lambda$ and $\Diamond$ where $f$ or $h$ do not vanish. In particular, $f\bar{h}$ vanishes in a neighborhood of $P$. In the same way, $\langle\bar{\partial}_\Lambda f, h\rangle+\langle f, \partial_\Diamond h\rangle=0$.
\end{proof}

\begin{remark}
In their work on discrete complex analysis on rhombic quad-graphs, Kenyon \cite{Ke02} and Mercat \cite{Me07} defined the discrete derivatives for functions on the faces in such a way that they were the formal adjoints of the discrete derivatives for functions on the vertices of the quad-graph.
\end{remark}

In Corollary~\ref{cor:commutativity}, we will prove that $\partial_\Lambda f$ is discrete holomorphic if $f:V(\Lambda) \to \mC$ is. Conversely, we can find discrete primitives of discrete holomorphic functions on simply-connected domains $\Diamond_0$, extending the corresponding result for rhombic quad-graphs given in the paper of Chelkak and Smirnov \cite{ChSm11}.

\begin{proposition}\label{prop:primitive}
Let $\Diamond_0 \subseteq \Diamond$ be simply-connected. Then, for any discrete holomorphic function $h$ on $V(\Diamond_0)$, there is a \textit{discrete primitive} $f:=\int h$ on $V(\Lambda_0)$, i.e., $f$ is discrete holomorphic and $\partial_\Lambda f=h$. $f$ is unique up to two additive constants on $\Gamma_0$ and $\Gamma^*_0$.
\end{proposition}
\begin{proof}
Since $h$ is discrete holomorphic, $\oint_P hdz=0$ for any discrete contour $P$. Thus, $h dz$ can be integrated to a well defined function $f_X$ on $V(X)$ that is unique up to an additive constant. Using that $h dz$ is a discrete one-form of type $\Diamond$, we can construct a function $f$ on $V(\Lambda)$ such that the equation $f_X\left(\left(v+w\right)/2\right)=\left(f(v)+f(w)\right)/2$ holds for any edge $(v,w)$ of $\Lambda$. Given $f_X$, $f$ is unique up to an additive constant.

In summary, $f$ is unique up to two additive constants that can be chosen independently on $\Gamma_0$ and $\Gamma^*_0$. By construction, $f$ satisfies \[\frac{f(b_+)-f(b_-)}{b_+-b_-}=h(Q)=\frac{f(w_+)-f(w_-)}{w_+-w_-}\] on any quadrilateral $Q$. So $f$ is discrete holomorphic and $\partial_\Lambda f=h$.
\end{proof}


\subsection{Discrete exterior calculus}\label{sec:exterior}

Our notation of discrete exterior calculus is similar to the approach of Mercat in \cite{Me01,Me07,Me08}, but differs in some aspects. The main differences are due to our different notation of multiplication of functions with discrete one-forms, which allows us to define a discrete exterior derivative on a larger class of discrete one-forms in Section~\ref{sec:exterior_derivative}. It coincides with Mercat's discrete exterior derivative in the case of discrete one-forms of type $\Diamond$ that Mercat considers. In contrast, our definitions are based on a coordinate representation, making the connection to the smooth case evident. Eventually, they lead to essentially the same definitions of a discrete wedge product in Section~\ref{sec:wedge} and a discrete Hodge star in Section~\ref{sec:hodge} as in \cite{Me08}.


\subsubsection{Discrete exterior derivative}\label{sec:exterior_derivative}

\begin{definition}
 Let $f:V(\Lambda) \to \mC$ and $h:V(\Diamond) \to \mC$. We define the \textit{discrete exterior derivatives} $df$ and $dh$ as follows:
\begin{align*}
 df&:=\partial_\Lambda f dz+\bar{\partial}_\Lambda f d\bar{z};\\
 dh&:=\partial_\Diamond h dz+\bar{\partial}_\Diamond h d\bar{z}.
\end{align*}

Let $\omega$ be a discrete one-form. Around faces $F_v$ and $F_Q$ of $X$ corresponding to vertices $v \in V(\Lambda)$ and $Q \in V(\Diamond)$, respectively, we write $\omega=p dz+ q d\bar{z}$ with functions $p,q$ defined on faces $Q_s \sim v$ or vertices $b_\pm, w_\pm \sim Q$, respectively. The \textit{discrete exterior derivative} $d\omega$ is given by
\begin{align*}
 d\omega|_{F_v}&:=\left(\partial_\Diamond q - \bar{\partial}_\Diamond p\right)   \Omega_\Lambda,\\
 d\omega|_{F_Q}&:=\left(\partial_\Lambda q - \bar{\partial}_\Lambda p\right)  \Omega_\Diamond.
\end{align*}
\end{definition}

The representation of $\omega$ as $p dz+ q d\bar{z}$ ($p,q$ defined on edges of $X$) we have used above may be nonunique. However, $d\omega$ is well defined by \textit{discrete Stokes' theorem} that also justifies our definition of $df$ and $d\omega$. Note that Mercat defined the discrete exterior derivative by the discrete Stokes' theorem \cite{Me01}.

\begin{theorem}\label{th:stokes}
 Let $f:V(\Lambda) \to \mC$ and $\omega$ be a discrete one-form. Then, for any directed edge $e$ of $X$ starting in the midpoint of the edge $vv'_-$ and ending in the midpoint of the edge $vv'_+$ of $\Lambda$ and for any face $F$ of $X$ with counterclockwise oriented boundary $\partial F$ we have:

\begin{align*}
 \int\limits_e df&=\frac{f(v'_+)-f(v'_-)}{2}=\frac{f(v)+f(v'_+)}{2}-\frac{f(v)+f(v'_-)}{2};\\
 \iint\limits_F d\omega&=\oint\limits_{\partial F} \omega.
\end{align*}
\end{theorem}

\begin{proof}
Let $v_-$ be the other vertex of the quadrilateral $Q$ with vertices $v$, $v'_-$ and $v'_+$. Without loss of generality, let $v$ be white. Then, $\int_e df$ equals
\begin{align*}
& \partial_\Lambda f \frac{v'_+-v'_-}{2}+\bar{\partial}_\Lambda f \frac{\overline{v'_+-v'_-}}{2}\\=&\frac{1}{2}(\lambda_Q+\bar{\lambda}_Q)(f(v'_+)-f(v'_-))+\frac{1}{2}\left(\bar{\lambda}_Q\frac{v'_+-v'_-}{v-v_-}+\lambda_Q\overline{\frac{v'_+-v'_-}{v-v_-}}\right)\left(f(v)-f(v_-)\right)\\=&\frac{f(v'_+)-f(v'_-)}{2}+\re\left(\bar{\lambda}_Q\frac{v'_+-v'_-}{v-v_-}\right)\left(f(v)-f(v_-)\right)\\=&\frac{f(v'_+)-f(v'_-)}{2}.
\end{align*}
Here, we have used $\lambda_Q+\bar{\lambda}_Q=1$ and \[\textnormal{arg}\left(\bar{\lambda}_Q\frac{v'_+-v'_-}{v-v_-}\right)=\textnormal{arg}\left(\pm\exp\left(i\left(\varphi_Q-\frac{\pi}{2}\right)\right)\exp\left(-i\varphi_Q\right)\right)=\pm \pi/2.\] The sign depends on the orientation of the vertices $v,v'_-v_-,v'_+$. But in any case, the expression inside $\textnormal{arg}$ is purely imaginary.

Let us write $\omega=p dz+q d\bar{z}$ around $F_Q$ or $F_v$, where $p,q$ are functions defined on the vertices of $Q\in V(\Diamond)$ or on the faces incident to $v\in V(\Lambda)$. Then, by Lemma~\ref{lem:derivative_lambda} and the definition of the discrete derivatives $\partial_\Diamond,\bar{\partial}_\Diamond$, \begin{align*}\iint\limits_{F_Q} d\omega&=\iint\limits_{F_Q} \left(\partial_\Diamond q -\bar{\partial}_\Diamond p\right)\Omega_\Diamond=-2i\textnormal{ar}(F_Q)\left(\partial_\Diamond q -\bar{\partial}_\Diamond p \right)=\oint\limits_{\partial F_Q}pdz+\oint\limits_{\partial F_Q}qd\bar{z}=\oint\limits_{\partial F_Q} \omega,\\ \iint\limits_{F_v} d\omega&=\iint\limits_{F_v} \left(\partial_\Lambda q -\bar{\partial}_\Lambda p\right)\Omega_\Lambda=-2i\textnormal{ar}(F_v)\left(\partial_\Lambda q -\bar{\partial}_\Lambda p \right)=\oint\limits_{\partial F_v}pdz+\oint\limits_{\partial F_v}qd\bar{z}=\oint\limits_{\partial F_v} \omega.\end{align*}
\end{proof}

Note that if $\omega$ is a discrete one-form of type $\Diamond$, $\iint_{F} d\omega=0$ for any face $F$ corresponding to a vertex of $\Diamond$. A discrete one-form $\omega$ is said to be \textit{closed} if $d\omega\equiv 0$. Examples for closed discrete one-forms are discrete exterior derivatives of complex functions on $V(\Lambda)$:

\begin{proposition}\label{prop:dd0}
 Let $f:V(\Lambda) \to \mC$. Then, $ddf=0$.
\end{proposition}
\begin{proof}
By discrete Stokes' Theorem~\ref{th:stokes}, $ddf=0$ if $\oint_P df=0$ for any discrete elementary cycle $P$. Since $df$ is of type $\Diamond$, the statement is trivially true if $P=P_Q$ for $Q\in V(\Diamond)$. So let $P=P_v$ for $v \in V(\Lambda)$. Using discrete Stokes' Theorem~\ref{th:stokes} again, \[\oint\limits_{P_v} df=\sum\limits_{Q_s \sim v} \frac{f(v'_{s})-f(v'_{s-1})}{2}=0.\]
\end{proof}
An immediate corollary of the last proposition is the commutativity of discrete differentials, generalizing the known result for rhombic quad-graphs as provided in \cite{ChSm11}.

\begin{corollary}\label{cor:commutativity}
Let $f:V(\Diamond) \to \mC$. Then, $\partial_\Diamond\bar{\partial}_\Lambda f\equiv\bar{\partial}_\Diamond\partial_\Lambda f.$ In particular, $\partial_\Lambda f$ is discrete holomorphic if $f$ is discrete holomorphic.
\end{corollary}
\begin{proof}
Due to Proposition~\ref{prop:dd0}, $0=ddf=\left(\partial_\Diamond \bar{\partial}_\Lambda f-\bar{\partial}_\Diamond \partial_\Lambda f\right) \Omega_\Lambda$.
\end{proof}

\begin{remark}
Note that even in the generic rhombic case, $\partial_\Lambda \bar{\partial}_\Diamond h$ does not always equal $\bar{\partial}_\Lambda\partial_\Diamond h$ for a function $h:V(\Diamond)\rightarrow\mC$ \cite{ChSm11}. Hence, an analog of Proposition~\ref{prop:dd0} cannot hold for such functions $h$ in general.
\end{remark}

\begin{corollary}\label{cor:f_holomorphic}
 Let $f:V(\Lambda) \to \mC$. Then, $f$ is discrete holomorphic if and only if $df=p dz$ for some $p:V(\Diamond) \to \mC$. In this case, $p$ is discrete holomorphic.
\end{corollary}
\begin{proof}
Since all quadrilaterals $Q$ are nondegenerate, the representation of $df|_{\partial{F_Q}}$ as $pdz+qd\bar{z}$ is unique (see also Lemma~\ref{lem:representation} below). Now, $df=\partial_\Lambda f dz+\bar{\partial}_\Lambda f d\bar{z}$. It follows that $f$ is discrete holomorphic at $Q$ if and only if $df|_{\partial{F_Q}}=p dz$.

Assuming that $df=pdz$ for some $p:V(\Diamond) \to \mC$, $ddf=0$ by Proposition~\ref{prop:dd0} yields $\bar{\partial}_\Diamond p\equiv0$.
\end{proof}

Let us say that a discrete one-form $\omega$ is \textit{discrete holomorphic} if $\omega=p dz$ for some $p:V(\Diamond) \to \mC$ and $d\omega=0$. This notion recurs in the more general setting of discrete Riemann surfaces in the thesis of the second author \cite{Gue14}. By Corollary~\ref{cor:f_holomorphic}, $df$ is discrete holomorphic if $f$ is, and by Proposition~\ref{prop:primitive}, any discrete holomorphic one-form $\omega$ on a simply-connected domain is the discrete exterior derivative of a discrete holomorphic function on $V(\Lambda)$.

Due to Chelkak and Smirnov \cite{ChSm11}, one of the unpleasant facts of all discrete theories of complex analysis is that (pointwise) multiplication of discrete holomorphic functions does not yield a discrete holomorphic function in general. We can define a product of complex functions on $V(\Lambda)$ that is defined on $V(X)$ and a product of complex functions on $V(\Lambda)$ with functions on $V(\Diamond)$ that is defined on $E(X)$. In general, the product of two discrete holomorphic functions is not discrete holomorphic according to the classical quad-based definition, but it will be discrete holomorphic in the sense that a discretization of its exterior derivative is closed and of the form $p dz$, $p:E(X)\to \mC$, or in the sense that it fulfills a discrete Morera's theorem.

\begin{corollary}\label{cor:discrete_product}
Let $f,g:V(\Lambda) \to \mC$ and $h:V(\Diamond) \to \mC$.
\begin{enumerate}
\item $fdg+g df$ is a closed discrete one-form.
\item If $f$ and $h$ are discrete holomorphic, $fh dz$ is a closed discrete one-form.
\end{enumerate}
\end{corollary}
\begin{proof}
 (i) Let $\omega:=fdg+g df$. By Proposition~\ref{prop:dd0}, $df$ and $dg$ are closed. Thus, \[\oint\limits_{\partial F_v} \omega =f(v) \oint\limits_{\partial F_v} dg+g(v)\oint\limits_{\partial F_v}df= 0\] for any face $F_v$ corresponding to $v\in V(\Lambda)$. Using Lemma~\ref{lem:derivative_lambda}, \begin{align*}2i\textnormal{ar}(F_Q)\oint\limits_{\partial F_Q} \omega &= 2i\textnormal{ar}(F_Q)\oint\limits_{\partial F_Q} \left(f\partial_\Lambda gdz+f\bar{\partial}_\Lambda gd\bar{z}+g\partial_\Lambda fdz+g\bar{\partial}_\Lambda fd\bar{z} \right)\\&=\bar{\partial}_\Lambda f \partial_\Lambda g - \partial_\Lambda f \bar{\partial}_\Lambda g +\bar{\partial}_\Lambda g \partial_\Lambda f -\partial_\Lambda g \bar{\partial}_\Lambda f=0\end{align*} for any face $F_Q$ corresponding to $Q\in V(\Diamond)$. It follows by discrete Stokes' Theorem~\ref{th:stokes} that $d\omega=0$.

(ii) By discrete Morera's Theorem~\ref{prop:morera}, $\oint_{\partial F} fh dz=0$ for any face $F$ of $X$, using that $f$ and $h$ are discrete holomorphic. Therefore, $fh dz$ is closed.
\end{proof}

\begin{remark}
In particular, a product $f\cdot g: V(X)\to \mC$ can be defined by integration and $f \cdot g$ is defined up to an additive constant. Furthermore, $f \cdot h:E(X) \to \mC$ can be defined by ``pointwise'' multiplication. If $f,g,h$ are discrete holomorphic, $fdg+g df=p dz$ is closed, where $p=f\cdot \partial_\Lambda g+g\cdot \partial_\Lambda f:E(X)\to \mC$, and so to say a discrete holomorphic one-form, meaning that $f\cdot g$ is discrete holomorphic in this sense. Similarly, $fh dz$ is closed, so $f\cdot h$ is discrete holomorphic in the sense that a discrete Morera's theorem holds true. However, $f\cdot g$ and $f\cdot h$ are generally not discrete holomorphic everywhere according to the classical quad-based definition of discrete holomorphicity on the dual of a bipartite quad-graph. For this, we place any $Q\in V(\Diamond)$ as a vertex somewhere in the interior of the corresponding quadrilateral.

In fact, $f\cdot g$ is a complex function on the vertices of $X$. The medial graph $X$ is not a quad-graph itself, but $X$ is the dual of a bipartite quad-graph. More precisely, $X$ is the dual of the bipartite quad-graph with vertex set $V(\Lambda)\sqcup V(\Diamond)$, edges connecting points $Q\in V(\Diamond)$ with all incident vertices $v\in V(\Lambda)$, and faces corresponding to edges of $\Lambda$. Then, $f\cdot g$ does not need to be a discrete holomorphic function on the faces of the latter quad-graph. For example, consider $f(v)=0$ if $v$ is black and $f(v)=1$ if $v$ is white, and a discrete holomorphic $g$ that is not biconstant. Then, the product $f\cdot g$ is not discrete holomorphic at all $Q \in V(\Diamond)$ (seen as vertices of the quad-graph described above) where $\partial_\Lambda g (Q)\neq 0$.

Furthermore, $f \cdot h$ is a complex function on the edges of $X$, so it is a function on the vertices of the medial graph of $X$. The medial graph of the medial graph of $\Lambda$ is usually not a quad-graph, but it is the dual of a bipartite quad-graph. Namely, it is the dual of the quad-graph with vertex set $\left(V(\Lambda)\cup V(\Diamond)\right)\sqcup V(X)$, edges connecting points $v\in V(\Lambda)$ or $Q\in V(\Diamond)$ with the midpoints of all incident edges, and each face having an edge of $X$ as a diagonal. Since $fhdz$ is closed, $f \cdot h$ is discrete holomorphic at vertices of $\Lambda$ or $\Diamond$ by discrete Morera's Theorem~\ref{prop:morera}. But there is no need for $f \cdot h$ to be discrete holomorphic at vertices of $X$, even for constant $h$ and $f$ defined by $f(v)=0$ if $v$ is black and $f(v)=1$ if $v$ is white.

In summary, we defined reasonable products $f\cdot g$ and $f \cdot h$, where $f,g:V(\Lambda) \to \mC$ and $h:V(\Diamond)\to\mC$ are discrete holomorphic. Somehow missing is a product $h\cdot h'$, where $h':V(\Diamond)\to\mC$. In the general case, we do not know an appropriate product so far. But we want to point out that Chelkak and Smirnov defined such a product for so-called \textit{spin holomorphic} functions $h,h'$ in \cite{ChSm12}. This product satisfies $\re\left(\bar{\partial}_{\Diamond}\left( h \cdot h'\right)\right)\equiv 0$.

\end{remark}


\subsubsection{Discrete wedge product}\label{sec:wedge}

Following Whitney \cite{Whi38}, Mercat defined in \cite{Me01} a discrete wedge product for discrete one-forms living on the edges of $\Lambda$. Then, the discrete exterior derivative defined by a discretization of Stokes' theorem is a derivation for the discrete wedge product. However, a discrete Hodge star cannot be defined on $\Lambda$. To circumvent this problem, Mercat used an averaging map to relate discrete one-forms on the edges of $\Lambda$ with discrete one-forms on the edges of $\Gamma$ and $\Gamma^*$, i.e., discrete one-forms of type $\Diamond$. Then, he could define a discrete Hodge star; however, the discrete exterior derivative was not a derivation for the now heteregoneous discrete wedge product anymore.

We propose a different interpretation of the discrete wedge product. It the end, we somehow recover the definitions Mercat proposed in \cite{Me01,Me07,Me08}, but our derivation is different. Starting with discrete one-forms of type $\Diamond$ that are defined on the edges of $X$, we obtain a discrete wedge product on the faces of $X$ that is of type $\Diamond$. This definition is different from Whitney's \cite{Whi38} and has the advantage that both a discrete wedge product and a discrete Hodge star can be defined on the same structure. In contrast to Mercat's work, we now can make sense out of the statement that the discrete exterior derivative is a derivation for the discrete wedge product, see Theorem~\ref{th:derivation}. This proposition is of crucial importance to deduce discrete integral formulae such as discrete Green's identities.

\begin{lemma}\label{lem:representation}
 Let $\omega$ be a discrete one-form of type $\Diamond$. Then, there is a unique representation $\omega=p dz+q d\bar{z}$ with functions $p,q:V(\Diamond)\to\mC$. On a quadrilateral $Q \in V(\Diamond)$, $p$ and $q$ are given by
\begin{align*}
 p(Q)&=\lambda_Q\frac{\int_e \omega}{e}+\bar{\lambda}_Q\frac{\int_{e^*} \omega}{e^*},\\
 q(Q)&=\bar{\lambda}_Q\frac{\int_e \omega}{\bar{e}}+\lambda_Q\frac{\int_{e^*} \omega}{\bar{e}^*}.
\end{align*}
Here, $e$ is an edge of $X$ parallel to a black edge of $\Gamma$, and $e^*$ corresponds to a white edge of $\Gamma^*$.
\end{lemma}
\begin{proof}
Since $\omega$ is of type $\Diamond$, a representation $\omega=p dz+q d\bar{z}$ exists for any quadrilateral $Q$. In fact, given $\omega$, we have to solve a nondegenerate system of two linear equations in the variables $p$ and $q$. Thus, $p,q$ are uniquely defined on $V(\Diamond)$.

Furthermore, we can find for any quadrilateral $Q$ a function $f$ that is defined on the vertices of $Q$ such that $2\int_e \omega =f(b_+)-f(b_-)$ and $2\int_{e^*} \omega =f(w_+)-f(w_-)$, where $e$ is one of the two oriented edges of $X$ going from the midpoint of $b_-$ and $w_\pm$ to the midpoint of $b_+$ and $w_\pm$, and $e^*$ is one of the two edges connecting the midpoint of $w_-$ and $b_\pm$ with the midpoint of $w_+$ and $b_\pm$. By discrete Stokes' Theorem~\ref{th:stokes}, we locally get $\omega=df=p dz +qd\bar{z}$ with $p=\partial_\Lambda f$ and $q=\bar{\partial}_\Lambda f$. Replacing the differences of $f$ by discrete integrals of $\omega$ yields the desired result. 
\end{proof}

\begin{definition}
 Let $\omega=p dz+ q d\bar{z}$ and $\omega'=p' dz+ q' d\bar{z}$ be two discrete one-forms of type $\Diamond$, where the functions $p,p',q,q':V(\Diamond) \to \mC$ are given by Lemma~\ref{lem:representation}. Then, the \textit{discrete wedge product} $\omega\wedge\omega'$ is defined as the discrete two-form of type $\Diamond$ that equals \[\left(pq'-qp'\right)\Omega_\Diamond\] on faces corresponding to $V(\Diamond)$.
\end{definition}

Note that if one considers $dz$ and $d\bar{z}$ as discrete one-forms of type $\Diamond$, \[\Omega_\Diamond=dz \wedge d\bar{z}.\]

\begin{proposition}\label{prop:wedge_Mercat}
Let $F$ be a face of $X$ corresponding to $Q\in V(\Diamond)$, and let $e,e^*$ be oriented edges of $X$ parallel to the black and white diagonal of $Q$, respectively, such that $\im \left(e^*/e\right)>0$. Then, \[\iint\limits_F \omega\wedge\omega' =  2\int\limits_e \omega \int\limits_{e^*} \omega'- 2\int\limits_{e^*} \omega \int\limits_e \omega'.\]
\end{proposition}
\begin{proof}
Both sides of the equation are bilinear and antisymmetric in $\omega,\omega'$. Hence, it suffices to check the identity for $\omega=dz$, $\omega'=d\bar{z}$. On the left hand side, we get $\iint_F \omega\wedge\omega'=-2i\textnormal{ar}(F)$. This equals the right hand side \[2e\bar{e}^*-2e^*\bar{e}=4i\im(e\bar{e}^*)=-i|2e||2e^*|\sin(\varphi_Q)=-2i\textnormal{ar}(F).\]
\end{proof}

\begin{remark}
Since the complex numbers $e$ and $e^*$ are just half of the oriented diagonals, we recover the definition of the discrete wedge product given by Mercat in \cite{Me01,Me07,Me08}.
\end{remark}

The discrete exterior derivative is a derivation for the discrete wedge product if one considers functions on $\Lambda$ and discrete one-forms of type $\Diamond$:

\begin{theorem}\label{th:derivation}
Let $f:V(\Lambda) \to \mC$ and let $\omega$ be a discrete one-form of type $\Diamond$. Then, \[d(f\omega)=df\wedge\omega+fd\omega.\]
\end{theorem}
\begin{proof}
 Let $\omega=pdz+qd\bar{z}$ with $p,q:V(\Diamond)\to\mC$ given by Lemma~\ref{lem:representation}. If $F_v$ and $F_Q$ are faces of $X$ corresponding to a vertex $v$ and a face $Q$ of $\Lambda$,
\begin{align*}
 d(f\omega)|_{F_v}&=\left(f(v)\left(\partial_\Diamond q\right)(v) - f(v)\left(\bar{\partial}_\Diamond p\right)(v)\right)  \Omega_\Lambda=fd\omega|_{F_v},\\
 d(f\omega)|_{F_Q}&=\left(q(Q)\left(\partial_\Lambda f\right)(Q) - p(Q)\left(\bar{\partial}_\Lambda f\right)(Q)\right) \Omega_\Diamond=(df\wedge\omega)|_{F_Q}.
\end{align*}
But $(df\wedge\omega)|_{F_v}=0$ and $fd\omega|_{F_Q}=0$, so $d(f\omega)=df\wedge\omega+fd\omega$.
\end{proof}

\begin{remark}
In \cite{Me01}, Mercat proved an analog of Theorem~\ref{th:derivation} in a setting where discrete one-forms are defined on edges of $\Lambda$. He also considered discrete differential forms as we do, however, the claim $d(f\omega)=df\wedge\omega+fd\omega$ could not be well defined in his setting. At first, our approach using the medial graph allows to make sense out of this statement. It turns out that Theorem~\ref{th:derivation} is a very powerful theorem leading to discretizations of Green's identities in Section~\ref{sec:definition} and of a Cauchy's integral formula for the discrete derivative of a discrete holomorphic function in Section~\ref{sec:Cauchy}.
\end{remark}

Above, we defined a discrete wedge product just of two discrete one-forms of type $\Diamond$. Actually, we could define a discrete wedge product of two discrete one-forms of type $\Lambda$ in essentially the same way, getting a discrete two-form of type $\Lambda$. Then, the analog of Theorem~\ref{th:derivation} would be true for this kind of discrete wedge product and functions on $V(\Diamond)$. Also the discrete Hodge star of a discrete one-form in the next section could be defined not only for those of type $\Diamond$. However, there exist no analogs of Proposition~\ref{prop:wedge_Mercat} and ~\ref{prop:hodge_Mercat}. These propositions imply that the discrete wedge product as well as the discrete Hodge star of discrete one-forms of type $\Diamond$ can be defined in a chart-independent way. This enables one to consider these objects on discrete Riemann surfaces, see the thesis of the second author \cite{Gue14}. There are no such statements if one chooses discrete one-forms of type $\Lambda$. In fact, a discrete one-form of type $\Lambda$ cannot be canonically defined on a discrete Riemann surface as opposed to discrete one-forms of type $\Diamond$. So since our interest lies in the latter, we do not define a discrete wedge product or a discrete Hodge star for discrete one-forms of type $\Lambda$.


\subsubsection{Discrete Hodge star}\label{sec:hodge}

\begin{definition}
 Let $f:F(\Lambda)\to\mC$, $h:V(\Diamond)\to\mC$, $\omega=p dz +qd\bar{z}$ a discrete one-form of type $\Diamond$ with complex functions $p,q:V(\Diamond)\to\mC$ given by Lemma~\ref{lem:representation}, and $\Omega_1,\Omega_2$ discrete two-forms of type $\Lambda$ and $\Diamond$. The \textit{discrete Hodge star} is given by \[ \star f:= -\frac{1}{2i}f \Omega_\Lambda; \quad \star h:= -\frac{1}{2i}h \Omega_\Diamond; \quad \star \omega:=-ip dz+iq d\bar{z} ; \quad \star \Omega_1:=-2i\frac{\Omega_1}{\Omega_\Lambda}; \quad \star \Omega_2:=-2i\frac{\Omega_2}{\Omega_\Diamond}.\]
If $\omega$ and $\omega^\prime$ are both discrete one-forms of type $\Diamond$, we define their \textit{discrete scalar product} \[\langle \omega, \omega^\prime \rangle:=\iint\limits_{F(X)} \omega \wedge \star\bar{\omega}^\prime,\] whenever the right hand side converges absolutely. Similarly, a discrete scalar product for discrete two-forms of the same type is defined.
\end{definition}

\begin{corollary}\label{cor:hodge}
\begin{enumerate}
 \item $\star^2=-\textnormal{Id}$ on discrete one-forms of type $\Diamond$.
 \item $\star^2=\textnormal{Id}$ on complex functions on $V(\Lambda)$ or $V(\Diamond)$ and discrete two-forms of type $\Lambda$ or $\Diamond$.
 \item $f:V(\Lambda)\to\mC$ is discrete holomorphic if and only if $\star df=-idf$.
 \item $\langle f_1,f_2 \rangle = \iint_{F(X)} f_1 \overline{\star f_2}$ and $\langle h_1,h_2 \rangle = \iint_{F(X)} h_1 \overline{\star h_2}$ for functions $f_1,f_2:V(\Lambda)\to\mC$ and functions $h_1,h_2:V(\Diamond)\to\mC$.
 \item $\langle \cdot,\cdot \rangle$ is a Hermitian scalar product on discrete differential forms (of type $\Lambda$ or of type $\Diamond$).
\end{enumerate}
\end{corollary}

\begin{proposition}\label{prop:hodge_Mercat}
Let $\omega$ be a discrete one-form of type $\Diamond$, let $Q\in V(\Diamond)$, and let $e,e^*$ be oriented edges of $X$ parallel to the black and white diagonal of $Q$, respectively, such that $\im \left(e^*/e\right)>0$. Then, \begin{align*}
 \int\limits_e \star\omega&=\cot\left(\varphi_Q\right) \int\limits_e \omega-\frac{|e|}{|e^*| \sin\left(\varphi_Q\right)}\int\limits_{e^*}\omega,\\
 \int\limits_{e^*} \star\omega&=\frac{|e^*|}{|e| \sin\left(\varphi_Q\right)} \int\limits_e \omega-\cot\left(\varphi_Q\right)\int\limits_{e^*}\omega.
\end{align*}
\end{proposition}
\begin{proof}
Both sides of any of the two equations are linear and behave the same under complex conjugation. Thus, it suffices to check the statement for $\omega=dz$. Hence, it remains to show that \begin{align*}
 -ie&=\cot\left(\varphi_Q\right) e-\frac{|e|}{|e^*| \sin\left(\varphi_Q\right)}e^*,\\
e^*&=\frac{|e^*|}{|e| \sin\left(\varphi_Q\right)} e-\cot\left(\varphi_Q\right)e^*.
\end{align*}
Now, both sides of the first equation behave the same under scaling and simultaneous rotation of $e$ and $e^*$, the same statement is true for the second equation. Thus, we may assume $e=1$ and $e^*=\cos\left(\varphi_Q\right)+i\sin(\varphi_Q)$. Multiplying both equations by $\sin(\varphi_Q)$ gives the equivalent statements \begin{align*}
 -i\sin(\varphi_Q)&=\cos\left(\varphi_Q\right)-\left(\cos\left(\varphi_Q\right)+i\sin(\varphi_Q)\right),\\
-i\sin(\varphi_Q)\exp(i\varphi_Q)&=1-\cos\left(\varphi_Q\right)\exp(i\varphi_Q).
\end{align*}
Both equations are true, noting for the second that $\cos\left(\varphi_Q\right)-i\sin(\varphi_Q)=\exp(-i\varphi_Q)$.
\end{proof}

\begin{remark}
Proposition~\ref{prop:hodge_Mercat} shows that our definition of a discrete Hodge star on discrete one-forms coincides with Mercat's definition given in \cite{Me08}. But on discrete two-forms and complex functions, our definition of the discrete Hodge star includes an additional factor of the area of the corresponding face of $X$.
\end{remark}

\begin{proposition}\label{prop:adjoint2}
$\delta:=-\star d \star$ is the \textit{formal adjoint} of the discrete exterior derivative $d$: Let $f:V(\Lambda)\to\mC$, and let $\omega$ be a discrete one-form of type $\Diamond$ and $\Omega$ a discrete two-form of type $\Lambda$. Assume that all of them are compactly supported. Then, \[\langle df, \omega \rangle =\langle f, \delta \omega \rangle \textnormal{ and }\langle d\omega,\Omega\rangle= \langle \omega, \delta \Omega\rangle.\]
\end{proposition}
\begin{proof}
By discrete Stokes' Theorem~\ref{th:stokes}, Theorem~\ref{th:derivation}, and Corollary~\ref{cor:hodge}~(ii), \begin{align*}0&=\iint\limits_{F(X)} d(f\star \bar{\omega})=\iint\limits_{F(X)} df\wedge\star \bar{\omega}+\iint\limits_{F(X)} fd\star \bar{\omega} = \langle df, \omega \rangle + \langle f, \star d \star \omega \rangle,\\
0&=\iint\limits_{F(X)} d(\star \bar{\Omega} \omega)=\iint\limits_{F(X)} (d\star \bar{\Omega})\wedge\omega+\iint\limits_{F(X)} \star \bar{\Omega}d\omega = \langle \omega, \star d \star \Omega \rangle + \langle d\omega, \Omega \rangle.
\end{align*}
\end{proof}


\subsection{Discrete Laplacian} \label{sec:Laplacian}

The discrete Laplacian and the discrete Dirichlet energy on general quad-graphs were first introduced by Mercat in \cite{Me08}. Later, Skopenkov reintroduced these definitions in \cite{Sk13}, taking the same definition in a different notation. In our discussion of the discrete Laplacian in Section~\ref{sec:definition}, we follow the classical approach of Mercat (up to sign) and adapt it to our notations. A feature of our notation is that we are able to formulate a discrete analog of Green's first identity and that our formulation of discrete Green's second identity is more intuitive than the previous ones of Mercat, Chelkak and Smirnov, and Skopenkov \cite{Me08,ChSm11,Sk13}.

In Section~\ref{sec:Dirichlet}, we investigate the discrete Dirichlet energy. In particular, we show how uniqueness and existence of solutions to the discrete Dirichlet boundary value problem imply surjectivity of the discrete differentials and the discrete Laplacian in Theorem~\ref{th:plane_surjective}. We conclude this section with a result concerning the asymptotics of discrete harmonic functions.


\subsubsection{Definition and basic properties} \label{sec:definition}

\begin{definition}
The \textit{discrete Laplacian} on functions $f:V(\Lambda)\to\mC$, discrete one-forms of type $\Diamond$, or discrete two-forms of type $\Lambda$ is defined as the linear operator \[\triangle:=-\delta d-d\delta=\star d \star d +d \star d \star.\]
$f$ is said to be \textit{discrete harmonic} at $v\in V(\Lambda)$ if $\triangle f(v)=0$. 
\end{definition}

The following factorization of the discrete Laplacian in terms of discrete derivatives generalizes the corresponding result given by Chelkak and Smirnov in \cite{ChSm11} to general quad-graphs. The local representation of $\triangle f$ at $v\in V(\Lambda)$ is, up to a factor involving the area of the face $F_v$ of $X$ corresponding to $v$, the same as Mercat's in \cite{Me08}.

\begin{corollary}\label{cor:factorization}
Let $f:V(\Lambda)\to\mC$. Then, $\triangle f=4\partial_\Diamond\bar{\partial}_\Lambda f=4\bar{\partial}_\Diamond\partial_\Lambda f$.  At a vertex $v$ of $\Lambda$, $\triangle f(v)$ equals
\[\frac{1}{2\textnormal{ar}(F_v)}\sum\limits_{Q_s\sim v}\frac{1}{\re\left(\rho_s\right)}\left(|\rho_s|^2 \left(f(v_s)-f(v)\right)+\im\left(\rho_s\right)\left(f(v'_s)-f(v'_{s-1})\right)\right).\]
Here, $\rho_s:=\rho_{Q_s}$ if $v$ is black, and $\rho_s:=1/\rho_{Q_s}$ if $v$ is white. 

In particular, $\re(\triangle f)\equiv\triangle \re(f)$ and $\im(\triangle f)\equiv \triangle \im(f)$.
\end{corollary}
\begin{proof}
The first statement follows from $\triangle f = \star d \star d f= 2 \partial_\Diamond\bar{\partial}_\Lambda f+2\bar{\partial}_\Diamond\partial_\Lambda f$, noting that $\partial_\Diamond\bar{\partial}_\Lambda f\equiv\bar{\partial}_\Diamond\partial_\Lambda f$ by Corollary~\ref{cor:commutativity}.

For the second statement, let us assume without loss of generality that $v\in V(\Gamma)$. Then, we have to show that $\triangle f(v)$ equals \[\frac{1}{2\textnormal{ar}(F_v)}\sum\limits_{Q_s\sim v}\left(\frac{|\rho_{Q_s}|}{\sin(\varphi_{Q_s})} \left(f(v_s)-f(v)\right)-\cot(\varphi_{Q_s})\left(f(v'_s)-f(v'_{s-1})\right)\right).\]
The structure is similar to the formula of the discrete Hodge star in Proposition~\ref{prop:hodge_Mercat}. Indeed, if $e_s$ denotes an edge of $X$ parallel to the black diagonal $vv_s$ and $e^*_s$ an edge parallel to the dual diagonal, $\triangle f(v)$ equals \begin{align*}&\frac{1}{\textnormal{ar}(F_v)}\iint\limits_{F_v} d \star df=\frac{1}{\textnormal{ar}(F_v)}\oint\limits_{\partial F_v} \star df\\ =& \frac{1}{\textnormal{ar}(F_v)}\sum\limits_{Q_s\sim v}\left(\frac{|e^*_s|}{|e_s|\sin(\varphi_{Q_s})} \int\limits_{e_s} df-\cot(\varphi_{Q_s})\int\limits_{e^*_s} df\right)\\=&\frac{1}{2\textnormal{ar}(F_v)}\sum\limits_{Q_s\sim v}\left(\frac{|\rho_{Q_s}|}{\sin(\varphi_{Q_s})} \left(f(v_s)-f(v)\right)-\cot(\varphi_{Q_s})\left(f(v'_s)-f(v'_{s-1})\right)\right),\end{align*}
using discrete Stokes' Theorem~\ref{th:stokes}, Proposition~\ref{prop:hodge_Mercat}, and $|\rho_{Q_s}|=|e^*_s|/|e_s|$.
\end{proof}

\begin{remark}
In the case that the diagonals of the quadrilaterals are orthogonal to each other, $\rho_Q$ is always a positive real number. In this case, the discrete Laplacian splits into two separate discrete Laplacians on $\Gamma$ and $\Gamma^*$. In this case, it is known and actually an immediate consequence of the local representation in Corollary~\ref{cor:factorization} that a discrete maximum principle holds true, i.e., a discrete harmonic function can attain its maximum only at the boundary of a region. This is not true for general quad-graphs, see for example Skopenkov's paper \cite{Sk13}. 
\end{remark}

\begin{corollary}\label{cor:holomorphic_harmonic}
Let $f:V(\Lambda)\to\mC$.
\begin{enumerate}
 \item If $f$ is discrete harmonic, $\partial_\Lambda f$ is discrete holomorphic. 
 \item If $f$ is discrete holomorphic, $f$, $\re f$, and $\im f$ are discrete harmonic.
\end{enumerate}
\end{corollary}
\begin{proof}
By Corollary~\ref{cor:factorization}, $\triangle f \equiv 4 \bar{\partial}_\Diamond \partial_\Lambda f\equiv 4 \partial_\Diamond \bar{\partial}_\Lambda f$. In particular, $\bar{\partial}_\Diamond \partial_\Lambda f \equiv 0$ if $\triangle f \equiv 0$, which shows (i). Also, $f$ is discrete harmonic if it is discrete holomorphic. Using $\re(\triangle f)\equiv\triangle \re(f)$ and $\im(\triangle f)\equiv \triangle \im(f)$, $\re(f)$ and $\im(f)$ are discrete harmonic if $f$ is.
\end{proof}

Similar to Proposition~\ref{prop:examples}, the discrete Laplacian coincides with the smooth one up to order one in the general case and up to order two for parallelogram-graphs. This was already shown by Skopenkov in \cite{Sk13}. Since this result follows immediately from our previous ones, we give a proof here as well.

\begin{proposition}\label{prop:examples3}
Let $f_{\mC}:\mC\to\mC$ and $f$ its restriction to $V(\Lambda)$.
\begin{enumerate}
\item If $f_{\mC}(z)$ is a polynomial in $\re(z)$ and $\im(z)$ of degree at most one, then the smooth and the discrete Laplacian coincide on vertices: $\triangle_{\mC} f_{\mC} (v)=\triangle f (v)$.
\item Let all quadrilaterals of $\Diamond$ be parallelograms. If $f_{\mC}(z)$ is a polynomial in $\re(z)$ and $\im(z)$ of degree at most two, then the smooth and the discrete Laplacian coincide on vertices: $\triangle_{\mC} f_{\mC} (v)=\triangle f (v)$.
\end{enumerate}
\end{proposition}
\begin{proof}
(i) It follows from Proposition~\ref{prop:examples}~(ii) and Corollary~\ref{cor:holomorphic_harmonic}~(ii) that constants as well as the complex function $f(v)=v$ are discrete harmonic since they are discrete holomorphic. Similarly, $f(v)=\overline{v}$ is discrete harmonic.

(ii) In the parallelogram case, let us place the vertices of $\Diamond$ at the centers of the parallelograms. Then, $f(v)=v^2$ is discrete harmonic by Proposition~\ref{prop:examples}~(iii) and Corollary~\ref{cor:holomorphic_harmonic}~(ii). Looking at real and imaginary part separately, $\triangle f_1^2 = \triangle f_2^2$ and $\triangle \left(f_1 f_2\right)=0$ with $f_1(v)=\re(v), f_2(v)=\im(v)$. Finally, \[\triangle |f|^2= 4 \partial_\Diamond\bar{\partial}_\Lambda |f|^2=4 \partial_\Diamond h=4\] with $h(Q)=Q$, due to Propositions~\ref{prop:examples}~(iv) and~\ref{prop:examples2}. Since any polynomial in $\re(z)$ and $\im(z)$ of monomials of degree two is a linear combination of $f_1^2-f_2^2$, $f_1^2+f_2^2$, and $f_1f_2$, and since we have shown that the discrete Laplacian $\triangle$ and the smooth Laplacian $\triangle_\mC$ coincide on these, we are done.
\end{proof}
\begin{remark}
The second part of the last proposition generalizes the known result for rhombi given by Chelkak and Smirnov in \cite{ChSm11}. Note that this is not true for general quadrilaterals even if one assumes that the diagonals of quadrilaterals are orthogonal to each other. For this, consider the following (finite) bipartite quad-graph of Figure~\ref{fig:laplace}: $0\in\Gamma$ is adjacent to the white vertices $\pm 1$ and $\pm i$ in $\Lambda$, and adjacent to the black vertices $2+2i$, $-1\pm i$, and $1-i$ in $\Gamma$. There are no further vertices. Then, $\triangle f (0)\neq 0$ for $f(v)=v^2$. Indeed, we would get $\triangle f (0)= 0$ if we had replaced $v=2+2i$ by $v=1+i$ obtaining a rhombic quad-graph; but $|\rho_Q|^2/\re(\rho_Q)\left(f(v)-f(0)\right)$ scales by a factor of $2$, whereas the other nonzero summands in the local representation of $\triangle f (0)$ remain invariant. 

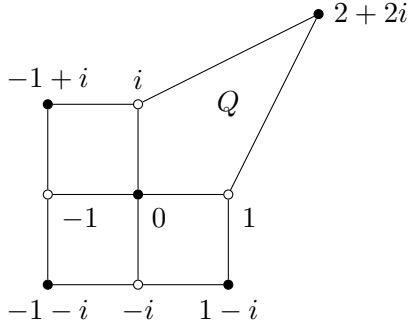
\begin{figure}[htbp]
\begin{center}
\beginpgfgraphicnamed{laplace}
\begin{tikzpicture}
[white/.style={circle,draw=black,fill=white,thin,inner sep=0pt,minimum size=1.2mm},
black/.style={circle,draw=black,fill=black,thin,inner sep=0pt,minimum size=1.2mm},
scale=1.2]
\node[black] (b1) [label=below right:$0$]
 at (0,0) {};
\node[black] (b2) [label=right:$2+2i$]
 at (2,2) {};
\node[black] (b3) [label=below:$1-i$]
 at (1,-1) {};
\node[black] (b4) [label=above:$-1+i$]
 at (-1,1) {};
\node[black] (b5) [label=below:$-1-i$]
 at (-1,-1) {};
\node[white] (w1) [label=below right:$1$]
 at (1,0) {};
\node[white] (w2) [label=above:$i$]
 at (0,1) {};
\node[white] (w3) [label=below:$-i$]
 at (0,-1) {};
\node[white] (w4) [label=below right:$-1$]
 at (-1,0) {};
\coordinate[label=center:$Q$] (z)  at (1,1) {};
\draw (b1) -- (w3) -- (b3) -- (w1) -- (b1) -- (w2) -- (b2) -- (w1);
\draw (b1) -- (w4) --  (b5) --  (w3);
\draw (w4) -- (b4) -- (w2);
\end{tikzpicture}
\endpgfgraphicnamed
\caption{$\triangle f (0)\neq 0$ for $f(v)=v^2$}
\label{fig:laplace}
\end{center}
\end{figure}

In the case of general quad-graphs, smooth functions $f_\mC:\mC \to \mC$, and restrictions $f$ to $V(\Lambda)$, Skopenkov compared the integral of $\triangle_\mC f_\mC$ over a square domain $R$ and a sum of $\triangle f (v)$ over black vertices of $\Lambda$ in $R$ \cite{Sk13}. Moreover, he showed that for $f(v)=|v|^2$, \[\triangle f(v)=\frac{2}{ \textnormal{ar}(F_v)} \sum\limits_{Q_s \sim v} \textnormal{area}(vv'_{s-1}Q_sv'_s)\] when $Q_s$ is placed at the intersection point of the middle perpendiculars to the diagonals of the corresponding quadrilateral (which equals the intersection point of the diagonals if the quadrilateral is a parallelogram). Note that in general, $h(Q)=Q$ is not discrete holomorphic if the vertices $Q \in V(\Diamond)$ are placed on the intersections of these middle perpendiculars.
\end{remark}

For a finite subset $\Diamond_0 \subset \Diamond$ and two functions $f,g:V(\Lambda_0)\to\mC$, we denote by \[\langle f,g\rangle_{\Diamond_0}:=-\frac{1}{2i}\iint\limits_{F(X_0)} f \bar{g} \Omega_\Lambda\] the discrete scalar product of $f$ and $g$ restricted to $\Diamond_0$. Similarly, the restriction of the discrete scalar product of two discrete one-forms of type $\Diamond$ on $\Diamond_0$ is defined.

In the rhombic setup, discrete versions of Green's second identity were already stated by Mercat \cite{Me01}, whose integrals were not well defined separately, and Chelkak and Smirnov \cite{ChSm11}, whose boundary integral was an explicit sum involving boundary angles. Skopenkov formulated a discrete Green's second identity with a vanishing boundary term \cite{Sk13}. We are able to provide a \textit{discrete Green's first identity} for the first time, prove it completely analogously to the smooth setting, and deduce \textit{discrete Green's second identity} out of it.

\begin{theorem}\label{th:Green_identities}
Let $\Diamond_0 \subset \Diamond$ be finite, and let $f,g:V(\Lambda_0)\to\mC$.
\begin{enumerate}
 \item $\langle f,\triangle g\rangle_{\Diamond_0}+ \langle df, dg\rangle_{\Diamond_0}=\oint\limits_{\partial X_0} f \star d\bar{g}.$
 \item $\langle \triangle f,g\rangle_{\Diamond_0}-\langle f,\triangle g\rangle_{\Diamond_0}=\oint\limits_{\partial X_0} \left(f \star d\bar{g} - \bar{g} \star df\right).$
\end{enumerate}
\end{theorem}
\begin{proof}
(i) Since the discrete exterior derivative is a derivation for the discrete wedge product by Theorem~\ref{th:derivation}, $d\left(f \star d\bar{g}\right)=df\wedge\star d\bar{g}+f \star (\star d \star d \bar{g})$. Now, discrete Stokes' Theorem~\ref{th:stokes} yields the desired result.

(ii) Just apply twice discrete Green's first identity.
\end{proof}

\textit{Discrete Weyl's lemma} is a direct consequence of discrete Green's second identity, Theorem~\ref{th:Green_identities}~(ii). A version for rhombic quad-graphs was given by Mercat in \cite{Me01}, proven by an explicit calculation.

\begin{corollary}\label{cor:Weyl}
 $f:V(\Lambda)\to\mC$ is discrete harmonic if and only if $\langle f, \triangle g \rangle =0$ for every compactly supported $g:V(\Lambda)\to\mC$.
\end{corollary}

Skopenkov introduced the notion of discrete harmonic conjugates in \cite{Sk13}. We recover his definitions in our notation, observing that his discrete gradient corresponds to our discrete exterior derivative and his counterclockwise rotation by $\pi/2$ corresponds to our discrete Hodge star.

\begin{definition}
Let $\Diamond_0 \subseteq \Diamond$ and $f$ be a real (discrete harmonic) function on $V(\Lambda_0)$. A real discrete harmonic function $\tilde{f}$ on $V(\Lambda_0)$ is said to be a \textit{discrete harmonic conjugate of} $f$ if $f+i\tilde{f}$ is discrete holomorphic.
\end{definition}

Note that the existence of a real function $\tilde{f}$ such that $f+i\tilde{f}$ is discrete holomorphic requires already that $f$ is discrete harmonic due to Corollary~\ref{cor:holomorphic_harmonic}~(ii).

\begin{lemma}\label{lem:conjugate}
Let $\Diamond_0 \subseteq \Diamond$ and $f$ be a real discrete harmonic function on $V(\Lambda_0)$.
\begin{enumerate}
 \item The discrete harmonic conjugate $\tilde{f}$ is unique up to two additive real constants on $\Gamma_0$ and $\Gamma_0^*$.
 \item If $\Diamond_0$ is simply-connected, a discrete harmonic conjugate $\tilde{f}$ exists.
\end{enumerate}
\end{lemma}
\begin{proof}
(i) If $\tilde{f}_1$ and $\tilde{f}_2$ are two real discrete harmonic conjugates, their difference $\tilde{f}_1-\tilde{f}_2$ is real and discrete holomorphic. So by Proposition~\ref{prop:zeroderivative}~(ii), it is biconstant.

(ii) Since $f$ is harmonic, $d\star d f=0$, i.e., $\star df$ is closed and of type $\Diamond$. Moreover, reality of $f$ implies $\star df=-i \partial_\Lambda f dz + i\bar{\partial}_\Lambda f d\bar{z}= 2\im\left(\partial_\Lambda f dz \right)$. So in the same manner as in the proof of Proposition~\ref{prop:primitive}, $\star df$ can be integrated to a real function $\tilde{f}$ on $V(\Lambda_0)$. Finally, $f+i\tilde{f}$ is discrete holomorphic by Corollary~\ref{cor:f_holomorphic} since $d(f+i\tilde{f})=df+i\star df=2\re\left(\partial_\Lambda f dz\right)+2i\im\left(\partial_\Lambda f dz \right)=2\partial_\Lambda f dz$.
\end{proof}

Note that in the case of quadrilaterals with orthogonal diagonals, such that $\triangle$ splits into two discrete Laplacians on $\Gamma$ and $\Gamma^*$, it follows that a discrete harmonic conjugate of a discrete harmonic function on $V(\Gamma)$ can be defined on $V(\Gamma^*)$ and vice versa, as was already noted by Chelkak and Smirnov in \cite{ChSm11}.

\begin{corollary}\label{cor:holomorphic_realimaginary}
Let $\Diamond_0\subseteq\Diamond$ and $f:V(\Lambda_0)\rightarrow\mC$ be discrete holomorphic. Then, $\im(f)$ is uniquely determined by $\re(f)$ up to two additive real constants on $\Gamma_0$ and $\Gamma_0^*$.
\end{corollary}


\subsubsection{Discrete Dirichlet energy} \label{sec:Dirichlet}

We follow the classical approach of discretizing the Dirichlet energy introduced by Mercat in \cite{Me08}. Note that Skopenkov's definition in \cite{Sk13} is exactly the same. In particular, Skopenkov's results, including an approximation property of the Laplacian, convergence of the discrete Dirichlet energy to the smooth Dirichlet energy for nondegenerate uniform sequences of quad-graphs, and further theorems for quad-graphs with orthogonal diagonals apply as well in our setting. We refer to his work \cite{Sk13} for details on these results.

\begin{definition}
Let $\Diamond_0 \subseteq \Diamond$. For $f:V(\Lambda_0)\to\mC$, we define the \textit{discrete Dirichlet energy} of $f$ on $\Diamond_0$ as $E_{\Diamond_0}(f):=\langle df,df \rangle_{\Diamond_0}\in[0,\infty]$.

If $\Diamond_0$ is finite, the \textit{discrete Dirichlet boundary value problem} asks for a real discrete harmonic function $f$ on $V(\Lambda_0)$ such that $f$ agrees with a preassigned real function $f_0$ on the boundary $V(\partial \Lambda_0)$. 
\end{definition}

\begin{proposition}\label{prop:Dirichlet_Mercat}
\begin{align*}
E_{\Diamond_0}(f)&=\sum\limits_{Q\in V(\Diamond_0)}\frac{1}{2\re\left(\rho_Q\right)}\left(|\rho_Q|^2\left|f(b_+)-f(b_-)\right|^2+\left|f(w_+)-f(w_-)\right|^2\right)\\&+\sum\limits_{Q\in V(\Diamond_0)}\frac{\im\left(\rho_Q\right)}{\re\left(\rho_Q\right)}\re\left(\left(f(b_+)-f(b_-)\right)\overline{\left(f(w_+)-f(w_-)\right)}\right).
\end{align*}
\end{proposition}
\begin{proof}
Let us denote the right hand side by $E'_{\Diamond_0}(f)$. Then, $E_{\Diamond_0}(f)=E_{\Diamond_0}(\re(f))+E_{\Diamond_0}(\im(f))$ and $E'_{\Diamond_0}(f)=E'_{\Diamond_0}(\re(f))+E'_{\Diamond_0}(\im(f))$. Therefore, we can restrict to real functions $f$. Furthermore, it suffices to check the identity for just a singular quadrilateral $Q$, because both $E_{\Diamond_0}(f)$ and $E'_{\Diamond_0}(f)$ are sums over $Q \in V(\Diamond_0)$.

Also, both expressions depend only on $df$, so we may assume $f(b_-)=f(w_-)=0$. In addition, both terms do not change if $Q$ is rotated or scaled. Thus, we can assume $b_+-b_-=1$ as well, such that $w_+-w_-=i\rho_Q$. Then, \begin{align*}
E_Q(f)=&\iint\limits_{F_Q} df \wedge \star df=4\textnormal{area}(Q)|\partial_\Lambda f|^2=2\re(\rho_Q)|\partial_\Lambda f|^2\\
=&2\re(\rho_Q)\left|\lambda_Q f(b_+)+\frac{\bar{\lambda}_Q}{i\rho_Q}f(w_+)\right|^2\\
=&2|\rho_Q|\sin(\varphi_Q)\left(\frac{f^2(b_+)}{4\sin^2(\varphi_Q)}+\frac{ f^2(w_+)}{4\sin^2(\varphi_Q)|\rho_Q|^2}+2f(b_+)f(w_+)\im\left(\frac{\bar{\lambda}_Q^2}{\rho_Q}\right) \right)\\
=&\frac{|\rho_Q|^2}{2\re(\rho_Q)}f^2(b_+)+\frac{1}{2\re(\rho_Q)} f^2(w_+)+2f(b_+)f(w_+)\frac{\im\left(\rho_Q\right)}{2\re\left(\rho_Q\right)}=E'_Q(f).
\end{align*}
Here, we used that $2\textnormal{area}(Q)=|w_+-w_-|\sin(\varphi_Q)=|\rho_Q|\sin(\varphi_Q)=\re(\rho_Q)$, $1/|\lambda_Q|=2\sin(\varphi_Q)$, and $\bar{\lambda}_Q=-\rho_Q/(4\textnormal{area}(Q))$ (see the proof of Lemma~\ref{lem:derivative_lambda}).
\end{proof}

The same formula of $E_{\Diamond_0}(f)$ was given by Mercat in \cite{Me08}.

In the case of rhombic quad-graphs, Duffin proved in \cite{Du68} that the discrete Dirichlet boundary value problem has a unique solution. The same argument applies for general quad-graphs with the discrete Dirichlet energy defined here. Using a different notation, Skopenkov proved existence and uniqueness of solutions of the discrete Dirichlet boundary value problem as well \cite{Sk13}.

\begin{lemma}\label{lem:Dirichlet_boundary}
Let $\Diamond_0 \subset \Diamond$ be finite and $f_0:V(\partial\Lambda_0)\rightarrow\mR$. We consider the vector space of real functions $f:V(\Lambda_0)\to\mR$ that agree with $f_0$ on the boundary.

Then, $E_{\Diamond_0}$ is a strictly convex nonnegative quadratic functional in terms of the interior values $f(v)$. Furthermore, \[-\frac{\partial E_{\Diamond_0}}{\partial f(v)}(f)=2\textnormal{ar}(F_v )\triangle f(v)\] for any $v \in V(\Lambda_0)\backslash V(\partial\Lambda_0)$. In particular, the discrete Dirichlet boundary value problem is uniquely solvable.
\end{lemma}
\begin{proof}
By construction, $E_{\Diamond_0}$ is a quadratic form in the vector space of real functions $f:V(\Lambda)\to\mR$. In particular, it is convex, nonnegative, and quadratic in terms of the values $f(v)$. Thus, global minima exist. To prove strict convexity, it suffices to check that the minimum is unique.

For an interior vertex $v_0 \in V(\Lambda_0)\backslash V(\partial\Lambda_0)$, let $\phi(v):=\delta_{vv_0}$ be the Kronecker delta function on $V(\Lambda_0)$. Then, \[\frac{\partial E_{\Diamond_0}}{\partial f(v_0)}(f)=\frac{d}{dt}E_{\Diamond_0}(f+t\phi)|_{t=0}=2\langle df,d\phi\rangle_{\Diamond_0}=-2 \langle \triangle f, \phi \rangle=-2\textnormal{ar}(F_{v_0} )\triangle f(v_0)\] due to Proposition~\ref{prop:adjoint2}. It follows that minima of $E_{\Diamond_0}$ are discrete harmonic in the interior of $\Diamond_0$. The difference of two minima is a discrete harmonic function vanishing on the boundary, so its energy has to be zero. But only biconstant functions have zero energy. Thus, the difference has to vanish everywhere, i.e., minima are unique.
\end{proof}

In the following, we apply Lemma~\ref{lem:Dirichlet_boundary} to show that $\partial_\Lambda,\bar{\partial}_\Lambda,\partial_\Diamond,\bar{\partial}_\Diamond,\triangle$ are surjective operators. This implies immediately the existence of discrete Green's functions and discrete Cauchy's kernels, as we will see in Sections~\ref{sec:Green} and~\ref{sec:Cauchy}.

\begin{lemma}\label{lem:finite_surjective}
Let $\Diamond_0\subset\Diamond$ be finite and homeomorphic to a disk. Then, the discrete derivatives $\partial_\Lambda, \bar{\partial}_\Lambda,\partial_\Diamond, \bar{\partial}_\Diamond$ and the discrete Laplacian $\triangle$ are surjective operators. That means, given any complex functions $h_0$ on $V(\Diamond_0)$ and $f_0$ on $V(\Lambda_0 \backslash \partial \Lambda_0)$, there exist functions $h_\partial, h_{\bar{\partial}}$ on $V(\Diamond_0)$ and functions $f_\partial, f_{\bar{\partial}}, f_\triangle$ on $V(\Lambda_0)$ such that $\partial_\Diamond h_\partial= \bar{\partial}_\Diamond h_{\bar{\partial}}=\triangle f_\triangle =f_0$ and $\partial_\Lambda f_\partial= \bar{\partial}_\Lambda f_{\bar{\partial}}=h_0$. If $f_0$ is real-valued, $f_\triangle$ can be chosen real-valued as well.
\end{lemma}
\begin{proof}
 Let $V$, $E$, and $F$ denote the number of vertices, edges, and faces of $\Lambda_0$, respectively. Denote by $BV$ and $BE$ the number of vertices and edges of $\partial \Lambda_0$, respectively. By assumption, $\partial \Lambda_0$ is a simple polygon and $BV=BE$.

By Lemma~\ref{lem:Dirichlet_boundary}, the space of real discrete harmonic functions on $V(\Lambda_0)$ has dimension $BV$. Clearly, real and imaginary part of a discrete harmonic function are itself discrete harmonic. Therefore, the complex dimension of the space of complex discrete harmonic functions, i.e., of the kernel of $\triangle$, is $BV$ as well. Thus, we have shown that $\triangle:\mathds{K}^{V(\Lambda_0)} \rightarrow \mathds{K}^{V(\Lambda_0\backslash\partial\Lambda_0)}$ is a surjective linear operator with $\mathds{K}\in\{\mR,\mC\}$.

Now, $\triangle=4\partial_\Diamond\bar{\partial}_\Lambda=4\bar{\partial}_\Diamond\partial_\Lambda$ by Corollary~\ref{cor:factorization}, so $\partial_\Diamond,\bar{\partial}_\Diamond:\mC^{V(\Diamond_0)} \rightarrow \mC^{V(\Lambda_0\backslash\partial\Lambda_0)}$ are surjective as well. The kernel of $\bar{\partial}_\Diamond$ consists of all discrete holomorphic functions on $V(\Diamond_0)$. By Proposition~\ref{prop:primitive}, any such function has a discrete primitive, i.e., the kernel is contained in the image of $\partial_\Lambda$. Using the surjectivity of $\triangle$, it follows that $\partial_\Lambda:\mC^{V(\Lambda_0)} \rightarrow \mC^{V(\Diamond_0)}$ is surjective. The same is true for $\bar{\partial}_\Lambda$.
\end{proof}

\begin{theorem}\label{th:plane_surjective}
The discrete derivatives $\partial_\Lambda,\bar{\partial}_\Lambda,\partial_\Diamond,\bar{\partial}_\Diamond$ and the discrete Laplacian $\triangle$ (defined on complex or real functions) are surjective operators on the vector space of functions on $V(\Lambda)$ or $V(\Diamond)$.
\end{theorem}
\begin{proof}
Let $\Diamond_0 \subset \Diamond_1 \subset \Diamond_2 \subset \ldots \subset\Diamond$ be a sequence of simply-connected finite domains such that $\bigcup_{k=0}^\infty \Diamond_k=\Diamond$. By $\Lambda_k$ we denote the subgraph of $\Lambda$ whose vertices and edges are the vertices and edges of quadrilaterals in $\Diamond_k$.

Let us first prove that any $h:V(\Diamond)\to\mC$ has a preimage under the discrete derivatives $\partial_\Lambda, \bar{\partial}_\Lambda$. By Lemma~\ref{lem:finite_surjective}, the affine space $A^{(0)}_k$ of all complex functions on $V(\Lambda_k)$ that are mapped to $h|_{V(\Diamond_k)}$ by $\partial_\Lambda$ (or $\bar{\partial}_\Lambda$) is nonempty. Let $A^{(0)}_k\Big|_{\Lambda_j}$ denote the affine space of restrictions of these functions to $V(\Lambda_j)\subseteq V(\Lambda_k)$. Clearly, \[A^{(0)}_0\supseteq A^{(0)}_1\Big|_{\Lambda_0}\supseteq A^{(0)}_2\Big|_{\Lambda_0}\supseteq \ldots\]

Since all affine spaces are finite-dimensional and nonempty, this chain becomes stationary at some point, giving a function $f_0$ on $V(\Lambda_0)$ mapped to $h|_{V(\Diamond_0)}$ by $\partial_\Lambda$ (or $\bar{\partial}_\Lambda$) that can be extended to a function in $A^{(0)}_k$ for any $k$.

Inductively, assume that $f_j:V(\Lambda_j)\to\mC$ is mapped to $h|_{V(\Diamond_j)}$ by $\partial_\Lambda$ (or $\bar{\partial}_\Lambda$) and that $f_j$ can be extended to a function in $A^{(j)}_k$ for all $k\geq j$. Let $A^{(j+1)}_k$, $k\geq j+1$, be the affine space of all complex functions on $V(\Lambda_k)$ that are mapped to $h|_{V(\Diamond_k)}$ by $\partial_\Lambda$ (or $\bar{\partial}_\Lambda$) and whose restriction to $V(\Lambda_j)$ is equal to $f_j$. By assumption, all these spaces are nonempty. In the same way as above, there is a function $f_{j+1}$ extending $f_j$ to $V(\Lambda_{j+1})$ that is mapped to $h|_{V(\Diamond_{j+1})}$ by $\partial_\Lambda$ (or $\bar{\partial}_\Lambda$) and that can be extended to a function in $A^{(j+1)}_k$ for all $k\geq j+1$.

For $v \in V(\Lambda_k)$, define $f(v):=f_k(v)$. $f$ is a well defined complex function on $V(\Lambda)$ with $\partial_\Lambda f=h$ (or $\bar{\partial}_\Lambda f=h$). Hence, $\partial_\Lambda,\bar{\partial}_\Lambda:\mC^{V(\Lambda)}\to \mC^{V(\Diamond)}$ are surjective.

Replacing $V(\Diamond_k)$ by $V(\Lambda_k \backslash \partial \Lambda_k)$, we obtain with the same arguments that $\triangle$ is surjective, regardless whether $\triangle$ is defined on real or complex functions.

Finally, $\partial_\Diamond,\bar{\partial}_\Diamond:\mC^{V(\Diamond)}\to\mC^{V(\Lambda)}$ are surjective due to $\triangle=4\partial_\Diamond\bar{\partial}_\Lambda=4\bar{\partial}_\Diamond\partial_\Lambda$ by Corollary~\ref{cor:factorization}.
\end{proof}

In the case of rhombic quad-graphs with bounded interior angles, Kenyon proved the existence of a discrete Green's function and a discrete Cauchy's kernel with asymptotic behaviors similar to the classical setting \cite{Ke02}. But in the general case, it seems to be practically impossible to speak about any asymptotic behavior of certain discrete functions. For this reason, we will not require any asymptotic behavior of discrete Green's functions and discrete Cauchy's kernels in Sections~\ref{sec:Green} and~\ref{sec:Cauchy}.

Still, one can expect results concerning the asymptotics of special discrete functions if the interior angles and the side lengths of the quadrilaterals are bounded, meaning that the quadrilaterals do not degenerate at infinity. And indeed, on such quad-graphs we can show that any discrete harmonic function whose difference functions on $V(\Gamma)$ and $V(\Gamma^*)$ have asymptotics $o(v^{-1/2})$ is biconstant. In the rhombic setting, Chelkak and Smirnov showed that a discrete Liouville's theorem holds true, i.e., any bounded discrete harmonic function on $V(\Lambda)$ vanishes \cite{ChSm11}.

\begin{theorem}\label{th:harmonic_asymptotics}
Assume that there exist constants $\alpha_{0}>0$ and $\infty>E_{1}>E_{0}>0$ such that $\alpha\geq\alpha_{0}$ and $E_{1}\geq e \geq E_{0}$ for all interior angles $\alpha$ and side lengths $e$ of quadrilaterals $Q\in V(\Diamond)$.

If $f:V(\Lambda) \rightarrow \mC$ is discrete harmonic and $f(v_+)-f(v_-)=o(v_{\pm}^{-1/2})$ for any two adjacent $v_\pm \in V(\Gamma)$ or $v_\pm \in V(\Gamma^*)$, $f$ is biconstant.
\end{theorem}
\begin{proof}
Without loss of generality, we can restrict to real functions $f$. Assume that $f$ is not biconstant. Then, $df \wedge \star df$ is nonzero somewhere on a face $F$ of $X$. In particular, the discrete Dirichlet energy of $f$ is bounded away from zero if the domain contains $F$. Now, the idea of proof is to show that if the domain is large enough, the function being zero in the interior and equal to $f$ on the boundary has a smaller discrete Dirichlet energy than $f$, contradicting Lemma~\ref{lem:Dirichlet_boundary}.

Let us first bound the intersection angles and the lengths of diagonals of the quadrilaterals. Take $Q \in V(\Diamond)$. Then, there are two opposite interior angles that are less than $\pi$, say $\alpha_\pm$ at vertices $b_\pm$. Since all interior angles are bounded by $\alpha_0$ from below, one of $\alpha_\pm$ is less than or equal to $\pi-\alpha_0$, say $\alpha_0\leq\alpha_-\leq \pi-\alpha_0$.

By triangle inequality, $|b_+-b_-|,|w_+-w_-|<2E_1$. Twice the area of $Q$ equals \[|w_--b_-||w_+-b_-|\sin(\alpha_-)+|w_--b_+||w_+-b_+|\sin(\alpha_+)\geq E_0^2 \sin(\alpha_0).\] On the other hand, twice the area of $Q$ is equal to $|b_+-b_-||w_+-w_-|\sin(\varphi_Q)$. Combining the two inequalities yields $|b_+-b_-|,|w_+-w_-|>E_0':=E_0^2 \sin(\alpha_0)/(2E_1)$. Furthermore, we have the lower estimate $\sin(\varphi_Q)>E_0^2 \sin(\alpha_0)/(4E_1^2)=E_0'/(2E_1)$. Thus, we can bound
 \[\rho_Q=\frac{|w_+-w_-|}{|b_+-b_-|}\exp\left(i\left(\varphi_Q-\frac{\pi}{2}\right)\right)=\frac{|w_+-w_-|}{|b_+-b_-|}\left(\sin(\varphi_Q)-i\cos(\varphi_Q)\right)\]
\[\textnormal{by }|\rho_Q|<\frac{2E_1}{E_0'} \textnormal{ and } \re\left(\rho_Q\right)>\left(\frac{E_0'}{2E_1}\right)^2.\]

For some $r>0$, denote by $B_{\Diamond}(0,r)\subset V(\Diamond)$ the set of quadrilaterals that have a nonempty intersection with the open ball around $0$ and radius $r$. Let $R>2E_1$, and consider the ball $B_{\Diamond}(0,R)\subset  V(\Diamond)$.

Since edge lengths are bounded by $E_1$, all faces in the boundary of $B_{\Diamond}(0,R)$ are contained in the set $B(0,R+2E_1)\backslash B(0,R-2E_1) \subset \mC$. The area of the latter is $8\pi R E_1$. Any quadrilateral has area at least $E_0^2 \sin(\alpha_0)/2$, so at most $16\pi R E_1/(E_0^2 \sin(\alpha_0))$ quadrilaterals are in the boundary of $B_{\Diamond}(0,R)$.

Consider the real function $f_R$ defined on the vertices of all quadrilaterals in $B_{\Diamond}(0,R)$ that is equal to $f$ at the boundary and equal to 0 in the interior of $B_{\Diamond}(0,R)$. When computing the discrete Dirichlet energy, only boundary faces can give nonzero contributions. If we look at the formula of the discrete Dirichlet energy in Proposition~\ref{prop:Dirichlet_Mercat}, and use that $f(v_+)-f(v_-)=o(R^{-1/2})$ at the boundary, we see that any contribution of a boundary face has asymptotics $o(R^{-1})$. For this, we use that $\left|\re\left(\rho_Q)\right)\right|$ is bounded from below by a constant and $\left|\im\left(\rho_Q\right)\right|\leq\left|\rho_Q\right|<2E_1/E'_0$. Using that there are only $O(R)$ faces in the boundary, the discrete Dirichlet energy $E_{B_{\Diamond}(0,R)}(f_R)$, considered as a function of $R$, behaves as $o(1)$. So if $R$ is large enough, \[E_{B_{\Diamond}(0,R)}(f_R)<\iint\limits_F df \wedge \star df \leq E_{B_{\Diamond}(0,R)}(f),\] contradicting that $f$ minimizes the discrete Dirichlet energy by Lemma~\ref{lem:Dirichlet_boundary}.
\end{proof}


\subsection{Discrete Green's functions} \label{sec:Green}

In this section, we discuss discrete Green's functions.

\begin{definition}
Let $v_0 \in V(\Lambda)$. A real function $G(\cdot;v_0)$ on $V(\Lambda)$ is a \textit{(free) discrete Green's function} for $v_0$ if \[G(v_0;v_0)=0\textnormal{ and }\triangle G(v;v_0)=\frac{1}{2\textnormal{ar}(F_{v_0})}\delta_{vv_0} \textnormal{ for all }v \in V(\Lambda).\]
\end{definition}

\begin{corollary}\label{cor:free_Green_existence}
A discrete Green's function exists for any $v_0 \in V(\Lambda)$.
\end{corollary}
\begin{proof}
By Theorem~\ref{th:plane_surjective}, there exists $G:V(\Lambda)\to\mR$ such that $\triangle G(v)=\delta_{vv_0}/\left(4\textnormal{ar}(F_{v_0})\right)$. Since constant functions are discrete harmonic, we can adjust $G$ to get $G(v_0)=0$.
\end{proof}

\begin{remark}
In contrast to the smooth setting or the rhombic case investigated by Kenyon \cite{Ke02} and Chelkak and Smirnov \cite{ChSm11}, we do not require any asymptotic behavior of the discrete Green's function. But when considering planar parallelogram-graphs with bounded interior angles and bounded ratio of side lengths in Section~\ref{sec:Green_asymptotics} of Section~\ref{sec:parallel}, we will prove the existence of a discrete Green's function with asymptotics generalizing the corresponding result for rhombic quad-graphs.
\end{remark}

Our notion of discrete Greens' functions in a discrete domain and the proof of their existence follow the presentation of Chelkak and Smirnov in \cite{ChSm11}.

\begin{definition}
Let $\Diamond_0 \subset \Diamond$ be finite. For $v_0 \in V(\Lambda_0 \backslash \partial \Lambda_0)$, a real function $G_{\Lambda_0}(\cdot;v_0)$ on $V(\Lambda_0)$ is a \textit{discrete Green's function in $\Lambda_0$} for $v_0$ if \begin{align*}G_{\Lambda_0}(v;v_0)&= 0\textnormal{ for all }v\in V(\partial \Lambda_0)\\ \textnormal{ and }\triangle G_{\Lambda_0}(v;v_0)&=\frac{1}{4\textnormal{ar}(F_{v_0})}\delta_{vv_0} \textnormal{ for all }v\in V(\Lambda_0 \backslash \partial \Lambda_0).\end{align*}
\end{definition}

\begin{corollary}\label{cor:boundary_Green_exitence}
Let $\Diamond_0\subset\Diamond$ be finite and $v_0 \in V(\Lambda_0 \backslash \partial \Lambda_0)$. Then, there exists a unique discrete Green's function in $\Lambda_0$ for $v_0$.
\end{corollary}
\begin{proof}
Let $G_0:V(\Lambda)\to\mC$ be the free discrete Green's function for $v_0$ given by Corollary~\ref{cor:free_Green_existence}. By Lemma~\ref{lem:Dirichlet_boundary}, there exists a real discrete harmonic function $G_1$ on $V(\Lambda_0)$ such that $G_1$ and $G_0$ coincide on $V(\partial \Lambda_0)$. Thus, $G_{\Lambda_0}(\cdot;v_0):=G_0-G_1$ is a discrete Green's function in $\Lambda_0$ for $v_0$. Since the difference of two discrete Green's functions in $\Lambda_0$ for $v_0$ is discrete harmonic on $V(\Lambda_0)$ and equals zero on the boundary $V(\partial \Lambda_0)$, it has to be identically zero by Lemma~\ref{lem:Dirichlet_boundary}.
\end{proof}


\subsection{Discrete Cauchy's integral formulae}\label{sec:Cauchy}

In this section, we first formulate discretizations of the standard Cauchy's integral formula, both for discrete holomorphic functions on $V(\Lambda)$ and $V(\Diamond)$. Later, we give with Theorem~\ref{th:Cauchy_formula_derivative} a discrete formulation of Cauchy's integral formula for the derivative of a holomorphic function. We conclude this part with Section~\ref{sec:Cauchy_CS}, where we relate our formulation of the discrete Cauchy's integral formula applied to a discrete holomorphic function on $V(\Lambda)$ with the notation used by Chelkak and Smirnov in \cite{ChSm11}.

\begin{definition}
\textit{Discrete Cauchy's kernels} \textit{with respect to} $Q_0 \in V(\Diamond)$ \textit{and} $v_0\in V(\Lambda)$ are functions $K_{Q_0}:V(\Lambda) \rightarrow \mC$ and $K_{v_0}:V(\Diamond) \rightarrow \mC$, respectively, that satisfy for all $Q \in V(\Diamond)$, $v \in V(\Lambda)$: \[\bar{\partial}_\Lambda K_{Q_0}(Q)=\delta_{QQ_0}\frac{\pi}{\textnormal{ar}(F_Q)} \textnormal{ and } \bar{\partial}_\Diamond K_{v_0}(v)=\delta_{vv_0}\frac{\pi}{\textnormal{ar}(F_v)}.\]
\end{definition}

\begin{remark}
As for the discrete Green's functions in Section~\ref{sec:Green}, we do not require any asymptotic behavior of the discrete Cauchy's kernels. However, if interior angles and side lengths of quadrilaterals are bounded, it follows from Theorem~\ref{th:harmonic_asymptotics} that any discrete Cauchy's kernel with respect to a vertex of $\Diamond$ with asymptotics $o(v^{-1/2})$ is necessarily unique. In Section~\ref{sec:Cauchy_asymptotics}, we will construct discrete Cauchy's kernels with asymptotics similar to the smooth setting, generalizing Kenyon's result \cite{Ke02} on rhombic quad-graphs to parallelogram-graphs.
\end{remark}

The existence of discrete Cauchy's kernels follows from Theorem~\ref{th:plane_surjective}:

\begin{corollary}\label{cor:existence_Cauchykernel}
Let $Q_0 \in V(\Diamond)$ and $v_0\in V(\Lambda)$ be arbitrary. Then, discrete Cauchy's kernels with respect to $Q_0$ and $v_0$ exist.
\end{corollary}

\begin{theorem}\label{th:Cauchy_formula}
Let $f$ and $h$ be discrete holomorphic functions on $V(\Lambda)$ and $V(\Diamond)$, respectively. Let $v_0 \in V(\Lambda)$ and $Q_0 \in V(\Diamond)$, and let $K_{v_0}:V(\Diamond)\to\mC$ and $K_{Q_0}:V(\Lambda)\to\mC$ be discrete Cauchy's kernels with respect to $v_0$ and $Q_0$, respectively.

Then, for any discrete contours $C_{v_0}$ and $C_{Q_0}$ on $X$ surrounding $v_0$ and $Q_0$, respectively, once in counterclockwise order, \textit{discrete Cauchy's integral formulae} hold:
\begin{align*}
 f(v_0)&=\frac{1}{2\pi i}\oint\limits_{C_{v_0}}f K_{v_0} dz,\\
 h(Q_0)&=\frac{1}{2\pi i}\oint\limits_{C_{Q_0}}h K_{Q_0} dz.
\end{align*}
\end{theorem}
\begin{proof}
Let $P_v$ and $P_Q$ be discrete elementary cycles, $v$ being a vertex and $Q$ a face of $\Lambda$. By Lemma~\ref{lem:derivative_lambda} and the definition of $\bar{\partial}_{\Diamond}$, we get:

\begin{align*}
 \frac{1}{2\pi i}\oint\limits_{P_v}f K_{v_0}dz&=\frac{1}{\pi}\textnormal{ar}(F_v) f(v)\bar{\partial}_\Diamond K_{v_0}(v)=\delta_{vv_0} f(v),\\
\frac{1}{2\pi i}\oint\limits_{P_Q}f K_{v_0}dz&=\frac{1}{\pi}\textnormal{ar}(F_Q)\bar{\partial}_\Lambda f(Q) K_{v_0}(Q)=0.
\end{align*}
By definition, the discrete contour $C_{v_0}$ bounds a topological disk, and we can decompose the integration along $C_{v_0}$ into a couple of integrations along discrete elementary cycles $P_v$ and $P_Q$ as above. Summing up, only the contribution of $P_{v_0}$ is nonvanishing, and we get the desired result. The second formula is shown in an analog fashion.
\end{proof}

\begin{remark}
In the case of rhombic quad-graphs, Mercat formulated a discrete Cauchy's integral formula for the average of a discrete holomorphic function on $V(\Lambda)$ along an edge of $\Lambda$. In \cite{ChSm11}, Chelkak and Smirnov provided a discrete Cauchy's integral formula for discrete holomorphic functions on $V(\Diamond)$ using an integration along cycles on $\Gamma$ and $\Gamma^*$, see Section~\ref{sec:Cauchy_CS}.
\end{remark}

\begin{theorem}\label{th:Cauchy_formula_derivative}
Let $f:V(\Lambda)\to\mC$ be discrete holomorphic, $Q_0 \in V(\Diamond)$, and let $K_{Q_0}:V(\Lambda)\to\mC$ be a discrete Cauchy's kernel with respect to $Q_0$.

Then, for any discrete contour $C_{Q_0}$ in $X$ surrounding $Q_0$ once in counterclockwise order that does not contain any edge inside $Q_0$ (see Figure~\ref{fig:cycle}), the \textit{discrete Cauchy's integral formula} is true:
\begin{equation*}
 \partial_\Lambda f(Q_0)=-\frac{1}{2\pi i}\oint\limits_{C_{Q_0}}f\partial_\Lambda K_{Q_0}dz.
\end{equation*}
\end{theorem}
\begin{proof}
Let $D$ be the discrete domain in $X$ bounded by $C_{Q_0}$. Since no edge of $C_{Q_0}$ passes through $Q_0$, the discrete one-form $\bar{\partial}_\Lambda K_{Q_0}d\bar{z}$ vanishes on $C_{Q_0}$. Therefore, \[\oint\limits_{C_{Q_0}}f\partial_\Lambda K_{Q_0}dz=\oint\limits_{C_{Q_0}}fdK_{Q_0}=\iint\limits_D d(fdK_{Q_0})=\iint\limits_D df\wedge dK_{Q_0}\] due to discrete Stokes' Theorem~\ref{th:stokes}, Theorem~\ref{th:derivation}, and Proposition~\ref{prop:dd0}. Now, $f$ is discrete holomorphic, so 
$df\wedge dK_{Q_0}=\partial_\Lambda f \bar{\partial}_\Lambda K_{Q_0} \Omega_\Diamond$. But $\bar{\partial}_\Lambda K_{Q_0}$ vanishes on all vertices of $\Diamond$ but $Q_0$. Finally, \[-\frac{1}{2\pi i}\oint\limits_{C_{Q_0}}f\partial_\Lambda K_{Q_0}dz=-\frac{1}{2\pi i} \iint\limits_{F_{Q_0}}\partial_\Lambda f \bar{\partial}_\Lambda K_{Q_0} \Omega_\Diamond=\partial_\Lambda f(Q_0).\]
\end{proof}

\begin{figure}[htbp]
\begin{center}
\beginpgfgraphicnamed{cycle}
\begin{tikzpicture}
[white/.style={circle,draw=black,fill=black,thin,inner sep=0pt,minimum size=1.2mm},
black/.style={circle,draw=black,fill=white,thin,inner sep=0pt,minimum size=1.2mm},
gray/.style={circle,draw=black,fill=gray,thin,inner sep=0pt,minimum size=1.2mm}]
\node[white] (w1)
at (-2,-2) {};
\node[white] (w2)
 at (0,-2) {};
\node[white] (w3)
 at (-1,-1) {};
\node[white] (w4)
 at (1,-1) {};
\node[white] (w5)
 at (-2,0) {};
\node[white] (w6)
 at (0,0) {};
\node[white] (w7)
 at (-1,1) {};
\node[white] (w8)
 at (1,1) {};

\node[black] (b1)
 at (-1,-2) {};
\node[black] (b2)
 at (1,-2) {};
\node[black] (b3)
 at (-2,-1) {};
\node[black] (b4)
 at (0,-1) {};
\node[black] (b5)
 at (-1,0) {};
\node[black] (b6)
 at (1,0) {};
\node[black] (b7)
 at (-2,1) {};
\node[black] (b8)
 at (0,1) {};

\node[gray] (m1)
 at (0,-1.5) {};
\node[gray] (m2)
 at (0.5,-1) {};
\node[gray] (m3)
 at (0,-0.5) {};
\node[gray] (m4)
 at (0.5,0) {};
\node[gray] (m5)
 at (0,0.5) {};
\node[gray] (m6)
 at (-0.5,0) {};
\node[gray] (m7)
 at (-1,0.5) {};
\node[gray] (m8)
 at (-1.5,0) {};
\node[gray] (m9)
 at (-1,-0.5) {};
\node[gray] (m10)
 at (-1.5,-1) {};
\node[gray] (m11)
 at (-1,-1.5) {};
\node[gray] (m12)
 at (-0.5,-1) {};

\coordinate[label=center:$Q_0$] (z)  at (-0.5,-0.5) {};

\draw[dashed] (w1) -- (b1) -- (w2) -- (b2);
\draw[dashed] (b3) -- (w3) -- (b4) -- (w4);
\draw[dashed] (w5) -- (b5) -- (w6) -- (b6);
\draw[dashed] (b7) -- (w7) -- (b8) -- (w8);

\draw[dashed] (w1) -- (b3) -- (w5) -- (b7);
\draw[dashed] (b1) -- (w3) -- (b5) -- (w7);
\draw[dashed] (w2) -- (b4) -- (w6) -- (b8);
\draw[dashed] (b2) -- (w4) -- (b6) -- (w8);

\draw[color=gray] (m1) -- (m2) -- (m3) -- (m4) -- (m5) -- (m6) -- (m7) -- (m8) -- (m9) -- (m10) -- (m11) -- (m12) -- (m1);
\end{tikzpicture}
\endpgfgraphicnamed
\caption{Discrete contour as in Theorem~\ref{th:Cauchy_formula_derivative}}
\label{fig:cycle}
\end{center}
\end{figure}
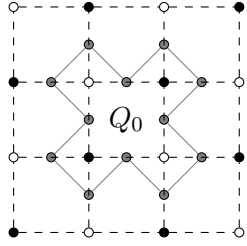

\begin{remark}
In general, there exists no analog of Theorem~\ref{th:Cauchy_formula_derivative} for the discrete derivative of a discrete holomorphic function on $V(\Diamond)$, because the discrete derivative itself does not need to be discrete holomorphic. However, in the special case of integer lattices, any discrete derivate of a discrete holomorphic function is itself discrete holomorphic. In Section~\ref{sec:integer_lattice}, we will obtain discrete analogs of Cauchy's integral formulae for higher derivatives of discrete holomorphic functions on $V(\Lambda)$ or $V(\Diamond)$.
\end{remark}


\subsubsection{A different notation}\label{sec:Cauchy_CS}

Let $W$ be a cycle on the edges of $\Gamma^*$, having (counterclockwise ordered) white vertices $w_0,w_1,\ldots, w_m$, $w_m=w_0$. Then, any edge connecting two consecutive vertices $w_k,w_{k+1}$ forms the diagonal of a face $Q(w_k,w_{k+1})\in V(\Diamond)$. We denote the set of such faces together with the induced orientation of their white diagonals by $W_\Diamond$. For $Q\in W_\Diamond$, we denote its white vertices by $w_-(Q),w_+(Q)$ such that the corresponding oriented diagonal goes from $w_-(Q)$ to $w_+(Q)$. Its black vertices are denoted by $b(Q),b'(Q)$ in such a way that $w_-(Q),b(Q),w_+(Q),b'(Q)$ appear in counterclockwise order. The reason why we do not choose our previous notation of Figure~\ref{fig:quadgraph} is that black and white vertices now play a different role that shall be indicated by the notation.

Now, we construct a cycle $B$ on the edges of $\Gamma$ having (counterclockwise ordered) black vertices $b_0,b_1,\ldots,b_n$, $b_n=b_0$, in the following way. We start with $b_0:=b\left(Q\left(w_0,w_1\right)\right)$. In the star of the vertex $w_1$, there are two simple paths on $\Gamma$ connecting $b_0$ and $b\left(Q\left(w_1,w_2\right)\right)$, and we choose the path that does not go through $Q(w_0,w_1)$. Note that it may happen that $b\left(Q\left(w_1,w_2\right)\right)=b_0$; in this case, we do not add any vertices to $B$. Also, $w_2=w_0$ is possible, which causes adding the nondirect path connecting $b_0$ and $b\left(Q\left(w_1,w_2\right)\right)=b'\left(Q\left(w_0,w_1\right)\right)$.

Continuing this procedure till we have connected $b\left(Q\left(w_{m-1},w_m\right)\right)$ with $b_0$, we end up with a closed path $B$ on $\Gamma$. Without loss of generality, any two consecutive vertices in $B$ are different. As above, any edge connecting two consecutive vertices $b_k,b_{k+1}$ forms the diagonal of a face $Q(b_k,b_{k+1})\in V(\Diamond)$. We denote the set of such faces together with the induced orientation of their black diagonals by $B_\Diamond$. For $Q\in B_\Diamond$, we denote its black vertices by $b_-(Q),b_+(Q)$ such that the corresponding oriented diagonal goes from $b_-(Q)$ to $b_+(Q)$. Finally, its white vertices are denoted by $w(Q),w'(Q)$ in such a way that $b_-(Q),w'(Q),b_+(Q),w(Q)$ appear in counterclockwise order.

\begin{definition}
Let $W$ and $B$ be as above and $h$ a function defined on $W_\Diamond \cup B_\Diamond$. We define the \textit{discrete integrals} along $W$ and $B$ by
\begin{align*}
 \oint_W h(Q)dz &:= \sum\limits_{k=0}^{m-1} f\left(Q\left(w_k,w_{k+1}\right)\right) \left(w_{k+1}-w_k\right),\\
 \oint_W h(Q)d\bar{z} &:= \sum\limits_{k=0}^{m-1} f\left(Q\left(w_k,w_{k+1}\right)\right) \overline{\left(w_{k+1}-w_k\right)};\\
 \oint_B h(Q)dz &:= \sum\limits_{k=0}^{n-1} f\left(Q\left(b_k,b_{k+1}\right)\right) \left(b_{k+1}-b_k\right),\\
 \oint_B h(Q)d\bar{z} &:= \sum\limits_{k=0}^{n-1} f\left(Q\left(b_k,b_{k+1}\right)\right) \overline{\left(b_{k+1}-b_k\right)}.
\end{align*}
\end{definition}

In between the closed paths $B$ and $W$, there is a cycle $P$ on the medial graph $X$ that comprises exactly all edges $[Q,v]$ with $Q \in W_\Diamond$ and $v \in B$ incident to $Q$ and all edges $[Q,v]$ with $Q \in B_\Diamond$ and $v \in W$ incident to $Q$. The orientation of $[Q,v]$ is induced by the orientation of the corresponding parallel white or black diagonal. Figure~\ref{fig:contours} gives an example for this construction.

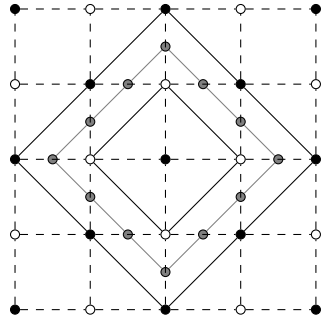
\begin{figure}[htbp]
\begin{center}
\beginpgfgraphicnamed{medial}
\begin{tikzpicture}
[white/.style={circle,draw=black,fill=black,thin,inner sep=0pt,minimum size=1.2mm},
black/.style={circle,draw=black,fill=white,thin,inner sep=0pt,minimum size=1.2mm},
gray/.style={circle,draw=black,fill=gray,thin,inner sep=0pt,minimum size=1.2mm}]
\node[white] (w1)
at (-2,-2) {};
\node[white] (w2)
 at (0,-2) {};
\node[white] (w3)
 at (2,-2) {};
\node[white] (w4)
 at (-1,-1) {};
\node[white] (w5)
 at (1,-1) {};
\node[white] (w6)
 at (-2,0) {};
\node[white] (w7)
 at (0,0) {};
\node[white] (w8)
 at (2,0) {};
\node[white] (w9)
 at (-1,1) {};
\node[white] (w10)
 at (1,1) {};
\node[white] (w11)
 at (-2,2) {};
\node[white] (w12)
 at (0,2) {};
\node[white] (w13)
 at (2,2) {};

\node[black] (b1)
 at (-1,-2) {};
\node[black] (b2)
 at (1,-2) {};
\node[black] (b3)
 at (-2,-1) {};
\node[black] (b4)
 at (0,-1) {};
\node[black] (b5)
 at (2,-1) {};
\node[black] (b6)
 at (-1,0) {};
\node[black] (b7)
 at (1,0) {};
\node[black] (b8)
 at (-2,1) {};
\node[black] (b9)
 at (0,1) {};
\node[black] (b10)
 at (2,1) {};
\node[black] (b11)
 at (-1,2) {};
\node[black] (b12)
 at (1,2) {};

\node[gray] (m1)
 at (0,-1.5) {};
\node[gray] (m2)
 at (0.5,-1) {};
\node[gray] (m3)
 at (1,-0.5) {};
\node[gray] (m4)
 at (1.5,0) {};
\node[gray] (m5)
 at (1,0.5) {};
\node[gray] (m6)
 at (0.5,1) {};
\node[gray] (m7)
 at (0,1.5) {};
\node[gray] (m8)
 at (-0.5,1) {};
\node[gray] (m9)
 at (-1,0.5) {};
\node[gray] (m10)
 at (-1.5,0) {};
\node[gray] (m11)
 at (-1,-0.5) {};
\node[gray] (m12)
 at (-0.5,-1) {};

\draw[dashed] (w1) -- (b1) -- (w2) -- (b2) -- (w3);
\draw[dashed] (b3) -- (w4) -- (b4) -- (w5) -- (b5);
\draw[dashed] (w6) -- (b6) -- (w7) -- (b7) -- (w8);
\draw[dashed] (b8) -- (w9) -- (b9) -- (w10) -- (b10);
\draw[dashed] (w11) -- (b11) -- (w12) -- (b12) -- (w13);

\draw[dashed] (w1) -- (b3) -- (w6) -- (b8) -- (w11);
\draw[dashed] (b1) -- (w4) -- (b6) -- (w9) -- (b11);
\draw[dashed] (w2) -- (b4) -- (w7) -- (b9) -- (w12);
\draw[dashed] (b2) -- (w5) -- (b7) -- (w10) -- (b12);
\draw[dashed] (w3) -- (b5) -- (w8) -- (b10) -- (w13);

\draw (w2) -- (w5) -- (w8) -- (w10) -- (w12) -- (w9) -- (w6) -- (w4) -- (w2);
\draw (b4) -- (b7) -- (b9) -- (b6) -- (b4);

\draw[color=gray] (m1) -- (m2) -- (m3) -- (m4) -- (m5) -- (m6) -- (m7) -- (m8) -- (m9) -- (m10) -- (m11) -- (m12) -- (m1);
\end{tikzpicture}
\endpgfgraphicnamed
\caption{Cycles $W$ on $\Gamma^*$, $B$ on $\Gamma$, and closed path $P$ on $X$ in between}
\label{fig:contours}
\end{center}
\end{figure}

\begin{remark}
Note that a discrete contour $P$ on $X$ induces a white cycle $W$ and a black cycle $B$ in such a way that $W$, $P$, and $B$ are related as above.
\end{remark}

\begin{lemma}\label{lem:medial_blackwhite}
Let $W$, $B$, and $P$ be as above. Let $f$ and $h$ be functions defined on the vertices of $W$ and $B$ and on $W_\Diamond \cup B_\Diamond$, respectively. Then, \[\oint_W f(b(Q))h(Q)dz+\oint_B f(w(Q)) h(Q)dz=2\oint_P fhdz.\]
\end{lemma}
\begin{proof}
Any edge $e=[Q,b(Q)]$ ($Q \in W_\Diamond$) or $[Q,w(Q)]$ ($Q \in B_\Diamond$) of $P$ corresponds to either an edge $w_-(Q)w_+(Q)$ of $W$ or to an edge $b_-(Q)b_+(Q)$ of $B$, respectively, and vice versa. Since this edge of $X$ is just half as long as the parallel edge of $\Gamma$ or $\Gamma^*$, $2\int_e fh =f\left(b\left(Q\right)\right)h(Q)(w_+-w_-)$ in the first and $2\int_e fh =f\left(w\left(Q\right)\right)h(Q)(b_+-b_-)$ in the second case. Therefore, the discrete integral along $P$ decomposes into one along $B$ and one along $W$.
\end{proof}

Note that the construction of $B$ and Lemma~\ref{lem:medial_blackwhite} are also valid if $W$ consists of a single point or of only two edges (being the same, but traversed in both directions). In any case, $P$ will be a discrete contour.

The discrete Cauchy's integral formula in the notation of Chelkak and Smirnov in \cite{ChSm11} reads as \[\pi i h(Q_0)= \oint_W h(Q) K_{Q_0}(b(Q))dz+\oint_B h(Q) K_{Q_0}(w(Q))dz\] if $Q_0 \in V(\Diamond)$ is surrounded once by $W$, $h$ is discrete holomorphic on $V(\Diamond)$, and $K_{Q_0}:V(\Lambda)\to\mC$ is a discrete Cauchy's kernel with respect to $Q_0$. Lemma~\ref{lem:medial_blackwhite} directly relates this formulation to ours in Theorem~\ref{th:Cauchy_formula}.

Finally, we conclude this section with a proposition relying on the decomposition of a discrete contour into black and white cycles. In Corollary~\ref{cor:discrete_product}~(i), we have already seen that $fdg+gdf$ is closed for functions $f,g:V(\Lambda)\to\mC$. Actually, a slightly stronger statement is true:

\begin{proposition}\label{prop:BW_Leibniz}
Let $W$ and $B$ be as above, and let $f,g:V(\Lambda)\to\mC$. Then, \[\oint_W f(b(Q)) \left(\partial_\Lambda g(Q)dz+\bar{\partial}_\Lambda g(Q) d\bar{z}\right)+\oint_B g(w(Q)) \left(\partial_\Lambda f(Q)dz+\bar{\partial}_\Lambda f(Q) d\bar{z}\right)=0.\]
\end{proposition}
\begin{proof}
We first rewrite the discrete integral along $W$. By definition, we have $dg=\partial_\Lambda g dz+\bar{\partial}_\Lambda g d\bar{z}$. Using discrete Stokes' Theorem~\ref{th:stokes}, \[\oint_W f(b(Q)) \left(\partial_\Lambda g(Q)dz+\bar{\partial}_\Lambda g(Q) d\bar{z}\right)=\sum\limits_{Q\in W_\Diamond}f(b(Q)) \left(g\left(w_+\left(Q\right)\right)-g\left(w_-\left(Q\right)\right)\right).\]
We can change the summation along the path $W$ into a summation along the path $B$ by \begin{equation*}\sum\limits_{Q\in W_\Diamond}f(b(Q)) \left(g\left(w_+\left(Q\right)\right)-g\left(w_-\left(Q\right)\right)\right)=-\sum\limits_{Q\in B_\Diamond}g\left(w\left(Q\right)\right)\left(f\left(b_+\left(Q\right)\right)-f\left(b_-\left(Q\right)\right)\right).\end{equation*} 

In an analog way as for $W$, we can rewrite the discrete integral along $B$ as \[\oint_B g(w(Q)) \left(\partial_\Lambda f(Q)dz+\bar{\partial}_\Lambda f(Q) d\bar{z}\right)=\sum\limits_{Q\in B_\Diamond}g(w(Q)) \left(f\left(b_+\left(Q\right)\right)-f\left(b_-\left(Q\right)\right)\right).\]

In summary, \[\oint_W f(b(Q)) \left(\partial_\Lambda g(Q)dz+\bar{\partial}_\Lambda g(Q) d\bar{z}\right)+\oint_B g(w(Q)) \left(\partial_\Lambda f(Q)dz+\bar{\partial}_\Lambda f(Q) d\bar{z}\right)=0.\]
\end{proof}


\section{Discrete complex analysis on planar parallelogram-graphs}\label{sec:parallel}

This part deals with discrete complex analysis on planar parallelogram-graphs $\Lambda$ that are assumed to be locally finite. The vertices $Q$ of the dual $\Diamond$ are placed at the centers of the parallelograms $Q$. Note that any planar parallelogram-graph is strongly regular and bipartite. As in Section~\ref{sec:general}, the induced graphs on black and white vertices are denoted by $\Gamma$ and $\Gamma^*$, respectively.

In Propositions~\ref{prop:examples}, \ref{prop:examples2}, and~\ref{prop:examples3}, we have already seen that discrete complex analysis on parallel\-o\-gram-graphs is closer to the classical theory than on general quad-graphs. For example, $f(v)=v^2$ is a discrete holomorphic function on $V(\Lambda)$ and $\partial_\Lambda f(Q)=2Q$; $h(Q)=Q$ is a discrete holomorphic function on $V(\Diamond)$ and $\partial_\Diamond h\equiv1$; and the discrete Laplacian $\triangle$ approximates the smooth one correctly up to order two.

In order to concentrate on the calculation of the asymptotics of a certain discrete Green's function and discrete Cauchy's kernels, we postpone the discussion of some necessary combinatorial and geometric results on planar parallelogram-graphs to the appendix. Our setup is closely related to the quasicrystallic paralle\-lo\-gram-graphs discussed by the first author, Mercat, and Suris \cite{BoMeSu05}. In their case, the quad-graph can be lifted to $\mZ^d$ and a discrete exponential defined on $\mZ^d$ can be restricted to the quad-graph. We adapt their construction to our setting in Section~\ref{sec:exp}. The discrete exponential is the basic building block for the construction of a discrete Green's function in Section~\ref{sec:Green_asymptotics} and discrete Cauchy's kernels in Section~\ref{sec:Cauchy_asymptotics}. These constructions base on the ideas of Kenyon in the case of rhombic quad-graphs \cite{Ke02}.

The corresponding functions can be defined for general planar parallelogram-graphs, but we need more regularity of the graph to calculate their asymptotics. The two conditions we use are that all interior angles of the parallelograms are bounded (the same condition was used in the presentation of Chelkak and Smirnov in \cite{ChSm11}) and that the ratio of side lengths of the parallelograms is bounded as well. For rhombic quad-graphs, the second condition is trivially fulfilled; for quasicrystallic graphs, there are only a finite number of interior angles. Note that instead of using boundedness of the ratio of side lengths of the parallelograms, we can assume that the side lengths themselves are bounded. This seems to be a stronger condition at first, but actually, both conditions are equivalent, see Proposition~\ref{prop:ratio_bounded} in the appendix.

We conclude this part by a discussion of integer lattices in Section~\ref{sec:integer_lattice}. On these graphs, discrete holomorphic functions can be discretely differentiated infinitely many times, and for all higher order discrete derivatives, discrete Cauchy's integral formulae with the right asymptotics hold true.

During this section, we use the following shorthand notation.

\begin{definition}
 Let $v,v' \in V(\Lambda)$ and $Q,Q' \in V(\Diamond)$.
\begin{enumerate}
\item Choose any directed path of edges $e_1,\ldots,e_n$ on $\Lambda$ going from $v'$ to $v$. Define \[J(v,v'):=\sum\limits_{j=1}^{n}e_{j}^{-1}.\]

\item Choose any directed path from $v$ to $Q$ that begins with a directed path on $\Lambda$ with edges $e_1,\ldots,e_n$ to a vertex $v_Q$ of the parallelogram $Q$ and that ends with two half-edges $d_1/2,d_2/2$, where $d_1,d_2$ are emanating from $v_Q$. Define \[-J(v,Q)=J(Q,v):=\sum\limits_{j=1}^{n}e_{j}^{-1}+\frac{1}{2}d_1^{-1}+\frac{1}{2}d_2^{-1}.\]

Moreover, let $\tau(v,Q)=\tau(Q,v):=1/(d_1 d_2)$ if $v_Q,v$ are both in $V(\Gamma)$ or both in $V(\Gamma^*)$ and $\tau(v,Q)=\tau(Q,v):=-1/(d_1 d_2)$ otherwise.

\item Choose any directed path from $Q'$ to $Q$ consisting of half-edges $e_1/2,e_2/2$ connecting the center of the parallelogram $Q'$ to one of its vertices $v_{Q'}$, a directed path on $\Lambda$ with edges $e_3,\ldots,e_n$ going from $v_{Q'}$ to a vertex $v_Q$ of the parallelogram $Q$, and two half-edges $d_1/2,d_2/2$ emanating from $v_Q$. Define \[J(Q,Q'):=\frac{1}{2}e_1^{-1}+\frac{1}{2}e_2^{-1}+\sum\limits_{j=3}^{n}e_{j}^{-1}+\frac{1}{2}d_1^{-1}+\frac{1}{2}d_2^{-1}.\]

Furthermore, let $\tau(Q,Q'):=1/(e_1 e_2 d_1 d_2)$ if $v_Q,v_{Q'}$ are both in $V(\Gamma)$ or both in $V(\Gamma^*)$ and $\tau(Q,Q'):=-1/(e_1 e_2 d_1 d_2)$ otherwise.
\end{enumerate}

It is easy to see that these definitions do not depend on the choice of paths.
\end{definition}

\begin{remark}
In the case that all parallelograms are rhombi of side length one, we have $J(x,x')=\overline{x-x'}$.
\end{remark}


\subsection{Discrete exponential function} \label{sec:exp}

\begin{definition}
Let $v_0\in V(\Lambda)$. Then, the \textit{discrete exponentials} $\textnormal{e}(\cdot,\cdot;v_0),\exp(\cdot,\cdot;v_0):\mC \times V(\Lambda) \rightarrow \mC$ are defined by
\begin{align*}
 \textnormal{e}(\lambda,v_0;v_0)&=1=\exp(\lambda,v_0;v_0),\\
 \frac{\textnormal{e}(\lambda,v';v_0)}{\textnormal{e}(\lambda,v;v_0)}&= \frac{\lambda+(v'-v)}{\lambda-(v'-v)},\\
 \frac{\exp(\lambda,v';v_0)}{\exp(\lambda,v;v_0)}&= \frac{1+\frac{\lambda}{2}(v'-v)}{1-\frac{\lambda}{2}(v'-v)}
\end{align*}
for all vertices $v,v' \in V(\Lambda)$ adjacent to each other and all $\lambda\in\mC$.

For a face $Q_0 \in V(\Diamond)$ with incident vertices $v_-,v'_-,v_+,v'_+$ in counterclockwise order and $v_0 \in V(\Lambda)$, we define the \textit{discrete exponentials} $\textnormal{e}(\cdot,\cdot;Q_0):\mC \times V(\Lambda) \rightarrow \mC$ and $\exp(\cdot,\cdot;v_0):\mC \times V(\Diamond) \rightarrow \mC$ as
\begin{align*}
 \textnormal{e}(\lambda,v;Q_0)&:=\frac{\textnormal{e}(\lambda,v;v_\pm)}{\left(\lambda-(v_\pm-v'_+)\right)\left(\lambda-(v_\pm-v'_-)\right)},\\
\exp(\lambda,Q_0;v_0)&:=\frac{\exp(\lambda,v_\pm;v_0)}{\left(1-\frac{\lambda}{2}(v'_+-v_\pm)\right)\left(1-\frac{\lambda}{2}(v'_--v_\pm)\right)}.
\end{align*}
\end{definition}

\begin{remark}
Note that $\exp(\lambda,\cdot;v_0)=\textnormal{e}(2/\lambda,\cdot;v_0)$. Hence, $\textnormal{e}$ and $\exp$ are equivalent up to reparametrization. On square lattices, the discrete exponential was already considered by Ferrand \cite{Fe44} and Duffin \cite{Du56}. The discrete exponential $\textnormal{e}$ on rhombic lattices was used in the work of Kenyon \cite{Ke02}, the first author, Mercat, and Suris \cite{BoMeSu05}, and B\"ucking \cite{Bue08}. To be comparable to their work, we use their notion to perform our calculations of the asymptotic behavior. In contrast, Mercat \cite{Me07} and Chelkak and Smirnov \cite{ChSm11} preferred the parametrization of $\exp$ that is closer to the smooth setting. Indeed, Mercat remarked that the discrete exponential $\exp$ in the rhombic setting is a generalization of the formula \[\exp(\lambda x)=\left(\frac{1+\frac{\lambda x}{2n}}{1-\frac{\lambda x}{2n}}\right)^n+O\left(\frac{\lambda^3x^3}{n^2}\right)\] to the case when the path from the origin to $x$ consists of $O(|x|/\delta)$ straight line segments of length $\delta$ of any directions \cite{Me07}.
\end{remark}

If $v$ is fixed, $\exp(\cdot,v;v_0)$ is a rational function with poles at all edges of a shortest directed path connecting $v_0$ with $v$. It follows from Lemma~\ref{lem:cover} that the arguments of all poles are contained in an interval of length less than $\pi$. If in addition the interior angles of parallelograms are bounded from below by $\alpha_0$, the arguments of all poles lie even in an interval of length at most $\pi-\alpha_0$ by Lemma~\ref{lem:cover}. But in any case, the following motivation to call $\exp$ discrete exponential can be checked in a straightforward way.

\begin{proposition}\label{prop:exp_holomorphic}
For any $\lambda \in \mC$, $v_0 \in V(\Lambda)$, and $Q \in V(\Diamond)$, $\exp(\lambda,\cdot;v_0)$ is discrete holomorphic and \[\left(\partial_\Lambda \exp\left(\lambda,\cdot;v_0\right)\right) \left(Q\right)=\lambda\exp(\lambda,Q;v_0).\]
\end{proposition}


\subsection{Asymptotics of the discrete Green's function} \label{sec:Green_asymptotics}

Following the presentation of the first author, Mercat, and Suris \cite{BoMeSu05}, we first define a discrete logarithmic function on a certain branched covering $\tilde{\Lambda}_{v_0}$ of $\Lambda$.

Fix $v_0 \in V(\Lambda)$, and let $e_1, e_2, \ldots, e_n$ be the directed edges starting in $v_0$, ordered according to their slopes. To each of these edges $e$ we assign the real angle $\theta_e:=\arg(e) \in [0,2\pi)$. We assume that $\theta_{e_1}<\theta_{e_n}$. Now, define $\theta_{a+bn}:=\theta_a+2\pi b$, where $a \in \left\{1,\ldots,n\right\}$ and $b\in \mZ$.

Let $U_e \subset V(\Lambda)$ denote the set of all vertices to that $v_0$ can be connected by a directed path of edges whose arguments lie in $[\arg(e),\arg(e)+\pi)$, where $e$ is one of the $e_k$. Lemma~\ref{lem:cover} shows that the union of all these $U_{e_k}$, $k=1,\ldots,n$, covers the whole quad-graph. It follows that \[\tilde{U}:=\bigcup\limits_{m=-\infty}^\infty\tilde{U}_m\] defines a parallelogram-graph $\tilde{\Lambda}_{v_0}$ that is a branched covering of $\Lambda$, branched over $v_0$. Here, $\tilde{U}_m$ is the set of all vertices of $\Lambda$ to that $v_0$ can be connected by a directed path of edges whose arguments lie in $[\theta_m,\theta_m+\pi)$, equipped with the additional datum of this interval. Then, all $\tilde{U}_{m+bn}$, $b\in\mZ$ and $1\leq m\leq n$, cover the same sector $U_{e_m}$, and $\tilde{U}_{m} \cap \tilde{U}_{m'}\neq\emptyset$ if and only if $|m-m'|<n$.

To each vertex $\tilde{v} \in V(\tilde{\Lambda}_{v_0})$ covering a vertex $v \neq v_0$ of $\Lambda$, we assign $\theta_{\tilde{v}} \equiv \arg(v-v_0) \mod 2\pi$ such that $\theta_{\tilde{v}} \in [\theta_m,\theta_m+\pi)$ if $\tilde{v}\in \tilde{U}_{m}$. Then, $\theta_{\tilde{v}}$ increases by $2 \pi$ when $\tilde{v}$ winds once around $v_0$ in counterclockwise order; and if $\tilde{v},\tilde{v}' \neq v_0$ are adjacent vertices of $\tilde{\Lambda}_{v_0}$, $|\theta_{\tilde{v}}-\theta_{\tilde{v}'}|<\pi$.

Note that the strip passing through an edge $\tilde{e}$ separates its two endpoints from each other in the quad-graph $\tilde{\Lambda}_{v_0}$. For the definition of a strip, see the appendix. By construction, if we connect $v_0$ to some $\tilde{v} \neq v_0$ by a shortest directed path of edges of $\tilde{\Lambda}_{v_0}$, the angles assigned to the edges lie all in $(\theta_{\tilde{v}}-\pi,\theta_{\tilde{v}}+\pi)$.

\begin{definition}
Let $v_0 \in V(\Lambda)$ and let $\tilde{\Lambda}_{v_0}$ be the corresponding branched covering of $\Lambda$. The \textit{discrete logarithmic function} on $V(\tilde{\Lambda}_{v_0})$ is given by \[\log(\tilde{v};v_0):=\frac{1}{2\pi i}\int\limits_{C_{\tilde{v}}} \frac{\log(\lambda)}{2 \lambda}\textnormal{e}(\lambda,v;v_0) d\lambda,\] where $C_{\tilde{v}}$ is a collection of counterclockwise oriented loops going once around each pole of $\textnormal{e}(\cdot,v;v_0)$, $v \in V(\Lambda)$ being the projection of $\tilde{v}\in V(\tilde{\Lambda}_{v_0})$. The arguments of all poles shall lie in $(\theta_{\tilde{v}}-\pi,\theta_{\tilde{v}}+\pi)$.
\end{definition}

It is easy to see that the real part of the discrete logarithm $\log(\cdot;v_0)$ is a well defined function on $V(\Lambda)$. Divided by $2\pi$, one actually obtains a discrete Green's function with respect to $v_0$. The first author, Mercat, and Suris showed that this function coincides with the one of Kenyon \cite{Ke02} in the rhombic case \cite{BoMeSu05}. Their proof can be adapted to our setting.

\begin{proposition}\label{prop:Green_construction}
Let $v_0 \in V(\Lambda)$. The function $G(\cdot;v_0):V(\Lambda)\to\mR$ defined by $G(v_0;v_0)=0$ and \[G(v;v_0)=\frac{1}{2\pi}\re\left(\frac{1}{2\pi i}\int\limits_{C_v} \frac{\log(\lambda)}{2 \lambda}\textnormal{e}(\lambda,v;v_0) d\lambda\right)\] is a (free) discrete Green's function with respect to $v_0$. Here, $C_v$ is a collection of counterclockwise oriented loops going once around each pole of $\textnormal{e}(\cdot,v;v_0)$, where the arguments of all poles shall lie in $(\arg(v-v_0)-\pi,\arg(v-v_0)+\pi)$.
\end{proposition}

\begin{theorem}\label{th:Green_asymptotics}
Assume that there are $\alpha_0,q_0>0$ such that $\alpha\geq\alpha_0$ and $e/e'\geq q_0$ for all interior angles $\alpha$ and the two side lengths $e,e'$ of any parallelogram of $\Lambda$. Let $v_0 \in V(\Lambda)$ be fixed.

Then, the discrete Green's function $G(\cdot;v_0)$ constructed in Proposition~\ref{prop:Green_construction} has the following asymptotic behavior:
\begin{align*}
 G(v;v_0)&=\frac{1}{4\pi}\log\left|\frac{v-v_0}{J(v,v_0)}\right|+O\left(|v-v_0|^{-2}\right) \textnormal{ if } v \textnormal{ and } v_0 \textnormal{ are of different color},\\
 G(v;v_0)&=\frac{\gamma_{\textnormal{Euler}}+\log(2)}{2\pi}+\frac{1}{4\pi}\log\left|(v-v_0) J(v,v_0)\right|+O\left(|v-v_0|^{-2}\right) \textnormal{ otherwise.} 
\end{align*}
Here, $\gamma_{\textnormal{Euler}}$ denotes the Euler-Mascheroni constant.
\end{theorem}
The proof follows the ideas of Kenyon \cite{Ke02} and B\"ucking \cite{Bue08}. Both considered just quasicrystallic rhombic quad-graphs. But the main difference to \cite{Ke02} is that we deform the path of integration into an equivalent one different from Kenyon's, since his approach does not generalize to parallelogram-graphs. As Chelkak and Smirnov did for rhombic quad-graphs with bounded interior angles in \cite{ChSm11}, Kenyon used that two points $v,v' \in V(\Lambda)$ can be connected by a directed paths of edges such that the angle between each directed edge and $v'-v$ is less than $\pi/2$ or the angle between the sum of two consecutive edges and $v'-v$ is less than $\pi/2$. This is true for rhombic quad-graphs, but not for parallelogram-graphs. Instead, we use essentially the same deformation of the path of integration as B\"ucking did. In this paper, we will omit the straightforward calculations, the full proof can be found in the thesis of the second author \cite{Gue14}.
\begin{proof}
The poles $e_1,\ldots,e_{k(v)}$ of $\textnormal{e}(\cdot,v;v_0)$ correspond to the directed edges of a shortest path from $v_0$ to $v$. By Lemma~\ref{lem:cover}, there is a real $\theta_0$ such that the angles associated to the directed edges above lie all in $[\theta_0,\theta_0+\pi-\alpha_0]$. Without loss of generality, we assume $\theta_0=-(\pi-\alpha_0)/2$. By definition,
\[G(v;v_0)=\re\left(\frac{1}{8\pi^2 i}\int\limits_{C_v} \frac{\log \lambda}{ \lambda}\textnormal{e}(\lambda,v;v_0) d\lambda\right),\]
where $C_v$ is a collection of counterclockwise oriented loops going once around each $e_1,\ldots,e_{k(v)}$. The arguments of the poles correspond to the assigned angles. By residue theorem, we can deform $C_v$ into a new path of integration $C_v'$ that goes first along a circle centered at $0$ with large radius $R(v)$ (such that all poles lie inside this disk) in counterclockwise direction starting and ending in $-R(v)$, then goes along the line segment $[-R(v),-r(v)]$ followed by the circle centered at $0$ with small radius $r(v)$ (such that all poles lie outside this disk) in clockwise direction, and finally goes the line segment $[-R(v),-r(v)]$ backwards. Note that $\log$ differs by $2\pi i$ for the two integrations along $[-R(v),-r(v)]$.

By Proposition~\ref{prop:ratio_bounded}, there are $E_1>E_0>0$ such that $E_1\geq e \geq E_0$ for all lengths $e$ of edges of $\Lambda$. In particular, we can choose the radii of order \[R(v)=\Omega(|v-v_0|^4) \textnormal{ and } r(v)=\Omega(|v-v_0|^{-4}).\] By construction, $k(v)\leq \re(v-v_0)/(E_0\cos(\theta_0))$, so $k(v)=O(|v-v_0|)$.  Also, \[\left|J(v,v_0)\right|=\left|\sum\limits_{j=1}^{k(v)} e_j^{-1}\right|\leq \frac{k(v)}{E_0} \textnormal{ and } \re(J(v,v_0))\geq\frac{\cos(\theta_0)|v-v_0|}{E_1^2}.\] Hence, $|J(v,v_0)|=\Omega(|v-v_0|).$

We first look at the contributions of the circles with radii $r(v)$ and $R(v)$ to $G(v;v_0)$. For $\lambda \rightarrow 0$, 
$(-1)^{k(v)} \textnormal{e}(\lambda,v;v_0)=1+O(\lambda |v-v_0|)$, and for $\lambda \rightarrow \infty$, $\textnormal{e}(\lambda,v;v_0)=1+O(\lambda^{-1} |v-v_0|).$

Thus, we get $(-1)^{k(v)}\log(r(v))/(4\pi)\left(1+O(|v-v_0|^{-3})\right)$ as the contribution of the integration along the small circle with radius $r(v)$. Similarly, $\log(R(v))/(4\pi)\cdot\left(1+O(|v-v_0|^{-3})\right)$ is the contribution of the circle of radius $R(v)$.

The two integrations along $[-R(v),-r(v)]$ can be combined into the integral \[\frac{1}{4\pi}\int\limits_{-R(v)}^{-r(v)}\frac{\textnormal{e}(\lambda,v;v_0)}{\lambda}d\lambda.\] For $|v-v_0|\geq 1$ large enough, we split the integration into the three parts along \[[-R(v),-E_1\sqrt{|v-v_0|}], [-E_1\sqrt{|v-v_0|},-\frac{E_1}{\sqrt{|v-v_0|}}], \textnormal{ and }[-\frac{E_1}{\sqrt{|v-v_0|}},-r(v)].\]

We first consider the intermediate range $\lambda \in [-E_1\sqrt{|v-v_0|},-E_1/\sqrt{|v-v_0|}]$. Using that \[\frac{|\lambda+e|^2}{|\lambda-e|^2}=1+\frac{4\lambda \re(e)}{\lambda^2-2\lambda\re(e)+|e|^2}\leq 1+\frac{4\lambda \re(e)}{(\lambda-|e|)^2}\leq \exp\left(\frac{4\lambda \re(e)}{(\lambda-E_1)^2}\right),\] \[|\textnormal{e}(\lambda,v;v_0)|\leq \exp\left(\frac{2\lambda \re(v-v_0)}{(\lambda-E_1)^2}\right)\leq\exp\left(\frac{2\lambda \cos(\theta_0)|v-v_0|}{(\lambda-E_1)^2}\right).\]
As a consequence, the contribution of the intermediate range decays faster to zero than any power of $|v-v_0|$.

Now, let $\lambda\in[-E_1/\sqrt{|v-v_0|},-r(v)]$ be small. Then,
\begin{align*}(-1)^{k(v)}\textnormal{e}(\lambda,v;v_0)&=\exp\left(2\lambda J(v,v_0)+O(k(v)\lambda^3)\right)=\exp\left(2\lambda J(v,v_0)\right)\left(1+O(|v-v_0|\lambda^3)\right).
\end{align*}
Thus, the integral near the origin is equal to 
\begin{align*}
&\frac{(-1)^{k(v)}}{4\pi}\int\limits_{-E_1/\sqrt{|v-v_0|}}^{-r(v)}\left(\frac{\exp\left(2\lambda J(v,v_0)\right)}{\lambda} +\exp\left(2\lambda J(v;v_0)\right)O(|v-v_0|\lambda^2)\right)  d\lambda\\
=&\frac{(-1)^{k(v)}}{4\pi}\left(\gamma_{\textnormal{Euler}}+\Omega(|v-v_0|^{-3})+ \log(2r(v) J(v,v_0))\right).
\end{align*}

Similarly, the integration along $\lambda\in[-R(v),-E_1\sqrt{|v-v_0|}]$ gives \[\gamma_{\textnormal{Euler}}-\log\left(\frac{R(v)}{2(v-v_0)}\right)+O(|v-v_0|^{-2}).\]

Summing up and taking the real part, we finally get \[4\pi G(v;v_0)=\left(1+(-1)^{k(v)}\right)\left(\gamma_{\textnormal{Euler}}+\log(2)\right)+\log|v-v_0|+(-1)^{k(v)}\log\left|J(v;v_0)\right|+O(|v-v_0|^{-2}).\]
\end{proof}
\begin{remark}
Under the conditions of Theorem~\ref{th:Green_asymptotics}, it follows from Proposition~\ref{prop:ratio_bounded} and Theorem~\ref{th:harmonic_asymptotics} that there is exactly one discrete Green's function having the same asymptotics as the one we constructed above up to terms of order $o(|v-v_0|^{-1/2})$.

Let us compare this result to the case of rhombi of side length one. Assume that $v_0\in V(\Gamma)$. Then, the discrete logarithm is purely real and nonbranched on $V(\Gamma)$ and purely imaginary and branched on $V(\Gamma^*)$. Thus, $G(v;v_0)=0$ if $v \in V(\Gamma^*)$, well fitting to the fact that $\triangle$ splits into two discrete Laplacians on $\Gamma$ and $\Gamma^*$. Using $J(v,v_0)=\overline{v-v_0}$, \[G(v;v_0)=\frac{1}{2\pi}\left(\gamma_{\textnormal{Euler}}+\log(2)+\log|v-v_0|\right)+O(|v-v_0|^{-2}),\] exactly as in the work of B\"ucking \cite{Bue08}, who slightly strengthened the error term in Kenyon's work \cite{Ke02}. In this paper, Kenyon showed that there is no further discrete Green's function with asymptotics $o(|v-v_0|)$. 
\end{remark}


\subsection{Asymptotics of discrete Cauchy's kernels}\label{sec:Cauchy_asymptotics}

Let $v_0 \in V(\Lambda)$ and $Q_0 \in V(\Diamond)$. We first construct a discrete Cauchy's kernel $K_{v_0}$ with respect to $v_0$ on $V(\Diamond)$ that has asymptotics $\Omega(|Q-v_0|^{-1})$. Then, we construct a discrete Cauchy's kernel $K_{Q_0}$ with respect to $Q_0$ on $V(\Lambda)$ with asymptotics $\Omega(|v-Q_0|^{-1})$. In both cases, there are no further discrete Cauchy's kernels with asymptotics $o(|Q-v_0|^{-1/2})$ or $o(|v-Q_0|^{-1/2})$. In the end of this section, we determine the asymptotics of $\partial_\Lambda K_{Q_0}$.

\begin{theorem}\label{th:Cauchy_asymptotics_1}
Let $v_0 \in V(\Lambda)$ and $G(\cdot;v_0)$ be a discrete Green's function with respect to $v_0$.
\begin{enumerate}
\item $K_{v_0}:=8\pi \partial_{\Lambda} G(\cdot;v_0)$ is a discrete Cauchy's kernel with respect to $v_0$.
\item Assume additionally that there exist positive real numbers $\alpha_0,q_0$ such that $\alpha\geq\alpha_0$ and $e/e'\geq q_0$ for all interior angles $\alpha$ and the two side lengths $e,e'$ of any parallelogram of $\Lambda$. Suppose that $G(\cdot;v_0)$ is the discrete Green's function constructed in Proposition~\ref{prop:Green_construction} and $K_{v_0}$ the discrete Cauchy's kernel given in (i). Then, \[K_{v_0}(Q)=\frac{1}{Q-v_0}+\frac{\tau(Q,v_0)}{J(Q,v_0)}+O(|Q-v_0|^{-2}).\]
\item Under the conditions of (ii), there is exactly one discrete Cauchy's kernel with respect to $v_0$ with asymptotics $o(|Q-v_0|^{-1/2})$.
\end{enumerate}
\end{theorem}
The first part is an immediate consequence of Corollary~\ref{cor:factorization}. A straightforward calculation shows the second part, Theorem~\ref{th:harmonic_asymptotics} the third part.

Since we do not have discrete Green's functions on $V(\Diamond)$, we have to construct discrete Cauchy's kernels on $V(\Lambda)$ differently. To do so, we follow the original approach of Kenyon using the discrete exponential \cite{Ke02} that was reintroduced by Chelkak and Smirnov in \cite{ChSm11}. Their proofs can be adapted to our setting.

\begin{proposition}\label{prop:Cauchy_exp_construction}
Let $Q_0 \in V(\Diamond)$. The function defined by \[K_{Q_0}(v):=\frac{1}{\pi i}\int\limits_{C_v} \log(\lambda) \textnormal{e}(\lambda,v;Q_0) d\lambda=2\int\limits_{-(v-Q_0)\infty}^0 \textnormal{e}(\lambda,v;Q_0)d\lambda\] is a discrete Cauchy's kernel with respect to $Q_0$. Here, $C_v$ is a collection of counterclockwise oriented loops going once around each pole of $\textnormal{e}(\cdot,v;Q_0)$. The arguments of all poles shall lie in $(\theta_v-\pi,\theta_v+\pi)$, where $\theta_v:=\arg(v-Q_0)$.
\end{proposition}

\begin{theorem}\label{th:Cauchy_asymptotics_2}
Assume that there are $\alpha_0,q_0>0$ such that $\alpha\geq\alpha_0$ and $e/e'\geq q_0$ for all interior angles $\alpha$ and the two side lengths $e,e'$ of any parallelogram of $\Lambda$. Let $Q_0 \in V(\Diamond)$ be fixed.

\begin{enumerate}
 \item The discrete Cauchy's kernel $K_{Q_0}$ constructed in Proposition~\ref{prop:Cauchy_exp_construction} has the following asymptotics:
\begin{equation*}
 K_{Q_0}(v)=\frac{1}{v-Q_0}+\frac{\tau(v,Q_0)}{J(v,Q_0)}+O\left(|v-Q_0|^{-3}\right).
\end{equation*}
 \item There is no further discrete Cauchy's kernel with respect to $Q_0$ of asymptotics $o(|v-Q_0|^{-1/2})$.
 \item For $K_{Q_0}$ constructed in Proposition~\ref{prop:Cauchy_exp_construction}, $\partial_\Lambda K_{Q_0}$ has asymptotics \[\partial_\Lambda K_{Q_0}(Q)=-\frac{1}{(Q-Q_0)^2}-\frac{\tau(Q,Q_0)}{J(Q,Q_0)^2}+O\left(|Q-Q_0|^{-3}\right).\]
 \item Up to two additive constants on $\Gamma$ and $\Gamma^*$, there is no further discrete Cauchy's kernel with respect to $Q_0$ such that its discrete derivative has asymptotics $o(|Q-Q_0|^{-1/2})$.
\end{enumerate}
\end{theorem}

The proof of the first part follows the ideas of Kenyon \cite{Ke02}. Similar to the proof of Theorem~\ref{th:Green_asymptotics}, the path of integration is deformed into one different from $(-(v-Q_0)\infty,0]$ that was used by Kenyon. For the same reasons as before, his approach does not generalize to parallelogram-graphs. The second and the fourth part of the theorem are immediate consequences of Theorem~\ref{th:harmonic_asymptotics}; the third part is shown by a direct computation.

\begin{remark}
Note that Kenyon and Chelkak and Smirnov proved that in the rhombic setting there is a unique discrete Cauchy's kernel with asymptotics $o(1)$ \cite{Ke02,ChSm11}.
\end{remark}

\subsection{Integer lattice} \label{sec:integer_lattice}

Let us consider a locally finite planar parallelogram-graph $\Lambda$ such that each vertex has degree four. This happens if and only if $\Diamond$ is also a quad-graph or, equivalently, if $\Lambda$ has the combinatorics of the integer lattice $\mZ^2$.

As usual, let $\Gamma$ and $\Gamma^*$ denote the subgraphs on black and white vertices, and let us place the vertices of $\Diamond$ at the centers of corresponding faces. Then, the vertices of $\Diamond$ lie at the midpoints of diagonals. Since any vertex of $\Gamma$ or $\Gamma^*$ is enclosed by a quadrilateral of $\Gamma^*$ or $\Gamma$, respectively, the faces of $\Diamond$ are parallelograms by Varignon's theorem. Thus, $\Diamond$ becomes a planar parallelogram-graph as well.

Of particular interest is the case that the two notions of discrete holomorphicity on $\Diamond$, the one coming from $\Diamond$ being the dual of $\Lambda$ and the other coming from the quad-graph $\Diamond$ itself, coincide. It is not hard to show that this happens only for the integer lattice of a skew coordinate system, onto which we restrict ourselves in the following. If $e_1$, $e_2$ denote two spanning vectors, $\Diamond$ is a parallel shift of $\Lambda$ by $e_1/2+e_2/2$. Furthermore, the discrete differentials on $\Diamond$ seen as the dual of $\Lambda$ coincide with the discrete differentials on $\Diamond$ seen as a parallelogram-graph.

Since corresponding notions coincide and $\Diamond$ and $\Lambda$ are congruent, we can skip all subscripts $\Lambda$ and $\Diamond$ in the definitions of discrete derivatives. Moreover, the discrete Laplacian $\triangle$ is now defined for functions on $V(\Lambda)$ and functions on $V(\Diamond)$ in the same way. Due to Corollary~\ref{cor:factorization}, $4\partial\bar{\partial}=\triangle=4\bar{\partial}\partial$ is now true on both graphs. It follows that all discrete derivatives $\partial^n f$ of a discrete holomorphic function $f$ are discrete holomorphic themselves. Conversely, a discrete primitive exists for any discrete holomorphic function on a simply-connected domain by Proposition~\ref{prop:primitive}.

Our main interests lie in giving discrete Cauchy's integral formulae for higher order derivatives of a discrete holomorphic function and determining the asymptotics of higher order discrete derivatives of the discrete Cauchy's kernel constructed in Section~\ref{sec:Cauchy_asymptotics}. Note that due to the uniqueness statements in Theorems~\ref{th:Cauchy_asymptotics_1} and~\ref{th:Cauchy_asymptotics_2}, both constructions yield the same discrete Cauchy's kernel.

Without loss of generality, we restrict our attention to functions on $V(\Lambda)$. For the ease of notation, we introduce the \textit{discrete distance} $D(\cdot,\cdot)$ on $V(\Lambda) \cup V(\Diamond)$ that is induced by the $|\cdot|_{\infty}$-distance on the integer lattice spanned by $e_1/2,e_2/2$.

\begin{theorem}\label{th:Cauchy_formula_n_derivative}
Let $v_0\in V(\Lambda)$, $Q_0:=v_0+e_1/2+e_2/2 \in V(\Diamond)$, let $f$ be a discrete holomorphic function on $V(\Lambda)$ and let $K_{v_0}$ and $K_{Q_0}$ be discrete Cauchy's kernels with respect to $v_0$ and $Q_0$, respectively. Let $n$ be a nonnegative integer and define $x_0:=v_0$ if $n$ is even and $x_0:=Q_0$ if $n$ is odd. Similarly, let $x \in V(\Lambda)$ if $n$ is even and $x \in V(\Diamond)$ if $n$ is odd.

\begin{enumerate}
 \item For any counterclockwise oriented discrete contour $C_{x_0}$ in the medial graph $X$ enclosing all points $x' \in V(\Lambda) \cup V(\Diamond)$ with $D(x',x_0)\leq n/2$, 
\begin{equation*}
 \partial^n f(x_0)=\frac{(-1)^n}{2\pi i}\oint\limits_{C_{x_0}}f\partial^n K_{x_0} dz.
\end{equation*}
 \item If $K_{Q_0}$ is the discrete Cauchy's kernel constructed in Proposition~\ref{prop:Cauchy_exp_construction}, \[\frac{(-1)^n}{n!} \partial^n K_{Q_0}(x)=\frac{1}{(x-Q_0)^{n+1}}+\frac{\tau'(x,Q_0)}{(J(x,Q_0)e_1e_2)^{n+1}}+O(|x-Q_0|^{-n-3}),\]
where $\tau'(x,Q_0)=1$ if $x$ and $Q_0$ or $(x+e_1/2+e_2/2)$ and $Q_0$ can be connected by a path on $V(\Diamond)$ of even length and $\tau'(x,Q_0)=-1$ otherwise.
\end{enumerate}
\end{theorem}

\begin{proof}
(i) Let $D$ be the discrete domain in $X$ bounded by $C_{x_0}$. By the assumptions on $C_{x_0}$, the discrete one-form $\bar{\partial} \partial^{n-1} K_{x_0} d\bar{z}$ vanishes on $C_{x_0}$. Thus, \[\oint\limits_{C_{x_0}}f\partial^n K_{x_0}dz=\oint\limits_{C_{x_0}}fd\left(\partial^{n-1} K_{x_0}\right)=\iint\limits_D d(fd(\partial^{n-1} K_{x_0}))=\iint\limits_D df\wedge d\left(\partial^{n-1}K_{x_0}\right)\] by discrete Stokes' Theorem~\ref{th:stokes}, Theorem~\ref{th:derivation}, and Proposition~\ref{prop:dd0}. Now, $f$ is discrete holomorphic, so 
$df\wedge d\left(\partial^{n-1}K_{x_0}\right)=\partial f \bar{\partial}\partial^{n-1}K_{x_0} \Omega_\Diamond$. But since the discrete derivatives commute according to Corollary~\ref{cor:commutativity}, $\bar{\partial}\partial^{n-1}K_{x_0}=\partial^{n-1}\bar{\partial}K_{x_0}$ vanishes outside $C_{x_0}$, so by subsequent use of Proposition~\ref{prop:adjoint}, \[\oint\limits_{C_{x_0}}f\partial^n K_{x_0}dz=-2i\langle \partial f, \bar{\partial}^{n-1}\partial\bar{K}_{x_0}  \rangle = 2i(-1)^n\langle \partial^n f,\partial\bar{K}_{x_0}  \rangle=2\pi i(-1)^n \partial^n f(x_0).\]

(ii) Again, the proof follows the lines of the proof of Theorem~\ref{th:Green_asymptotics}. In addition, we use that Proposition~\ref{prop:exp_holomorphic} shows that the discrete derivative of the discrete exponential is up to an explicit factor again a discrete exponential. Again, we exchange the discrete differentiation and the integration in the proof.
\end{proof}


\begin{appendix}
\numberwithin{theorem}{section}

\section{Appendix: Planar parallelogram-graphs} \label{sec:basic_parallel}

The aim of this appendix is to discuss some combinatorial and geometric properties of parallelogram-graphs that were used in Section~\ref{sec:parallel}. The following notion of a strip is standard, see for example the book of the first author and Suris \cite{BoSu08}.

\begin{definition}
A \textit{strip} in a planar quad-graph $\Lambda$ is a path on its dual $\Diamond$ such that two successive faces share an edge and the strip leaves a face in the opposite edge where it enters it. Moreover, strips are assumed to have maximal length, i.e., there are no strips containing it apart from itself.
\end{definition}

Note that a strip is uniquely determined by the edges it passes through, meaning the edges two successive faces share.

\begin{definition}
For a strip $S$ of a parallelogram-graph $\Lambda$, there exists a complex vector $a_S$ such that any (nonoriented) edge through which $S$ passes is $\pm a_S$. We call $a_S$ a \textit{common parallel}.
\end{definition}

$a_S$ is unique up to sign; the choice of the sign induces an orientation on all edges. The parallel edges of the strip can be rescaled to length $|a_S|=1$, without changing the combinatorics. Hence, rhombic planar quad-graphs and planar parallelogram-graphs are combinatorially equivalent. Rhombic planar quad-graphs are characterized by the following proposition of Kenyon and Schlenker \cite{KeSch05}:

\begin{proposition}\label{prop:rhombic}
A planar quad-graph $\Lambda$ admits a combinatorially equivalent embedding in $\mC$ with all rhombic faces if and only if the following two conditions are satisfied:
\begin{itemize}
\item No strip crosses itself or is periodic.
\item Two distinct strips cross each other at most once.
\end{itemize}
\end{proposition}

That planar parallelogram-graphs fulfill these two conditions was already noted by Kenyon in \cite{Ke02}. The reason for that is that any strip $S$ is monotone with respect to the direction $ia_S$: The coordinates of the endpoints of the edges parallel to $a_S$ are strictly increasing or strictly decreasing if they are projected to $ia_S$. Whether the projections are decreasing or increasing depends on the direction in which the faces of $S$ are passed through. Without loss of generality, we assume that the faces of $S$ are passed through in such a way that the projections of the corresponding coordinates are strictly increasing. For $Q \in S$, let $S^Q$ denote the semi-infinite part of $S$ starting in the quadrilateral $Q$ that passes through the faces of $S$ in the same order.

As a consequence, no strip crosses itself or is periodic. Furthermore, $S$ divides the complex plane into two unbounded regions, to one is $a_S$ pointing and to the other $-a_S$. When a distinct strip $S'$ crosses $S$, it enters a different region determined by $S$, say it goes to the one to which $a_S$ is pointing. Due to monotonicity, the angle between $ia_{S'}$ and $a_S$ is less than $\pi/2$. It follows that $S'$ cannot cross $S$ another time, since it would then go to the region $-a_S$ is pointing to, contradicting that the angle between $ia_{S'}$ and $-a_S$ is greater and not less than $\pi/2$.

In order to construct the discrete Green's function and the discrete Cauchy's kernels in Sections~\ref{sec:Green_asymptotics} and~\ref{sec:Cauchy_asymptotics}, we chose a particular directed path connecting two vertices (or a face and a vertex) by edges of the parallelogram-graph $\Lambda$. This path was monotone in one direction and any angle between two consecutive was less than $\pi$. The existence of such a path follows from the following lemma, generalizing a result of the first author, Mercat and Suris \cite{BoMeSu05} to general parallelogram-graphs. The proof bases on the same ideas.

\begin{lemma}\label{lem:cover}
Let $\Lambda$ be a parallelogram-graph and let $v_0 \in V(\Lambda)$ be fixed. For a directed edge $e$ of $\Lambda$ starting in $v_0$, consider the set $U_e \subset V(\Lambda)$ of all vertices to that $v_0$ can be connected by a directed path of edges whose arguments lie in $[\arg(e),\arg(e)+\pi)$.

Then, the union of all $U_e$, $e$ a directed edge starting in $v_0$, covers the whole quad-graph.

If there is $\alpha_0>0$ such that $\alpha\geq\alpha_0$ for all interior angles $\alpha$ of faces of $\Lambda$, the same statement holds true if $[\arg(e),\arg(e)+\pi)$ is replaced by $[\arg(e),\arg(e)+\pi-\alpha_0]$.
\end{lemma}
\begin{proof}
Let us rescale the edges such that all of them have length one. By this, we change neither the combinatorics of $\Lambda$ nor the size of interior angles.

For a directed edge $e$ starting in $v_0$, let $U_e^-$ and $U_e^+$ denote the (directed) paths on $\Lambda$ starting in $v_0$, obtained by choosing the directed edge with the least or largest argument in $[\arg(e),\arg(e)+\pi)$ (or $[\arg(e),\arg(e)+\pi-\alpha_0]$) at a vertex, respectively. We first show that all vertices in between $U_e^-$ and $U_e^+$ are contained in $U_e$. Then, it follows that $U_e$ is the conical sector with boundary $U_e^-$ and $U_e^+$.

Suppose the contrary, i.e., suppose that there is a vertex $v$ between $U_e^-$ and $U_e^+$ to which $v_0$ cannot be connected by a directed path of edges whose arguments lie all in the interval $[\arg(e),\arg(e)+\pi)$ (or $[\arg(e),\arg(e)+\pi-\alpha_0]$). Let the combinatorial distance between $v_0$ and $v$ be minimal among all such vertices.

In the case that interior angles of rhombi are bounded by $\alpha_0$ from below, they are bounded from above by $\pi-\alpha_0$. Hence, there is a vertex $v'$ adjacent to $v$ such that the argument of the directed edge $v'v$ lies in $[\arg(e),\arg(e)+\pi-\alpha_0]$. Even if interior angles are not uniformly bounded, $v'$ can be chosen in such a way that the argument of $v'v$ lies in $[\arg(e),\arg(e)+\pi)$. By construction, $v'$ is not in $U_e$, but still between $U_e^-$ and $U_e^+$. Let us look at the strip $S$ passing through $v'v$. Suppose that the common parallel $a_S$ points from $v'$ to $v$.

If $S$ intersects $U_e^-$ or $U_e^+$, an edge parallel to $a_S$ is contained in $U_e^-$ or $U_e^+$, respectively. By construction, $v_0$ and $v'$ then lie on the same side of the strip $S$.

If $S$ does neither intersect $U_e^-$ nor $U_e^+$, it is completely contained in the left half space determined by the oriented line $v_0+te$, $t \in \mR$, as $U_e^-$ and $U_e^+$ are. Suppose $S$ intersects the ray $v_0+ta_S$, $t\geq 0$. Again, it follows that $v_0$ and $v'$ lie on the same side of $S$.

It remains the case that $S$ neither intersects $U_e^-$, $U_e^+$, nor the ray $v_0+ta_S$, $t\geq 0$. Consider the quadrilateral area $R$ in between the parallels $v_0+ta_S$, $v'+ta_S$ and $v_0+te$, $v'+te$, $t\in \mR$. By assumption, the semi-infinite part of $S$ that starts with the edge $v'v$ and then goes into $R$ does not intersect an edge of $R$ incident to $v_0$, and by monotonicity, it does not intersect $v'+ta_S$ again. Now, $\Lambda$ is locally finite, such that only finitely many quadrilaterals of $S$ are inside $P$. Thus, $S$ leaves $P$ on the edge $v'+te$, $t\in \mR$, and it follows that $S$ separates $v_0$ and $v$.

So in any case, $S$ separates $v_0$ from $v$. Any shortest path $P$ connecting both points has to pass through $S$. Let $w$ be the first point of $P$ that lies on the same side of $S$ as $v$ does. Any strip passing through an edge on the shortest path connecting $w$ and $v$ on $S$ has to intersect $P$ as well. It follows that replacing all edges of $P$ on the same side of $S$ as $v$ by the path from $w$ to $v$ does not change its length. But then, $v'$ is combinatorially nearer to $v_0$ than $v$, contradiction.

Finally, we can cyclically order the directed edges starting in $v_0$ according to their slopes. Then, the sectors $U_e$ are interlaced, i.e., $U_e$ contains both $U_{e_-}^+$ and $U_{e_+}^-$, where $e_-,e,e_+$ are consecutive according to the cyclic order. As a consequence, the union of all these $U_e$ covers the whole of $V(\Lambda)$.
\end{proof}

To perform our computations in Sections~\ref{sec:Green_asymptotics} and~\ref{sec:Cauchy_asymptotics}, we needed not only that the interior angles were bounded, but also that the ratio of side lengths of all parallelograms were bounded. For general quad-graphs, we used instead boundedness of side lengths to show in Theorem~\ref{th:harmonic_asymptotics} of Section~\ref{sec:general} that biconstant functions are the only discrete harmonic function whose difference functions on $V(\Gamma)$ and $V(\Gamma^*)$ have asymptotics $o(v^{-1/2})$. When the side lengths are bounded, the ratio of side lengths is trivially bounded as well. Under the assumption of bounded interior angles, the converse is true for parallelogram-graphs.

\begin{proposition}\label{prop:ratio_bounded}
Let $\Lambda$ be a parallelogram-graph and assume that there are $\alpha_0,q_0>0$ such that $\alpha\geq\alpha_0$ and $e/e'\geq q_0$ for all interior angles $\alpha$ and two side lengths $e,e'$ of any parallelogram of $\Lambda$. Then, there exist positive numbers $E_1>E_0$ such that $E_1\geq e \geq E_0$ for all edge lengths $e$.
\end{proposition}
\begin{proof}
Let $Q' \in V(\Diamond)$ be fixed with edge lengths $e_1\geq e_0$ and let $Q \in V(\Diamond)$ be another parallelogram with center $x$. In the following, we construct a sequence of $n$ strips such that any two consecutive strips are crossing each other, the first one contains $Q'$, the last one contains $Q$, and $n \leq N:= \lceil \lfloor 2\pi/\alpha_0 \rfloor/2\rceil$. Then, it follows that the side lengths of $Q$ are bounded by $E_0:=e_0 q_0^N$ and $E_1:=e_1 q_0^{-N}$.

Let $S_0$ be a strip containing $Q'$. If $Q \in S_0$, we are done. Otherwise, we choose the common parallel $a_{S_0}$ in such a way that $x$ lies in the region $-a_{S_0}$ is pointing to. Since $S_0$ is monotone in the direction $ia_{S_0}$ and interior angles are uniformly bounded, the ray $x+ta_{S_0}$, $t>0$, intersects $S$ in exactly one line segment. Let $y_0$ be the first intersection point and $Q_0$ a quadrilateral of $S$ containing $y_0$.

Because $\Lambda$ is locally finite, the line segment connecting $x$ and $y_0$ intersects only finitely many parallelograms. Through any such parallelogram at most two strips are passing. Thus, only a finite number of strips intersect this line segment. Therefore, we can choose a strip $S_1$ intersecting $S_0^{Q_0}$ in a parallelogram $Q_{0,1}$ such that $S_1$ does not contain $Q$ and does not intersect the line segment connecting $x$ and $y_0$. Moreover, we require that $Q' \notin S_0^{Q_{0,1}}$. Now, choose the common parallel $a_{S_1}$ of $S_1$ in such a way that $\pi+\arg(a_{S_0})>\arg(a_{S_1})>\arg(a_{S_0})$. By construction, $x$ lies in the region $-a_{S_1}$ is pointing to. Note that $S_1^{Q_{0,1}}$ cannot cross $S_0$ a second time.

Suppose we have already constructed the strip $S_k$ with common parallel $a_{S_k}$, $k>0$, and $x$ lies in the region $-a_{S_k}$ is pointing to. $S_k$ shall not intersect the line segments connecting $x$ and $y_0$, or connecting $x$ and $y_{k-1}$. Moreover, assume that the semi-infinite part $S_k^{Q_{k-1,k}}$ starting in the intersection $Q_{k-1,k}$ of $S_k$ with $S_{k-1}$ does not cross $S_0$.

Let $y_k$ be the first intersection of the ray $x+ta_{S_k}$, $t>0$, with a quadrilateral $Q_k \in S_k$. By the same arguments as above, there exists a strip $S_{k+1}$ intersecting $S_k^{Q_{k-1,k}}\cap S_k^{Q_k}$ that does not contain $Q$ and does not intersect the line segments connecting $x$ and $y_0$ or $x$ and $y_{k}$. Choose its common parallel $a_{S_{k+1}}$ in such a way that $\pi+\arg(a_{S_k})>\arg(a_{S_{k+1}})>\arg(a_{S_k})$. By construction, $x$ lies in the region $-a_{S_{k+1}}$ is pointing to. If the semi-infinite part $S_{k+1}^{Q_{k,k+1}}$ starting in the intersection $Q_{k,k+1}$ with $S_{k}$ does not cross $S_0$, we continue this procedure.

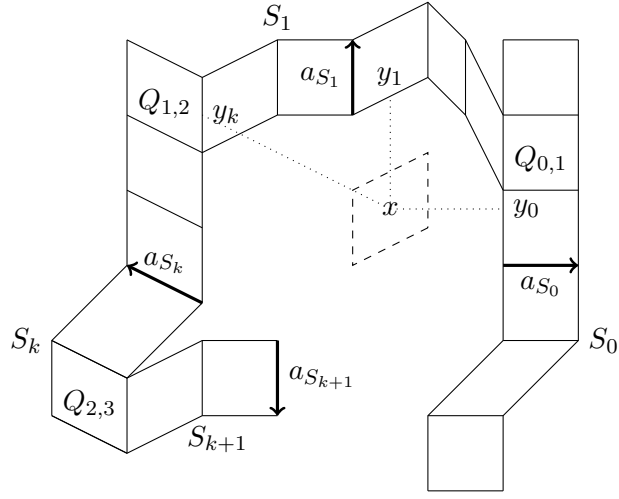
\begin{figure}[htbp]
\begin{center}
\beginpgfgraphicnamed{ratio_bounded}
\begin{tikzpicture}
\draw (1,-3) -- (1,-2) -- (2,-1) -- (2,3);
\draw (2,-3) -- (2,-2) -- (3,-1) -- (3,3);
\draw (1,-3) -- (2,-3);
\draw (1,-2) -- (2,-2);
\draw (2,-1) -- (3,-1);
\draw[very thick,->] (2,0) --node[midway,below] {$a_{S_0}$} (3,0);
\draw (2,1) -- (3,1);
\draw (2,2) -- (3,2);
\draw (2,3) -- (3,3);
\draw (1.5,2) -- (1.5,3);
\draw (1,2.5) -- (1,3.5);
\draw[very thick,->] (0,2) --node[midway,left] {$a_{S_1}$} (0,3);
\draw (-1,2) -- (-1,3);
\draw (-2,1.5) -- (-2,2.5);
\draw (-3,2) -- (-3,3);
\draw (2,1) -- (1.5,2) -- (1,2.5) -- (0,2) -- (-1,2) -- (-2,1.5) -- (-3,2);
\draw (2,2) -- (1.5,3) -- (1,3.5) -- (0,3) -- (-1,3)-- (-2,2.5)-- (-3,3);
\draw (-3,2) -- (-3,0) -- (-4,-1) -- (-4,-2);
\draw (-2,1.5) -- (-2,-0.5) -- (-3,-1.5) -- (-3,-2.5);
\draw (-3,1) -- (-2,0.5);
\draw[very thick,->] (-2,-0.5) --node[midway,above] {$a_{S_k}$} (-3,0);
\draw (-4,-1) -- (-3,-1.5);
\draw (-4,-2) -- (-3,-2.5);
\draw[very thick,->] (-1,-1) --node[midway,right] {$a_{S_{k+1}}$} (-1,-2);
\draw (-2,-2) -- (-2,-1);
\draw(-1,-2) -- (-2,-2) -- (-3,-2.5) -- (-4,-2);
\draw (-1,-1) -- (-2,-1)-- (-3,-1.5)-- (-4,-1);

\coordinate[label=right:$S_0$] (S0) at (3,-1);
\coordinate[label=above:$S_1$] (S1) at (-1,3);
\coordinate[label=left:$S_k$] (S2) at (-4,-1);
\coordinate[label=below:$S_{k+1}$] (S3) at (-1.8,-2);

\coordinate[label=center:$Q_{0,1}$] (z01) at (2.5,1.4);
\coordinate[label=center:$Q_{1,2}$] (z12) at (-2.5,2.15);
\coordinate[label=center:$Q_{2,3}$] (z23) at (-3.5,-1.85);

\draw[dashed] (1,0.5) -- (0,0) -- (0,1) -- (1,1.5) -- (1,0.5);
\coordinate[label=center:$x$] (x) at (0.5,0.75);
\coordinate[label=right:$y_0$] (y0) at (2,0.75);
\coordinate[label=above:$y_1$] (y1) at (0.5,2.25);
\coordinate[label=right:$y_k$] (y2) at (-2,2);
\draw[dotted] (x) -- (y0);
\draw[dotted] (x) -- (y1);
\draw[dotted] (x) -- (y2);
\end{tikzpicture}
\endpgfgraphicnamed
\caption{Schematic picture of the proof of Proposition~\ref{prop:ratio_bounded}}
\label{fig:ratio_bounded}
\end{center}
\end{figure}

After at most $l:=\lceil 2\pi/\alpha_0 \rceil$ steps, we end up with a strip $S_l$ such that $S_l^{Q_{l-1,l}}$ intersects $S_0$. Indeed, let us suppose the contrary, that is, let us suppose that all $S_2^{Q_{1,2}},\ldots,S_l^{Q_{l-1,l}}$ do not cross $S_0$.

By assumption, $\arg(a_{S_k})+\pi-\alpha_0\geq\arg(a_{S_{k+1}})\geq\arg(a_{S_k})+\alpha_0$. It follows that the first $j$ such that $\arg(a_{S_j})$ is greater or equal than $\arg(a_{S_0})+2\pi$ satisfies $j\leq\lceil 2\pi/\alpha_0 \rceil$. In addition, we have $\arg(a_{S_j})<\arg(a_{S_0})+3\pi-\alpha_0$.

By construction, $S_j$ does not intersect the line segment connecting $x$ and $y_{j-1}$. Moreover, $S_j^{Q_{j-1,j}}$ cannot cross $S_{j-1}$ a second time. It follows that $S_j^{Q_{j-1,j}}$ cannot intersect the ray $x+ta_{S_{j-1}}$, $t>0$.

Also, $S_j^{Q_{j-1,j}}$ does neither cross $S_0$ nor does it intersect the line segment connecting $x$ and $y_{0}$, so it does not intersect the ray $x+ta_{S_0}$, $t>0$. Thus, $S_j^{Q_{j-1,j}}$ is contained in the cone with tip $x$ spanned by $a_{S_{j-1}}$ and $a_{S_0}$ (with angle less than $\pi$). This contradicts the monotonicity of $S_j^{Q_{j-1,j}}$ into the direction $ia_{S_j}$, because the ray $x+ta_{S_j}$, $t>0$, is not contained in the interior of the cone above.

In summary, we found a cycle of $m$ strips $S_0,S_1,\ldots,S_{m-1}$ surrounding $x$, where $m\leq\lfloor 2\pi/\alpha_0 \rfloor+1$. Actually, $m\leq\lfloor 2\pi/\alpha_0 \rfloor$, because the $a_{S_k}$ are cyclically ordered. Since only finitely many strips intersect the strip $S_0$ in between $Q'$ and $Q_{0,1}$, we can assume that $Q'$ is contained in $S_0^{Q_{m-1,0}}$.

These $m$ strips determine a bounded region $x$ is contained in. If $Q'\neq Q_{m-1,0}$, we look at the semi-infinite part of the strip $\tilde{S}_0$ different from $S_0$ that passes through $Q'$ and goes into the interior of the bounded region above. It has to intersect one of the strips $S_1,\ldots,S_{m-1}$, say $S_k$. Then, $S_0,\ldots,S_k,\tilde{S}_0$ or $\tilde{S}_0,S_k,\ldots,S_{m-1},S_0$ determine a bounded region $x$ is contained in ($Q$ may be an element of $\tilde{S}_0$). Clearly, they are at most $\lfloor 2\pi/\alpha_0 \rfloor$ such strips, and $Q'$ lies on an intersection.

If $Q \notin \tilde{S}_0$, a strip $S_Q$ containing $Q$ has to cross two different strips of the cycle due to local finiteness. In the same way as above, we can find a cycle of at most $m'\leq\lfloor 2\pi/\alpha_0 \rfloor$ strips $S'_0,S'_1,\ldots,S'_{m'-1}$ such that $Q$ lies on one of the strips, say $S'_k$, and the intersection of $S'_0$ and $S'_{m'-1}$ is $Q'$. If $k \leq \lceil m'/2\rceil$, we choose the sequence of strips $S'_0,S'_1,\ldots,S'_k$; otherwise, we take $S'_{m'-1}, S'_{m'-2},\ldots,S'_k$. Any two consecutive strips are crossing each other, $Q'$ is on the first strip, $Q$ on the last one, and there are at most $\lceil \lfloor 2\pi/\alpha_0 \rfloor/2\rceil$ of them.
\end{proof}

\begin{remark}
In general, the bound $\lceil \lfloor 2\pi/\alpha_0 \rfloor/2\rceil$ in the proof is optimal. Indeed, consider $n$ rays emanating from 0 such that the angle between any two neighboring rays is $2\pi/n$. In each of the $n$ segments, choose the quad-graph combinatorially equivalent to the positive octant of the integer lattice that is spanned by two consecutive rays. For example, if $n=4$, we obtain $\mZ^2$. Then, any strip passes through exactly two adjacent segments, and $\lceil n/2 \rceil$ is the optimal bound.
\end{remark}
\end{appendix}

\section*{Acknowledgments}

The authors thank Richard Kenyon, Christian Mercat, and Mikhail Skopenkov for fruitful discussions and helpful suggestions.

\newpage
\bibliographystyle{plain}
\bibliography{Discrete_complex_analysis}

\end{document}